\documentclass[10pt, oneside]{article}
\newcommand{\dive}{\operatorname{div}}
\usepackage{mathrsfs}
\usepackage{amsfonts}
\usepackage{amssymb} 
\usepackage{amsmath,amsfonts,amssymb,amsthm}
\usepackage{caption}
\usepackage{subfigure}
\usepackage{cite}
\usepackage{color}
\usepackage{graphicx}
\usepackage{subfigure}
\usepackage{tikz}
\usetikzlibrary{arrows.meta}
\usepackage{float}
\usepackage{paralist}
\usepackage{indentfirst}
\usepackage{authblk}
\usepackage[colorlinks,linkcolor=black,citecolor=black]{hyperref}
\usepackage[T1]{fontenc}
\def\var{\varepsilon}

\newtheorem{theorem}{Theorem}[section]
\newtheorem{lemma}{Lemma}[section]
\newtheorem{prop}{Proposition}[section]
\newtheorem{cor}[theorem]{Corollary}
\newtheorem{remark}{Remark}[section]
\newtheorem{defn}{Definition}[section]

\definecolor{red}{rgb}{1.00,0.00,0.00}
\definecolor{blue}{rgb}{0.00,0.00,0.63}
\definecolor{black}{rgb}{0.00,0.00,0.00}
\definecolor{purple}{rgb}{0.00,1.00,0.00}
\definecolor{pink}{rgb}{0.95,0.01,0.08}

\def\bma#1\ema{{\allowdisplaybreaks\begin{aligned}#1\end{aligned}}}

\allowdisplaybreaks

\numberwithin{equation}{section}

\begin{document}
	
\title{\bf Large-friction and incompressible limits for pressureless Euler-Navier-Stokes flows}

\author{Hai-Liang Li, Ling-Yun Shou, and Yue Zhang}

\date{}
\maketitle


\begin{abstract}

We study the global macroscopic limits associated with kinetic-fluid interaction models for sprays. Motivated by the Vlasov-Navier-Stokes system under the monokinetic ansatz, we consider the pressureless Euler-Navier-Stokes (Euler-NS) system in $\mathbb{R}^{d}$ ($d\geq2$) coupled through the singular drag force $\frac{1}{\tau} \rho (u-v)$, where $\tau$ is the Stokes relaxation time. For initial data uniformly close to equilibrium in critical Besov spaces, we establish global-in-time regularity estimates of solutions to the Cauchy problem for the Euler-NS system, uniformly with respect to $\tau$. These estimates yield the global strong convergence of the Euler-NS system toward a one-velocity two-phase drift-flux (DF) model as $\tau\to0$, with an explicit convergence rate of order $\sqrt{\tau}$. A key point in the analysis is the introduction of an effective mixed velocity, which allows us to handle the singular relative-velocity relaxation and obtain global error estimates for ill-prepared data. We also derive large-time asymptotic estimates for the Euler-NS system, uniformly in $\tau$, including the improved decay of the relative velocity and the convergence of the non-dissipative density toward an asymptotic profile. Furthermore, after introducing the Mach number $\varepsilon>0$, we justify the incompressible limit of the DF model toward the Transport-Navier-Stokes (TNS) system as $\varepsilon\to0$, and prove the combined large-friction and incompressible limit from the Euler-NS system to the TNS system in the regime $\tau=\varepsilon\to0$ in an ill-prepared setting.  These results provide a unified and quantitative macroscopic picture connecting the Euler-NS, DF, and TNS systems through the large-friction and incompressible regimes.
\end{abstract}



\section{Introduction}

\subsection{Models}
Spray flows describe the motion of dispersed particles, such as droplets or dust, suspended in a surrounding gas. The interactions between particles and the fluid can be modeled by two-phase kinetic-fluid systems, which play an important role in various applications, including biotechnology, medical bioaerosols, chemical engineering, and combustion theory. We refer to \cite{williams85,orourke81,boudin03,baranger05,berres03,desvillettes10} and the references therein for more details on the physical background, modeling, and applications. A typical kinetic-fluid model is the Vlasov-Navier-Stokes (Vlasov-NS) system: 
\begin{equation}\label{VNS}
\left\{
\begin{aligned}
&\partial_{t}f+\xi\cdot\nabla_{x}f+\frac{1}{\tau}\operatorname{div}_{\xi}\bigl((v-\xi)f\bigr)=0,\\
&\partial_{t}n+\operatorname{div}(nv)=0,\\
&\partial_{t}(nv)+\operatorname{div}(nv\otimes v)+\nabla P(n)=\mu\Delta v+(\mu+\lambda)\nabla\operatorname{div}v+\frac{1}{\tau}\int_{\mathbb R^d}(\xi-v)f\,\mathrm{d}\xi.
\end{aligned}
\right.
\end{equation}
Here, $f=f(t,x,\xi)\geq 0$ denotes the number density of particles located at time $t\geq 0$, position $x\in\mathbb{R}^{d}$, and velocity $\xi\in\mathbb{R}^{d}$, while $n=n(t,x)\geq 0$ and $v=v(t,x)\in\mathbb{R}^{d}$ denote the density and velocity of the fluid, respectively. The pressure function $P(n)$ satisfies
\begin{equation*}
P(n)\in C^{\infty}(\mathbb{R}_{+}),\qquad P'(n)>0\quad\hbox{for}~~ n>0,
\end{equation*}
and the constants $\mu$ and $\lambda$ are the shear and bulk viscosity coefficients, satisfying
\begin{equation*}
\mu>0,\qquad 2\mu+\lambda>0.
\end{equation*}
The parameter $\tau>0$ is the so-called Stokes relaxation time. In system \eqref{VNS}, the particle phase is described by a Vlasov-type kinetic equation, while the carrier fluid is governed by the compressible Navier-Stokes equations, and the two phases are coupled through a drag force. On one hand, it can be formally derived as the mean-field limit of a $N$-particle system interacting with a fluid; refer to 
\cite{MR2398959,MR3984746,MR4050096,MR3795188} and the references therein. On the other hand, it gives rise to various macroscopic fluid systems through hydrodynamic limits in different regimes, cf. \cite{hankwanmichel24,choijungJMFM21, figallikang19, choijungM3AS21}.

In this paper, we consider the system \eqref{VNS} in the \emph{monokinetic} regime
\begin{equation}\label{monokinetic}
f(t,x,\xi)=\rho(t,x)\delta_{\xi=u(t,x)},
\end{equation}
where $\rho(t,x)$ and $u(t,x)$ are the local density and velocity of particles, satisfying
\begin{equation*}
\rho(t,x)=\int_{\mathbb{R}^d} f(t,x,\xi)\,\mathrm{d}\xi,\qquad u(t,x)=\frac{\int_{\mathbb{R}^d} \xi f(t,x,\xi)\,\mathrm{d}\xi}{\int_{\mathbb{R}^d} f(t,x,\xi)\,\mathrm{d}\xi}.
\end{equation*}
Under the monokinetic ansatz \eqref{monokinetic}, the kinetic-fluid model \eqref{VNS} formally reduces to the following two-phase flow model coupling the pressureless Euler equations and the isentropic compressible Navier-Stokes equations (Euler-NS):
\begin{equation}\label{pENS}
\left\{
\begin{aligned}
& \partial_{t}\rho+\operatorname{div}(\rho u)=0,\\
& \partial_{t}(\rho u)+\operatorname{div}(\rho u\otimes u)=-\frac{1}{\tau}\rho(u-v),\\
& \partial_{t}n+\operatorname{div}(nv)=0,\\
& \partial_{t}(nv)+\operatorname{div}(nv\otimes v)+\nabla P(n)=\mu\Delta v+(\mu+\lambda)\nabla\operatorname{div}v+\frac{1}{\tau}\rho(u-v).
\end{aligned}
\right.
\end{equation}
We consider the Cauchy problem subject to the initial data
\begin{equation}\label{indt}
(\rho,u,n,v)(0,x)=(\rho_{0},u_{0},n_{0},v_{0})(x)\rightarrow (0,0,1,0),\qquad \text{as}~~ |x|\rightarrow\infty .
\end{equation}

At the formal level, in the large-friction regime, let $(\rho^{\tau}, u^{\tau}, n^{\tau}, v^{\tau})$ be the global solution of the system \eqref{pENS}, and assume that it converges to $(\rho, u, n, v)$ as $\tau\rightarrow 0$. Due to the relaxation effect of the friction drag force acting on the relative velocity $u^\tau-v^\tau$, we deduce that $u=v$. Passing to the limit as $\tau\rightarrow0$ and summing $\eqref{pENS}_2$ and $\eqref{pENS}_4$ together, we derive the following one-velocity two-phase drift-flux (DF) model:
\begin{equation}\label{DF}
\left\{
\begin{aligned}
&\partial_{t}\rho+\operatorname{div} (\rho v)=0,\\
&\partial_{t}n+\operatorname{div} (nv)=0,\\
&\partial_{t}((\rho+n) v)+\operatorname{div} ( (\rho+n) v\otimes v)+\nabla P(n)=\mu\Delta v+(\mu+\lambda)\nabla\operatorname{div} v,
\end{aligned}
\right.
\end{equation}
subject to the initial data
\begin{equation}\label{DFindt}
(\rho,n,v)(0,x)=(\rho_{0},n_{0},v_{0})(x)\rightarrow (0,1,0), \quad\text{as}~~ |x|\rightarrow\infty.
\end{equation}
Similarly, one can formally derive \eqref{DF} from the Vlasov-NS system \eqref{VNS} as $\tau\rightarrow 0$.
Note that a similar two-phase flow model to \eqref{DF}, with an additional pressure term $\nabla \rho$ in $\eqref{DF}_{3}$, has been derived by Mellet and Vasseur \cite{melletvasseur08} from a kinetic-fluid system corresponding to strong drag force and Brownian motion.

We are also interested in studying the dynamics of the two-phase flow models \eqref{pENS} and \eqref{DF} in the low-Mach-number regime. Let $\var>0$ be the Mach number, which is a non-dimensional quantity. As in \cite{schochet07,lionsmasmoudi98}, we first consider the transformations
\begin{equation*}
\rho(t,x)=\var \rho^\var(\var t,x),~~ n(t,x)= n^\var(\var t,x),~~ v(t,x)=\var v^\var(\var t,x),~~ \mu=\var \bar{\mu},~~ \lambda=\var \bar{\lambda},
\end{equation*}
where $\bar{\mu}$ and $\bar{\lambda}$ are two constants independent of $\var$, and we further set $n^\var=1+\var a^\var$.
Then, the DF model \eqref{DF} can be rewritten as
\begin{equation}\label{DFvar}
\left\{
\begin{aligned}
&\partial_{t}\rho^\var+\operatorname{div} (\rho^\var v^\var)=0,\\
&\var\partial_{t}a^\var+\operatorname{div}  v^\var+\var\operatorname{div} (a^\var v^\var)=0,\\
&\partial_{t}\big((1+\var\rho^\var+\var a^\var)v^\var\big)+\operatorname{div}\big((1+\var \rho^\var+\var a^\var)v^\var\otimes v^\var\big)+\frac{1}{\var^2}\nabla P(1+\var a^\var)\\
&\qquad\qquad=\bar{\mu} \Delta v^\var+(\bar{\mu}+\bar{\lambda})\nabla \operatorname{div} v^\var.
\end{aligned}
\right.
\end{equation}
Assume that $(\rho^\var,a^\var)$ have uniform bounds with respect to $\var$ in a suitable sense. Consequently, $(\rho^\var,v^\var)$ converges to a limit $(\varrho,w)$, which satisfies a transport equation driven by a velocity field governed by the incompressible Navier-Stokes equations, forming the following Transport-Navier-Stokes (TNS) system:
\begin{equation}\label{TNS}
\left\{
\begin{aligned}
&\partial_{t}\varrho+\operatorname{div} (\varrho w)=0,\\
&\partial_{t}w+ \operatorname{div}(w\otimes w)+\nabla \Pi=\bar{\mu}\Delta w,\\
&\operatorname{div}w=0,
\end{aligned}
\right.
\end{equation}
with the initial data
\begin{equation}\label{TNSindt}
(\varrho,w)(0,x)=(\varrho_0,w_0)(x)\rightarrow (0,0),\quad \text{as}~~|x|\rightarrow \infty. 
\end{equation}

Furthermore, it is natural to investigate the Euler-NS system \eqref{pENS} in both large-friction and low-Mach-number regimes. Similarly, by introducing the transformations 
\begin{equation*}
\rho(t,x)=\var \rho^\var(\var t,x),~~u(t,x)=\var u^\var(\var t,x),~~ n(t,x)= n^\var(\var t,x),~~ v(t,x)=\var v^\var(\var t,x),~~ \mu=\var \bar{\mu},~~ \lambda=\var \bar{\lambda},
\end{equation*}
and setting $n^\var=1+\var a^\var$, the Euler-NS system \eqref{pENS} can be reformulated into
\begin{equation}\label{pENSvar}
\left\{
\begin{aligned}
&\partial_{t}\rho^\var+\operatorname{div} (\rho^\var u^\var)=0,\\
&\var \partial_{t}  (\rho^\var u^\var)+\var\operatorname{div} ( \rho^\var u^\var \otimes u^\var)=- \frac{1}{\tau}\rho^\var (u^\var-v^\var),\\
&\var\partial_{t}a^\var+\operatorname{div} v^\var +\var \operatorname{div} (a^\var v^\var)=0,\\
&\partial_{t}\big( (1+\var a^\var) v^\var\big)+ \operatorname{div} \big( (1+\var a^\var)v^\var\otimes v^\var\big)+\frac{1}{\var^2}\nabla P(1+\var a^\var)\\
&\qquad\qquad=\bar{\mu} \Delta v^\var+(\bar{\mu}+\bar{\lambda})\nabla \operatorname{div} v^\var+ \frac{1}{\tau}\rho^\var(u^\var-v^\var).
\end{aligned}
\right.
\end{equation}
We focus on the case where the Stokes relaxation time is proportional to the Mach number, e.g., $\tau=\var$. Assume that $(\rho^\var, u^\var, v^\var)$ converges to a limit $(\varrho,w,w)$ as $\var\rightarrow 0$. Adding $\eqref{pENSvar}_2$ and $\eqref{pENSvar}_4$ together and passing to the limit as $\var\rightarrow0$, we derive that $(\varrho,w)$ is a solution to the TNS system \eqref{TNS}.

The mathematical study of the Vlasov-NS system goes back to the work \cite{hamdache1,Anoshchenko}. The local well-posedness of classical solutions was investigated by Baranger and Desvillettes \cite{MR2214160}. Global weak solutions were constructed by Boudin et al.~\cite{boudin3} on the torus, and later extended to bounded domains by Yu \cite{Yu1}. The large-time behavior of such weak solutions has been the subject of several recent works. It is now known that the dynamics may exhibit \emph{monokinetic} concentration in velocity, together with velocity alignment. Bae et al. \cite{MR3223813} and Choi \cite{MR3518989} obtained exponential stability results in spatially periodic domains under additional uniform bounds on the solutions. Han-Kwan, Mossa and Moyano \cite{han1} justified these bounds and proved an unconditional exponential stability result for the incompressible Vlasov-NS system with small initial data. In the one-dimensional compressible setting, Li and Shou \cite{MR4213662,LiShou2021SciSinMath} established the global well-posedness and unconditional exponential convergence to equilibrium for large data. 
In \cite{LSZ1}, the authors introduced a new class of Lyapunov functionals capturing the dissipation rate of kinetic-fluid interactions, and applied them to study the exponential stability of the inhomogeneous incompressible Vlasov-NS system in the whole space. Using a related method, Danchin \cite{MR5041103} obtained the global stability and optimal algebraic decay rates for Fujita-Kato solutions in the homogeneous case. 

Regarding the hydrodynamic limits of the Vlasov-NS system, Han-Kwan and Michel \cite{hankwanmichel24} introduced a general framework in high-friction regimes, where the dispersed particles are light, or both light and fast, compared with the fluid phase. In particular, they showed that the TNS system \eqref{TNS} arises as the hydrodynamic limit of the Vlasov-NS system in the light or light and fast particle regimes.


Let us provide an overview of existing research on related systems of the Euler-NS system \eqref{pENS}, i.e., the Vlasov-NS system in the monokinetic regime. In the absence of interactions between the fluid, the system described by $\eqref{pENS}_1$-$\eqref{pENS}_2$ reduces to the pressureless Euler equations, which are widely used in plasma physics. In this case, it is well-known that classical solutions develop singularities (e.g., shock waves) in finite time, even for small and smooth initial data (refer to \cite{dafermos16}). Therefore, it is natural to investigate the dynamics of global weak solutions, cf. \cite{bouchut94, breniergrenier98, huangwang01, weinan96} and the references therein.

When the coupling with a viscous fluid through the drag force is considered, the interactions may prevent the blow-up phenomenon and ensure global well-posedness.
In the case where the fluid is incompressible (e.g., $n\equiv 1$ in \eqref{pENS}), Choi and Jung \cite{choijungJMFM21} proved the global-in-time existence and uniqueness of classical solutions for small initial data in $H^s(\mathbb{R}^d)\cap L^1(\mathbb{R}^d)$ with $d\geq 3$ and sufficiently large $s$. They further obtained the optimal time-decay rates of solutions in \cite{choijungkim24}. Danchin \cite{danchin2026elementary} employed an elementary method to prove the global existence of strong solutions in the whole space, without assuming the initial density to be small and regular. As for the compressible Euler-NS system \eqref{pENS}, Choi and Kwon \cite{choikwonJDE16} established a unique global strong solution and obtained exponential convergence rates over time in a spatial periodic domain. The global existence and algebraic time-decay rates of solutions close to the equilibrium state $(\bar{\rho},0,1,0)$ with constant $\bar{\rho}\geq0$ have been investigated by Guo et al. \cite{guoetal24} without vacuum ($\bar{\rho} > 0$) and by Li et al. \cite{litangzhang25} with far-field vacuum ($\bar{\rho} = 0$). Additionally, one can refer to \cite{dongetal19, liwangzhang25} for the blow-up or non-existence phenomena of classical solutions.

To the best of our knowledge, there are very few results on the uniform-in-time derivation of such macroscopic equations from kinetic-fluid interaction models. In particular, uniform-in-time derivations of models \eqref{pENS}, \eqref{DF}, and \eqref{TNS} in the present setting seem to be open. Moreover, although these limiting systems arise from related asymptotic regimes, the links between these macroscopic models do not seem to have been investigated in the literature.

\subsection{Our goals}

The first goal of this paper is to provide a novel perspective on the {\emph{global, purely macroscopic relationships}} between the three fluid-dynamical systems \eqref{pENS}, \eqref{DF}, and \eqref{TNS}. More precisely, we perform a quantitative analysis of the Cauchy problem \eqref{pENS}-\eqref{indt}, with particular emphasis on how the friction parameter $\tau$ and the Mach number $\varepsilon$ affect the global regularity and uniform estimates of solutions. This allows us to rigorously justify, for {\emph{ill-prepared}} initial data, the global-in-time strong convergence of the Euler-NS system \eqref{pENS} toward the DF model \eqref{DF} as $\tau\to0$, and toward the TNS model \eqref{TNS} in the joint limit $\tau=\varepsilon\to0$. The limiting procedures investigated in this paper are schematically summarized in Figure~\ref{fig:tri-limits} below.

\begin{figure}[!ht]
\centering
\begin{tikzpicture}[
system/.style={draw, rectangle, minimum width=3cm, minimum height=1cm, inner sep=8pt, align=center, font=\small},
arrow/.style={-{Stealth[length=3mm,width=2mm]}, thick}
]

\node[system] (EulerNS) at (0,4) {Euler-NS system};
\node[system] (DF)      at (0,0) {DF system};
\node[system] (TNS)     at (5.2,2) {TNS system};
\node[system] (VNS)     at (-5.2,2) {Vlasov-NS system};

\draw[arrow] (EulerNS.south) -- (DF.north)
node[midway, right=10pt, font=\small, fill=white, inner sep=1pt] {\(\tau\to0\)};

\draw[arrow] (EulerNS.south east) -- (TNS.north west)
node[midway, above=10pt, xshift=6pt, font=\small, fill=white, inner sep=1pt]
{\(\tau=\varepsilon\to0\)};

\draw[arrow] (DF.north east) -- (TNS.south west)
node[midway, below=10pt, xshift=6pt, font=\small, fill=white, inner sep=1pt]
{\(\varepsilon\to0\)};

\draw[arrow] (VNS.north east) -- (EulerNS.south west)
node[midway, above=10pt, xshift=-6pt, font=\small, fill=white, inner sep=1pt]
{\(f=\rho\delta_{\xi=u}\)};

\draw[arrow] (VNS.south east) -- (DF.north west)
node[midway, below=10pt, xshift=-6pt, font=\small, fill=white, inner sep=1pt]
{\(\tau\to0\)};

\end{tikzpicture}
\caption{Large-friction and incompressible limits}
\label{fig:tri-limits}
\end{figure}
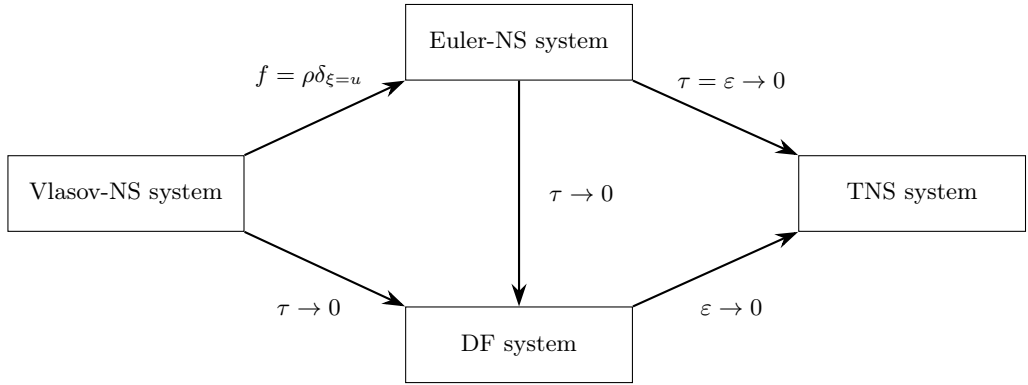

In addition to the singular limit analysis, we aim to study the global dynamics for the Euler-NS system \eqref{pENS} with minimal regularity assumptions. A fundamental challenge in studying the global well-posedness for \eqref{pENS} arises from ensuring the $L^1$ time integrability of $\nabla u$ and $\nabla v$, namely,
\begin{align}
\int_0^\infty \|(\nabla u,\nabla v)(t)\|_{L^{\infty}}\,{\rm d}t<\infty,\label{L1Lip}
\end{align}
which is crucial for controlling the transport terms in system \eqref{pENS} and preventing the finite-time blow-up phenomenon. We mention that, to achieve the bound \eqref{L1Lip}, the authors in \cite{litangzhang25, guoetal24} combined spectral analysis with the weighted energy method to obtain sufficient time-decay properties of the velocities, under an additional $L^1$-assumption on the initial perturbation $(u_0, n_0 - 1, v_0)$. Indeed, it is a possible way to overcome the difficulty in \eqref{L1Lip} by using the framework of homogeneous Besov spaces of $\dot{B}^s_{2,1}$-type, where the factor `1' ensures the maximal $L^1$ time integrability. According to the spectral structure of system \eqref{pENS} (see \hyperref[sectionlinear]{Appendix A}), the key estimate \eqref{L1Lip} determines the minimal regularity requirements for the initial data, that is, $(\rho_0, u_0)\in (\dot{B}^{\frac{d}{2}-1}_{2,1} \cap \dot{B}^{\frac{d}{2}}_{2,1}) \times (\dot{B}^{\frac{d}{2}-1}_{2,1} \cap \dot{B}^{\frac{d}{2}+1}_{2,1})$ and $(n_0 - 1, v_0)\in (\dot{B}^{\frac{d}{2}-1}_{2,1} \cap \dot{B}^{\frac{d}{2}}_{2,1}) \times \dot{B}^{\frac{d}{2}-1}_{2,1}$. These spaces, known as the {\emph{critical}} functional spaces (refer to \cite{danchinIM00, bahourietal11}), are carefully chosen to preserve the scaling invariance of the viscous flow while ensuring the essential Lipschitz bound for the velocity of the pressureless Euler flow.
Moreover, if the initial perturbation $(\rho_{0},u_{0},n_{0}-1,v_{0})$ is further uniformly bounded in $\dot{B}^{\sigma_1}_{2,\infty}$ for $\sigma_{1}\in[-\frac{d}{2},\frac{d}{2}-1)$, we provide quantitative and qualitative descriptions of the large-time behavior of global solutions to the Cauchy problem \eqref{pENS}-\eqref{indt} for the Euler-NS system, with explicit dependence on $\tau\in(0,1)$.
Furthermore, due to the transport nature of $\eqref{pENS}_1$, we establish the asymptotic convergence behavior of the non-dissipative component $\rho$ to the non-constant equilibrium state $\rho_\infty = \rho_\infty(x)$.
In particular, we deduce improved decay estimates of the relative velocity $u-v$ and obtain the optimal time-decay rates of the velocities $u$ and $v$ for all spatial regularity indices up to $\frac{d}{2}+1$.

\subsection{Outline}

The rest of the paper is organized as follows. In Section \ref{sectionresults}, we state the main theorems and provide strategies of the proofs.
In Section \ref{sectionexistence}, we derive the uniform \emph{a priori} estimates and prove the global existence result stated in Theorem \ref{theorem1}.
In Section \ref{sectionlimit}, we establish the error estimates between global solutions of the Euler-NS system \eqref{pENS} and the limiting DF model \eqref{DF}, which leads to the proof of Theorem \ref{theorem3}.
In Section \ref{sectiondecay}, we establish the optimal time-decay estimates of solutions and prove Theorem \ref{theorem2}.
In Section \ref{section:incompressible}, we rigorously justify the incompressible limits of systems \eqref{DFvar} and \eqref{pENSvar} established in Theorems \ref{Theoremincompressible1} and \ref{Theoremincompressible2}.
In \hyperref[sectionlinear]{Appendix A}, we consider the linear problem and present a detailed spectral analysis with respect to the parameter $\tau$, which indicates the optimal dissipation structures and time-decay estimates. The definitions of functional spaces and essential technical lemmas are provided in \hyperref[sectionpre]{Appendix B}.

\section{Main results}\label{sectionresults}

\subsection{Large-friction limit}

First, we establish the global existence of strong solutions to the Cauchy problem \eqref{pENS}-\eqref{indt} for $\tau>0$ near the equilibrium state $(0,0,1,0)$ in critical Besov spaces, without requiring the additional $L^1$-assumption. Moreover, we derive the uniform-in-$\tau$ regularity estimates of the solution.
\begin{theorem}\label{theorem1}
Let $\tau\in(0,1)$ be a constant. For any $d\geq2$, there exists a positive constant $\delta_{0}>0$ independent of the parameter $\tau$ such that if the initial data $(\rho^\tau_{0},u^\tau_{0},n^\tau_{0},v^\tau_{0})$ satisfies $\rho^\tau_{0}(x)>0$ for $x\in\mathbb{R}^d$ and
\begin{equation}\label{idcd1}
\begin{aligned}
\|(\rho^\tau_{0}, n^\tau_{0}-1)\|_{\dot{B}^{\frac{d}{2}-1}_{2,1}\cap\dot{B}^{\frac{d}{2}}_{2,1}}+\|(u^\tau_{0},v^\tau_{0})\|_{\dot{B}^{\frac{d}{2}-1}_{2,1}}+\tau\|u^\tau_{0}\|_{\dot{B}^{\frac{d}{2}+1}_{2,1}}\leq \delta_{0},
\end{aligned}
\end{equation}
then the Cauchy problem \eqref{pENS}-\eqref{indt} admits a unique global strong solution $(\rho^\tau,u^\tau,n^\tau,v^\tau)$ satisfying
\begin{equation}\label{soluspce1}
\left\{
\begin{aligned}
&\rho^\tau(t,x)>0~~\text{for}~~(t,x)\in\mathbb{R}_+\times\mathbb{R}^d,\quad \rho^\tau\in \mathcal{C}(\mathbb{R}_{+};\dot{B}^{\frac{d}{2}-1}_{2,1}\cap \dot{B}^{\frac{d}{2}}_{2,1}),\\
&(u^\tau,v^\tau)\in \mathcal{C}(\mathbb{R}_{+};\dot{B}^{\frac{d}{2}-1}_{2,1})\cap L^{1}(\mathbb{R}_{+};\dot{B}^{\frac{d}{2}+1}_{2,1}),\quad \tau u^\tau\in \mathcal{C}(\mathbb{R}_{+};\dot{B}^{\frac{d}{2}+1}_{2,1}),\\
&(n^\tau-1)\in \mathcal{C}(\mathbb{R}_{+};\dot{B}^{\frac{d}{2}-1}_{2,1}\cap\dot{B}^{\frac{d}{2}}_{2,1})\cap  L^{1}(\mathbb{R}_{+};\dot{B}^{\frac{d}{2}+1}_{2,1}+\dot{B}^{\frac{d}{2}}_{2,1}),\\
&\frac{1}{\sqrt{\tau}}(u^\tau-v^\tau)\in  L^{1}(\mathbb{R}_{+};\dot{B}^{\frac{d}{2}}_{2,1})\cap L^{2}(\mathbb{R}_{+};\dot{B}^{\frac{d}{2}-1}_{2,1}),\\
&\frac{1}{\tau}(u^\tau-v^\tau)\in L^{1}(\mathbb{R}_{+};\dot{B}^{\frac{d}{2}}_{2,1}+\dot{B}^{\frac{d}{2}-1}_{2,1}).
\end{aligned}
\right.
\end{equation}
Furthermore, there exists a positive constant $C_{1}>0$ independent of $\tau$ such that
\begin{align}
&\label{es:rho}
\|\rho^\tau\|_{\widetilde{L}^{\infty}(\mathbb{R}_+;\dot{B}^{\frac{d}{2}-1}_{2,1})}\leq C_{1}\|\rho^\tau_0\|_{\dot{B}^{\frac{d}{2}-1}_{2,1}},\quad \|\rho^\tau\|_{\widetilde{L}^{\infty}(\mathbb{R}_+;\dot{B}^{\frac{d}{2}}_{2,1})}\leq C_{1}\|\rho^\tau_0\|_{\dot{B}^{\frac{d}{2}}_{2,1}},
\end{align}
and
\begin{align}\label{result1}
\begin{aligned}
&\|(u^\tau,v^\tau)\|_{\widetilde{L}^{\infty}(\mathbb{R}_+;\dot{B}^{\frac{d}{2}-1}_{2,1})}+\tau\|u^\tau\|_{\widetilde{L}^{\infty}(\mathbb{R}_+;\dot{B}^{\frac{d}{2}+1}_{2,1})}+\|n^\tau-1\|_{\widetilde{L}^{\infty}(\mathbb{R}_+;\dot{B}^{\frac{d}{2}-1}_{2,1}\cap\dot{B}^{\frac{d}{2}}_{2,1})}\\
&\quad+\|(u^\tau,v^\tau)\|_{L^{1}(\mathbb{R}_+;\dot{B}^{\frac{d}{2}+1}_{2,1})}+\|n^\tau-1\|_{L^{1}(\mathbb{R}_+;\dot{B}^{\frac{d}{2}+1}_{2,1}+\dot{B}^{\frac{d}{2}}_{2,1})}\\
&\quad \leq C_1\Big(\|(u^\tau_{0},v^\tau_{0})\|_{\dot{B}^{\frac{d}{2}-1}_{2,1}}+\tau\|u^\tau_{0}\|_{\dot{B}^{\frac{d}{2}+1}_{2,1}}+\|(n^\tau_{0}-1)\|_{\dot{B}^{\frac{d}{2}-1}_{2,1}\cap\dot{B}^{\frac{d}{2}}_{2,1}}\Big).
\end{aligned}
\end{align}
In addition, the relative velocity $u^\tau-v^\tau$ has the following estimates
\begin{equation}\label{uvcon}
\left\{
\begin{aligned}
&\|u^\tau-v^\tau\|_{L^{1}(\mathbb{R}_+;\dot{B}^{\frac{d}{2}}_{2,1})}+\|u^\tau-v^\tau\|_{\widetilde{L}^{2}(\mathbb{R}_+;\dot{B}^{\frac{d}{2}-1}_{2,1})}\leq C_{1} \sqrt{\tau},\\
&\|u^\tau-v^\tau\|_{L^{1}(\mathbb{R}_+;\dot{B}^{\frac{d}{2}}_{2,1}+\dot{B}^{\frac{d}{2}-1}_{2,1})}\leq C_{1} \tau.
\end{aligned}
\right.
\end{equation}
\end{theorem}

\begin{remark}\normalfont
Theorem \ref{theorem1} provides a global stability result for the Euler-NS system \eqref{pENS} with critical regularity, uniformly with respect to the Stokes relaxation time $\tau$. The uniform $L^1$ time regularity \eqref{L1Lip} is established, crucially avoiding any use of decay estimates under additional lower-order assumptions.
\end{remark}

\begin{remark}\normalfont
It should be noted that the far-field vacuum condition $\rho^\tau(t,x)\rightarrow 0$ as $|x|\rightarrow\infty$ ensures the purely transport structure of $\eqref{pENS}_1$, and one needs $\rho^\tau(t,x)>0$ for every $(t,x)\in \mathbb{R}_+\times \mathbb{R}^d$ to derive a Burgers-type equation of $u^\tau$ in  $\eqref{pENS}_2$.
\end{remark}

\begin{remark}\normalfont
One of the main challenges arises from the highly singular nonlinear term $\frac{1}{\tau} \rho^\tau (u^\tau - v^\tau)$ in the large-friction regime. To address this, we establish the improved regularity estimates \eqref{uvcon} for the relative velocity $u^\tau - v^\tau$ with explicit convergence rates depending on $\tau$. In particular, the regularity $L^1(\mathbb{R}_{+}; \dot{B}^{\frac{d}{2}}_{2,1}) \cap L^2(\mathbb{R}_{+}; \dot{B}^{\frac{d}{2}-1}_{2,1})$ for $u^\tau - v^\tau$ is stronger than the regularity $L^1(\mathbb{R}_{+}; \dot{B}^{\frac{d}{2}+1}_{2,1}) \cap L^2(\mathbb{R}_{+}; \dot{B}^{\frac{d}{2}}_{2,1})$ for $(u^\tau, v^\tau)$ in low frequencies, where $(u^\tau, v^\tau)$ does not exhibit any convergence rate. Note that the estimates in \eqref{uvcon}, revealing the alignment phenomena between the velocities of two fluids, are crucial to derive the convergence rate in the large-friction limit.
\end{remark}

Due to the consistency of the estimates in Theorem \ref{theorem1} with respect to the parameter $\tau$, we have the following global existence result of strong solutions to the Cauchy problem \eqref{DF}-\eqref{DFindt} for the DF model as the large-friction limit of the Euler-NS system.

\begin{cor}\label{corollary1}
For any $d\geq 2$, there exists a positive constant $\delta_{0}^{*}>0$ such that if the initial data $(\rho_{0},n_{0},v_{0})$ satisfies $\rho_{0}(x)\geq 0$ for $x\in\mathbb{R}^d$ and
\begin{equation}\label{idcd4}
\begin{aligned}
\|(\rho_{0}, n_{0}-1)\|_{\dot{B}^{\frac{d}{2}-1}_{2,1}\cap\dot{B}^{\frac{d}{2}}_{2,1}}+\|v_{0}\|_{\dot{B}^{\frac{d}{2}-1}_{2,1}}\leq \delta_{0}^{*},
\end{aligned}
\end{equation}
then the Cauchy problem \eqref{DF}-\eqref{DFindt} admits a unique global strong solution $(\rho,n,v)$ satisfying
\begin{equation}\label{soluspce2}
\left\{
\begin{aligned}
&\rho(t,x)\geq 0 ~~\text{for}~~ (t,x)\in\mathbb{R}_+\times\mathbb{R}^d,\quad  \rho\in \mathcal{C}(\mathbb{R}_{+};\dot{B}^{\frac{d}{2}-1}_{2,1}\cap \dot{B}^{\frac{d}{2}}_{2,1}),\\
&n-1\in \mathcal{C}(\mathbb{R}_{+};\dot{B}^{\frac{d}{2}-1}_{2,1}\cap\dot{B}^{\frac{d}{2}}_{2,1})\cap  L^{1}(\mathbb{R}_{+};\dot{B}^{\frac{d}{2}+1}_{2,1}+\dot{B}^{\frac{d}{2}}_{2,1}),\\
&v\in \mathcal{C}(\mathbb{R}_{+};\dot{B}^{\frac{d}{2}-1}_{2,1})\cap L^{1}(\mathbb{R}_{+};\dot{B}^{\frac{d}{2}+1}_{2,1}).
\end{aligned}
\right.
\end{equation}
Moreover, there exists a positive constant $C_{1}^{*}>0$ such that
\begin{align}\label{rhoDF}
\|\rho\|_{\widetilde{L}^{\infty}(\mathbb{R}_+;\dot{B}^{\frac{d}{2}-1}_{2,1})}\leq C_{1}^*\|\rho_0\|_{\dot{B}^{\frac{d}{2}-1}_{2,1}},\quad \|\rho\|_{\widetilde{L}^{\infty}(\mathbb{R}_+;\dot{B}^{\frac{d}{2}}_{2,1})}\leq  C_{1}^*\|\rho_0\|_{\dot{B}^{\frac{d}{2}}_{2,1}},
\end{align}
and
\begin{equation}\label{result5}
\begin{aligned}
&\|n-1\|_{\widetilde{L}^{\infty}(\mathbb{R}_+;\dot{B}^{\frac{d}{2}-1}_{2,1}\cap\dot{B}^{\frac{d}{2}}_{2,1})}+\|v\|_{\widetilde{L}^{\infty}(\mathbb{R}_+;\dot{B}^{\frac{d}{2}-1}_{2,1})}\\
&\quad\quad+\|n-1\|_{L^{1}(\mathbb{R}_+;\dot{B}^{\frac{d}{2}+1}_{2,1}+\dot{B}^{\frac{d}{2}}_{2,1})}+\|v\|_{L^{1}(\mathbb{R}_+;\dot{B}^{\frac{d}{2}+1}_{2,1})}\\
&\quad \leq C_{1}^{*}\Big(\|(n_{0}-1)\|_{\dot{B}^{\frac{d}{2}-1}_{2,1}\cap\dot{B}^{\frac{d}{2}}_{2,1}}+\|v_{0}\|_{\dot{B}^{\frac{d}{2}-1}_{2,1}}\Big).
\end{aligned}
\end{equation}
\end{cor}

\begin{remark}\normalfont
Unlike the condition in Theorem \ref{theorem1} that the density $\rho^\tau$ has only far-field vacuum, Corollary \ref{corollary1} allows a non-negative density $\rho$ for $x\in\mathbb{R}^d$ since the positivity of the sum $\rho+n$ is enough to guarantee the parabolicity of equation $\eqref{DF}_{3}$.
\end{remark}

To prove the global strong convergence of the Euler-NS system \eqref{pENS} towards the DF model \eqref{DF} as $\tau \rightarrow 0$, we establish the following uniform error estimates between the global solutions of the system \eqref{pENS} and its limiting system \eqref{DF}.

\begin{theorem}\label{theorem3}
Let $\tau\in(0,1)$ be a constant, $(\rho^{\tau},u^{\tau},n^{\tau},v^{\tau})$ be the global solution to the Euler-NS system \eqref{pENS} associated with the initial data $(\rho^{\tau}_{0},u^{\tau}_{0},n^{\tau}_{0},v^{\tau}_{0})$ obtained in Theorem \ref{theorem1}, and $(\rho ,n ,v )$ be the global solution to the limiting DF model \eqref{DF} with the initial data $(\rho _{0},n _{0},v _{0})$ given in Corollary \ref{corollary1}.
For any $d\geq 3$, there exists a positive constant $C_{2}>0$ independent of $\tau$ such that
\begin{equation}\label{result6}
\begin{aligned}
&\|(\rho^{\tau}-\rho ,n^{\tau}-n )\|_{\widetilde{L}^{\infty}(\mathbb{R}_+;\dot{B}^{\frac{d}{2}-2}_{2,1}\cap\dot{B}^{\frac{d}{2}-1}_{2,1})}
+\left\|\frac{\rho^{\tau}}{\rho^{\tau}+n^{\tau}}u^{\tau}+\frac{n^{\tau}}{\rho^{\tau}+n^{\tau}}v^{\tau}-v\right\|_{\widetilde{L}^{\infty}(\mathbb{R}_+;\dot{B}^{\frac{d}{2}-2}_{2,1})}\\
&\quad+\|(u^{\tau}-v ,v^{\tau}-v )\|_{L^{1}(\mathbb{R}_+;\dot{B}^{\frac{d}{2}}_{2,1})}+\|n^{\tau}-n \|_{L^{1}(\mathbb{R}_+;\dot{B}^{\frac{d}{2}}_{2,1}+\dot{B}^{\frac{d}{2}-1}_{2,1})} \\
&\quad\leq C_{2}\|(\rho_{0}^{\tau}-\rho _{0},n_{0}^{\tau}-n _{0})\|_{\dot{B}^{\frac{d}{2}-2}_{2,1}\cap\dot{B}^{\frac{d}{2}-1}_{2,1}}+C_{2}\left\|\frac{\rho_0^{\tau}}{\rho_0^{\tau}+n_0^{\tau}}u_0^{\tau}+\frac{n_0^{\tau}}{\rho_0^{\tau}+n_0^{\tau}}v_0^{\tau}-v_0\right\|_{\dot{B}^{\frac{d}{2}-2}_{2,1}}+ C_{2}\sqrt{\tau}.
\end{aligned}
\end{equation}
If the initial data further satisfies
\begin{equation}\label{idcd5}
\begin{aligned}
&\|(\rho_{0}^{\tau}-\rho _{0},n_{0}^{\tau}-n _{0})\|_{\dot{B}^{\frac{d}{2}-2}_{2,1}\cap\dot{B}^{\frac{d}{2}-1}_{2,1}}+\left\|\frac{\rho_0^{\tau}}{\rho_0^{\tau}+n_0^{\tau}}u_0^{\tau}+\frac{n_0^{\tau}}{\rho_0^{\tau}+n_0^{\tau}}v_0^{\tau}-v_0\right\|_{\dot{B}^{\frac{d}{2}-2}_{2,1}}\leq \sqrt{\tau},
\end{aligned}
\end{equation}
then the right-hand side of \eqref{result6} can be bounded by $\mathcal{O}(\sqrt{\tau})$. Consequently, as $\tau\rightarrow0$, $(\rho^{\tau},u^{\tau},n^{\tau},v^{\tau})$ converges to $(\rho ,v ,n ,v )$ strongly in the following sense{\rm{:}}
\begin{equation}\label{conver}
\left\{
\begin{aligned}
&(\rho^{\tau},n^{\tau})\rightarrow (\rho ,n )&&\text{in}\quad L^{\infty}(\mathbb{R}_{+};\dot{B}^{\frac{d}{2}-\eta}_{2,1}),\\
&(u^{\tau},v^{\tau})\rightarrow (v ,v )&&\text{in}\quad  L^1(\mathbb{R}_+; \dot{B}^{\frac{d}{2}+1-\eta}_{2,1}) ,
\end{aligned}
\right.
\end{equation}
for an arbitrarily small constant $\eta\in(0,1]$.
\end{theorem}

\begin{remark}\normalfont
Theorem \ref{theorem3} provides the first global strong convergence result for the large-friction limit in terms of relative velocity relaxation, for initial data close to the equilibrium in critical spaces. To address the different velocities in error estimates, the key novelty lies in the introduction of the {\emph{effective mixed velocity}}:
\begin{align}\label{V0}
&V^{\tau}:=\frac{\rho^{\tau}}{\rho^{\tau}+n^{\tau}}u^{\tau}+\frac{n^{\tau}}{\rho^{\tau}+n^{\tau}}v^{\tau},
\end{align}
which enables us to reformulate the equations for three components $(u^\tau, n^\tau, v^\tau)$ to the two-component one in terms of $(n^\tau, V^\tau)$, up to some sources involving the relative velocity $u^\tau-v^\tau$ with an explicit convergence rate of order $\sqrt{\tau}$ in $\eqref{uvcon}_1$. To the best of our knowledge, the effective mixed velocity is entirely new and can offer broader applicability, for example, for the pressure case {\rm\cite{choijungM3AS21}} and the incompressible case {\rm\cite{choijungJMFM21}}.
\end{remark}

\begin{remark}\normalfont
Theorem \ref{theorem3} holds for general {\emph{ill-prepared}} initial data satisfying only the velocity difference condition $u_0^\tau-v_0^\tau=\mathcal{O}(1)$, in contrast to the {\emph{well-prepared}} case requiring $u_0^\tau-v_0^\tau=\mathcal{O}(\tau)$. To achieve it, our idea is to develop a careful analysis of the mixed velocity error term $V^\tau-v$.
\end{remark}

Then, we investigate the uniform time asymptotics in the large-friction regime. Under an additional lower-order regularity assumption, we establish the optimal time-decay rates of global solutions to the Cauchy problem \eqref{pENS}-\eqref{indt}, uniformly with respect to the friction parameter $\tau$.

\begin{theorem}\label{theorem2}
Let $\tau\in(0,1)$ be a constant. Suppose that the assumptions of Theorem \ref{theorem1} hold, and $(\rho^{\tau},u^{\tau},n^{\tau},v^{\tau})$ is the corresponding global strong solution to the Cauchy problem \eqref{pENS}-\eqref{indt} obtained in Theorem \ref{theorem1}. For any $d\geq 2$, there exists a positive constant $\delta_{1}>0$ independent of the parameter $\tau$ such that if the initial data $(\rho^{\tau}_{0}, u^{\tau}_{0}, n^{\tau}_{0}, v^{\tau}_{0})$ further satisfies for $\sigma_{1}\in [-\frac{d}{2},\frac{d}{2}-1)$ that
\begin{equation}\label{idcd2}
\begin{aligned}
\|(\rho^{\tau}_{0},u^{\tau}_{0}, n^{\tau}_{0}-1, v^{\tau}_{0})\|_{ \dot{B}^{\sigma_{1}}_{2,\infty}}^{\ell}\leq \delta_{1},
\end{aligned}
\end{equation}
then it holds for all $t\geq1$ and a generic constant $C_{3}>0$ that
\begin{equation}\label{result2}
\left\{
\begin{aligned}
&\|(u^{\tau},v^{\tau})(t)\|_{\dot{B}^{\sigma}_{2,1}}\leq C_{3}(1+t)^{-\frac{1}{2}(\sigma-\sigma_{1})},&&\sigma\in(\sigma_{1},\frac{d}{2}+1],\\
&\|(n^{\tau}-1)(t)\|_{\dot{B}^{\sigma}_{2,1}}\leq C_{3}(1+t)^{-\frac{1}{2}(\sigma-\sigma_{1})},&&\sigma\in(\sigma_{1},\frac{d}{2}],\\
&\|(u^{\tau}-v^{\tau})(t)\|_{\dot{B}^{\sigma}_{2,1}}\leq C_{3}(1+t)^{-\frac{1}{2}(\sigma-\sigma_{1}+1)},&&\sigma\in(\sigma_{1},\frac{d}{2}],\\
&\|(u^{\tau}-v^{\tau})(t)\|_{\dot{B}^{\sigma}_{2,1}}\leq C_{3}\tau(1+t)^{-\frac{1}{2}(\sigma-\sigma_{1}+1)},\quad &&\sigma\in(\sigma_{1},\frac{d}{2}-1].
\end{aligned}
\right.
\end{equation}
Additionally, for any $d\geq 3$, since $\operatorname{div} (\rho^\tau u^\tau)$ decays in $\dot{B}^{\sigma}_{2,1}$ with the rate $(1+t)^{-\frac{1}{2}(\sigma-\sigma_{1}+1)}$ for $\sigma_1\in[-\frac{d}{2},\frac{d}{2}-2)$ and $\sigma\in(\sigma_{1}+1,\frac{d}{2}-1]$ such that $\frac{1}{2}(\sigma-\sigma_{1}+1)>1$, we have
\begin{equation}\label{result3}
\begin{aligned}
\|(\rho^{\tau}-\rho^{\tau}_{\infty})(t)\|_{\dot{B}^{\sigma}_{2,1}}&\leq C_{3}(1+t)^{-\frac{1}{2}(\sigma-\sigma_{1}-1)},\qquad \sigma\in(\sigma_{1}+1,\frac{d}{2}-1],
\end{aligned}
\end{equation}
where $\rho^{\tau}_{\infty}$ is the asymptotic profile of $\rho^\tau$, satisfying
\begin{equation}\label{aspf}
\rho^{\tau}_{\infty}(x):=\rho^{\tau}_{0}(x)-\int_{0}^{\infty}\operatorname{div} (\rho^{\tau} u^{\tau})(t,x)\, {\rm d}t.
\end{equation}
\end{theorem}

\begin{remark}\normalfont
The time convergence rates of $(u^\tau,n^\tau-1,v^\tau)$ in \eqref{result2} are the same as those for the solution to the heat equation, which are optimal in the sense of linear analysis {\rm(}see Lemma \ref{lemmalinearoptimal}{\rm)}. However, for $\sigma\in (\sigma_1,\frac{d}{2}]$, while $(u^\tau,n^\tau-1,v^\tau)$ decays in $\dot{B}^{\sigma}_{2,1}$ at a rate $(1+t)^{-\frac{1}{2}(\sigma - \sigma_1)}$, the relative velocity $u^\tau-v^\tau$ exhibits an enhanced decay in $\dot{B}^{\sigma}_{2,1}$ at a faster rate $(1+t)^{-\frac{1}{2}(\sigma-\sigma_1 + 1)}$.
\end{remark}

\begin{remark}\normalfont
Under the classical $L^1$-assumption (slightly stronger than $\dot{B}^{-\frac{d}{2}}_{2,\infty}$), applying the interpolation arguments on \eqref{result2} for any $p\geq2$ and $t\geq 1$, the following optimal $L^p$-stability hold:
\begin{equation*}
\left\{
\begin{aligned}
&\|\Lambda^{\sigma}(u^{\tau},v^{\tau})(t)\|_{L^{p}}\lesssim (1+t)^{-\frac{1}{2}(\sigma+d-\frac{d}{p})},&&\sigma+\frac{d}{2}-\frac{d}{p}\in(-\frac{d}{2},\frac{d}{2}+1],\\
&\|\Lambda^{\sigma}(n^{\tau}-1)(t)\|_{L^{p}}\lesssim (1+t)^{-\frac{1}{2}(\sigma+d-\frac{d}{p})},&&\sigma+\frac{d}{2}-\frac{d}{p}\in(-\frac{d}{2},\frac{d}{2}],\\
&\|\Lambda^{\sigma}(u^{\tau}-v^{\tau})(t)\|_{L^{p}}\lesssim  (1+t)^{-\frac{1}{2}(\sigma+d-\frac{d}{p}+1)},\quad &&\sigma+\frac{d}{2}-\frac{d}{p}\in(-\frac{d}{2},\frac{d}{2}].
\end{aligned}
\right.
\end{equation*}
Due to the embedding property $\dot{B}^{\frac{d}{2}}_{2,1}\hookrightarrow L^{\infty}$, one can recover the optimal decay rates for the crucial Lipschitz norm of $(u^\tau, v^\tau)$ and the upper bound of $u^\tau-v^\tau$:
\begin{equation*}
\|(\nabla u^{\tau},\nabla v^{\tau})(t)\|_{L^{\infty}}+\|(u^{\tau}-v^{\tau})(t)\|_{L^{\infty}}\lesssim (1+t)^{-\frac{d}{2}-\frac{1}{2}}.
\end{equation*}
In addition, when the dimension $d\geq 3$, we have
\begin{equation*}
\|\Lambda^{\sigma}(\rho^{\tau}-\rho^{\tau}_{\infty})(t)\|_{L^p}\lesssim (1+t)^{-\frac{1}{2}(\sigma+d-\frac{d}{p}-1)},\qquad\sigma+\frac{d}{2}-\frac{d}{p}\in(-\frac{d}{2}+1,\frac{d}{2}-1].
\end{equation*}
\end{remark}

Since the decay estimates \eqref{result2}-\eqref{result3} in Theorem \ref{theorem2} are uniform with respect to the parameter $\tau$, we can pass to the limit as $\tau\rightarrow0$ and derive the optimal time-decay estimates of solutions to the Cauchy problem \eqref{DF}-\eqref{DFindt} for the DF model.

\begin{cor}\label{corollary2}
Suppose that the assumptions of Corollary \ref{corollary1} hold, and $(\rho,n,v)$ is the corresponding global solution to the Cauchy problem \eqref{DF}-\eqref{DFindt} stated in Corollary \ref{corollary1}. For any $d\geq 2$, there exists a positive constant $\delta_1^*>0$ such that if the initial data $(\rho_{0},n_{0},v_{0})$ further satisfies for $\sigma_{1}\in [-\frac{d}{2},\frac{d}{2}-1)$ that
\begin{equation}\label{idcd20}
\begin{aligned}
\|(\rho_{0},n_{0}-1, v_{0})\|_{ \dot{B}^{\sigma_{1}}_{2,\infty}}\leq \delta^*_{1},
\end{aligned}
\end{equation}
then it holds for all $t\geq1$ and a generic constant $C_{3}^{*}>0$ that
\begin{equation}\label{result20}
\left\{
\begin{aligned}
&\|(n-1)(t)\|_{\dot{B}^{\sigma}_{2,1}}\leq C_{3}^{*}(1+t)^{-\frac{1}{2}(\sigma-\sigma_{1})},\quad &&\sigma\in(\sigma_{1},\frac{d}{2}],\\
&\|v(t)\|_{\dot{B}^{\sigma}_{2,1}}\leq C_{3}^{*}(1+t)^{-\frac{1}{2}(\sigma-\sigma_{1})},&&\sigma\in(\sigma_{1},\frac{d}{2}+1].
\end{aligned}
\right.
\end{equation}
Additionally, for any $d\geq 3$, it holds for $\sigma_1\in[-\frac{d}{2},\frac{d}{2}-2)$ that
\begin{align}
&\|(\rho-\rho_{\infty})(t)\|_{\dot{B}^{\sigma}_{2,1}}\leq C_{3}^{*}(1+t)^{-\frac{1}{2}(\sigma-\sigma_{1}-1)},\qquad\sigma\in(\sigma_{1}+1,\frac{d}{2}-1],
\end{align}
where the asymptotic profile $\rho_{\infty}$ of $\rho$ is given by
\begin{equation}
\rho_{\infty}(x):=\rho_{0}(x)-\int_{0}^{\infty}\operatorname{div} (\rho v)(t,x)\, {\rm d}t.
\end{equation}
\end{cor}

\subsection{Incompressible limit}

We are in a position to investigate the incompressible limits of the scaled DF system \eqref{DFvar} and the scaled Euler-NS system \eqref{pENSvar} toward the TNS system \eqref{TNS}.
For later use, we denote by $\mathcal{P}$ and $\mathcal{Q}$ the orthogonal projectors onto divergence-free and potential vector fields such that
\begin{align}
\mathcal{P}:={\rm Id}+\nabla(-\Delta)^{-1}\operatorname{div}\quad \text{and}\quad \mathcal{Q}:=-\nabla(-\Delta)^{-1}\operatorname{div}.
\end{align}
For any function $z\in\mathcal{S}_h'(\mathbb{R}^{d})$, we define
\begin{align}
\|z\|_{\dot{B}^{s}_{2,1}}^{\var,\ell}:=\sum_{j\in\mathbb{Z},\, 2^j\leq \frac{1}{\var}}2^{js}\|\dot{\Delta}_{j}z\|_{L^2} \quad\text{and}\quad \|z\|_{\dot{B}^{s}_{2,1}}^{\var,h}:=\sum_{j\in\mathbb{Z},\, 2^j\geq \frac{1}{2\var}}2^{js}\|\dot{\Delta}_{j}z\|_{L^2},
\end{align}
and we write $\|z\|_{\dot{B}^{s}_{2,1}}^{\ell}=\|z\|_{\dot{B}^{s}_{2,1}}^{1,\ell}$ and $\|z\|_{\dot{B}^{s}_{2,1}}^{h}=\|z\|_{\dot{B}^{s}_{2,1}}^{1,h}$ for simplicity. 

\vspace{3mm}

First, we have the following incompressible limit result for the DF model \eqref{DFvar}.

\begin{theorem}\label{Theoremincompressible1}
Let $\var\in(0,1)$ be a constant. For any $d\geq 2$, there exists a positive constant $\delta_{2}>0$ such that if the initial data $(\rho^\var_0, a_0^\var,v_0^\var)$ satisfies $\rho_0^\var\geq0$ for $x\in\mathbb{R}^{d}$ and
\begin{equation}\label{Incompress:a1}
\begin{aligned}				
\|(\rho^\var_{0},a_0^\var)\|_{\dot{B}^{\frac{d}{2}-1}_{2,1}}+\var\|(\rho_0^\var,a^\var_0)\|_{\dot{B}^{\frac{d}{2}}_{2,1}}+\|v^\var_{0}\|_{\dot{B}^{\frac{d}{2}-1}_{2,1}}\leq \delta_{2},
\end{aligned}
\end{equation}
then a unique strong solution $(\rho^\var,a^\var,v^\var)$ to the Cauchy problem for the system \eqref{DFvar} with the initial data $(\rho^\var_0, a_0^\var,v_0^\var)$ exists globally and satisfies
\begin{align}\label{Incompress:r1}
&\|\rho^\var\|_{\widetilde{L}^{\infty}(\mathbb{R}_+;\dot{B}^{\frac{d}{2}-1}_{2,1})}\leq  C_{4}\|\rho_0^\var\|_{\dot{B}^{\frac{d}{2}-1}_{2,1}},\quad \var\|\rho^\var\|_{\widetilde{L}^{\infty}(\mathbb{R}_+;\dot{B}^{\frac{d}{2}}_{2,1})}\leq  C_{4}\var\|\rho_0^\var\|_{\dot{B}^{\frac{d}{2}}_{2,1}},
\end{align}
and
\begin{equation}\label{Incompress:r11}
\begin{aligned}
&\|a^\var\|_{\widetilde{L}^{\infty}(\mathbb{R}_+;\dot{B}^{\frac{d}{2}-1}_{2,1})}+\var\|a^\var\|_{\widetilde{L}^{\infty}(\mathbb{R}_+;\dot{B}^{\frac{d}{2}}_{2,1})}+\|v^\var\|_{\widetilde{L}^{\infty}(\mathbb{R}_+;\dot{B}^{\frac{d}{2}-1}_{2,1})}\\
&\quad+\|a^\var\|_{L^{1}(\mathbb{R}_+;\dot{B}^{\frac{d}{2}+1}_{2,1})}^{\var,\ell}+\frac{1}{\var}\|a^\var\|_{L^{1}(\mathbb{R}_+;\dot{B}^{\frac{d}{2}}_{2,1})}^{\var,h}+\|v^\var\|_{L^{1}(\mathbb{R}_+;\dot{B}^{\frac{d}{2}+1}_{2,1})}\\
&\quad \leq C_{4}\Big(\|a_0^\var\|_{\dot{B}^{\frac{d}{2}-1}_{2,1}}+\var\|a^\var_0\|_{\dot{B}^{\frac{d}{2}}_{2,1}}+\|v^\var_{0}\|_{\dot{B}^{\frac{d}{2}-1}_{2,1}}\Big),
\end{aligned}
\end{equation}
for some constant $C_{4}>0$ independent of $\var$.

Furthermore, it holds for $p\in(2,\infty)$ that
\begin{equation}
\left\{
\begin{aligned}\label{Incompress:limit1}
&\|(a^\var, \mathcal{Q}v^\var)\|_{\widetilde{L}^{2}(\mathbb{R}_+;\dot{B}^{\frac{d+1}{p}-\frac{1}{2}}_{p,1})}\leq C_{4}\var^{\frac{1}{2}-\frac{1}{p}},&& \text{if}~~ d\geq 3,\\
&\|(a^\var, \mathcal{Q}v^\var)\|_{\widetilde{L}^{2}(\mathbb{R}_+;\dot{B}^{\frac{5}{2p}-\frac{1}{4}}_{p,1})}\leq C_{4}\var^{\frac{1}{4}-\frac{1}{2p}},&& \text{if}~~ d=2.\\
\end{aligned}
\right.
\end{equation}
Assume that there exist $\varrho_{0}$ and $w_{0}$ such that as $\var\rightarrow 0$, $\rho_0^\var$ and $\mathcal{P}v_0^\var$ converge weakly to $\varrho_0$ and $w_{0}$, respectively. Then,  $(\rho^\var,\mathcal{P}v^\var)$ converges to the limit $(\varrho,w)$ in the sense of distributions, where $(\varrho,w)$ is the global solution of the TNS system \eqref{TNS} subject to the initial data $(\varrho_0,w_0)$.
\end{theorem}

\begin{remark}\normalfont
The convergence result stated in Theorem \ref{Theoremincompressible1} is justified for {\emph{ill-prepared}} initial data that satisfy $\operatorname{div} v_0^\var=\mathcal{O}(1)$ and $\nabla P(1+\var a^\var_0)=\mathcal{O}(\var)$, contrasting with the {\emph{well-prepared}} case, which requires $\operatorname{div} v_0^\var=\mathcal{O}(\var)$ and $\nabla P(1+\var a^\var_0)=\mathcal{O}(\var^2)$. 
\end{remark}

Then, we have the combined large-friction and incompressible limit result for the Euler-NS system \eqref{pENSvar} toward the TNS system \eqref{TNS}.

\begin{theorem}\label{Theoremincompressible2}
Let $\tau=\var\in(0,1)$ be a constant. For any $d\geq2$, there exists a positive constant $\delta_{3}>0$ such that if 
the initial data $(\rho^\var_0,u_0^\var,a_0^\var,v_0^\var)$ satisfies $\rho_0^\var>0$ for $x\in\mathbb{R}^d$ and
\begin{equation}\label{Incompress:a2}
\begin{aligned}
\|(\rho^\var_{0}, a^\var_{0})\|_{\dot{B}^{\frac{d}{2}-1}_{2,1}}+\var\|(\rho^\var_{0}, a^\var_{0})\|_{\dot{B}^{\frac{d}{2}}_{2,1}}+\|(u^\var_{0},v^\var_{0})\|_{\dot{B}^{\frac{d}{2}-1}_{2,1}}+\var^2\|u^\var_{0}\|_{\dot{B}^{\frac{d}{2}+1}_{2,1}}\leq \delta_{3},
\end{aligned}
\end{equation}
then a unique strong solution $(\rho^\var,u^\var,a^\var,v^\var)$ to the Cauchy problem for the system \eqref{pENSvar} with the initial data $(\rho^\var_0,u_0^\var,a_0^\var,v_0^\var)$ exists globally and satisfies
\begin{equation}\label{Incompress:r2}
\|\rho^\var\|_{\widetilde{L}^{\infty}(\mathbb{R}_+;\dot{B}^{\frac{d}{2}-1}_{2,1})}\leq C_{5}\|\rho_0^\var\|_{\dot{B}^{\frac{d}{2}-1}_{2,1}},\quad \var\|\rho^\var\|_{\widetilde{L}^{\infty}(\mathbb{R}_+;\dot{B}^{\frac{d}{2}}_{2,1})}\leq C_{5}\var\|\rho_0^\var\|_{\dot{B}^{\frac{d}{2}}_{2,1}},
\end{equation}
and
\begin{equation}\label{Incompress:r21}
\begin{aligned}
&\|(u^\var,v^\var)\|_{\widetilde{L}^{\infty}(\mathbb{R}_+;\dot{B}^{\frac{d}{2}-1}_{2,1})}+\var^2\|u^\var\|_{\widetilde{L}^{\infty}(\mathbb{R}_+;\dot{B}^{\frac{d}{2}+1}_{2,1})}+\|a^\var\|_{\widetilde{L}^{\infty}(\mathbb{R}_+;\dot{B}^{\frac{d}{2}-1}_{2,1})}+\var\|a^\var\|_{\widetilde{L}^{\infty}(\mathbb{R}_+;\dot{B}^{\frac{d}{2}}_{2,1})}\\
&\quad+\|(u^\var,v^\var)\|_{L^{1}(\mathbb{R}_+;\dot{B}^{\frac{d}{2}+1}_{2,1})}+\|a^\var\|_{L^{1}(\mathbb{R}_+;\dot{B}^{\frac{d}{2}+1}_{2,1})}^{\var,\ell}+\frac{1}{\var}\|a^\var\|_{L^{1}(\mathbb{R}_+;\dot{B}^{\frac{d}{2}}_{2,1})}^{\var,h}\\
&\quad\leq C_{5}\Big(\|(u^\var_{0},v^\var_{0})\|_{\dot{B}^{\frac{d}{2}-1}_{2,1}}+\var^2\|u^\var_{0}\|_{\dot{B}^{\frac{d}{2}+1}_{2,1}}+\| a^\var_{0}\|_{\dot{B}^{\frac{d}{2}-1}_{2,1}}+\var\| a^\var_{0}\|_{\dot{B}^{\frac{d}{2}}_{2,1}}\Big),
\end{aligned}
\end{equation}
for some constant $C_{5}>0$ independent of $\tau$ and $\var$. In addition, the relative velocity $u^\var-v^\var$ has the following estimates
\begin{equation}\label{Incompress:r22}
\begin{aligned}
&\|u^\var-v^\var\|_{L^{1}(\mathbb{R}_+;\dot{B}^{\frac{d}{2}}_{2,1})}+\|u^\var-v^\var\|_{\widetilde{L}^{2}(\mathbb{R}_+;\dot{B}^{\frac{d}{2}-1}_{2,1})}\leq C_{5}\var.
\end{aligned}
\end{equation}

Furthermore, it holds that
\begin{equation}
\left\{
\begin{aligned}\label{Incompress:limit2}
&\|(a^\var,\mathcal{Q}u^\var,\mathcal{Q}v^\var)\|_{\widetilde{L}^{2}(\mathbb{R}_+;\dot{B}^{\frac{d+1}{p}-\frac{1}{2}}_{p,1})}\leq C_{5}\var^{\frac{1}{2}-\frac{1}{p}},&& \text{if}~~ d\geq 3,\,\, p\in (2,\infty)\\
&\|(a^\var,\mathcal{Q}u^\var,\mathcal{Q}v^\var)\|_{\widetilde{L}^{2}(\mathbb{R}_+;\dot{B}^{\frac{5}{2p}-\frac{1}{4}}_{p,1})}\leq C_{5}\var^{\frac{1}{4}-\frac{1}{2p}},&& \text{if}~~ d=2,\,\,p\in (2,6].
\end{aligned}
\right.
\end{equation}
Assume that there exist $\varrho_0$ and $w_0$ such that as $\var\rightarrow 0$, $\rho_0^\var$, $\mathcal{P}u_0^\var$ and $\mathcal{P}v_0^\var $ converge weakly to $\varrho_0$, $w_{0}$ and $w_{0}$, respectively. Then,  $(\rho^\var,\mathcal{P}u^\var,\mathcal{P}v^\var)$ converges to the limit $(\varrho,w,w)$ in the sense of distributions, where $(\varrho,w)$ is the global solution of the TNS system \eqref{TNS} with the initial data $(\varrho_0,w_0)$.
\end{theorem}

\begin{remark}\normalfont
Theorem \ref{Theoremincompressible2} holds for {\emph{ill-prepared}} initial data satisfying $u_0^\var-v_0^\var=\mathcal{O}(1)$, $\operatorname{div} u_0^\var=\mathcal{O}(1)$, and $\nabla P(1+\var a^\var_0)=\mathcal{O}(\var)$, while the {\emph{well-prepared}} case requires stronger conditions $u_0^\var-v_0^\var=\mathcal{O}(\var)$, $\operatorname{div} u_0^\var=\mathcal{O}(\var)$, and $\nabla P(1+\var a^\var_0)=\mathcal{O}(\var^2)$.
\end{remark}

\subsection{Key strategies}

To begin with, we consider the Euler-NS system \eqref{pENS} near the equilibrium state $(0,0,1,0)$. In what follows, we set
\begin{align*}P'(1)=1.
\end{align*}
Denote the perturbation functions
\begin{equation*}
a^\tau:=n^\tau-1\quad\text{and}\quad a^\tau_{0}:=n^\tau_{0}-1,
\end{equation*}
then the Cauchy problem \eqref{pENS}-\eqref{indt} can be reformulated into
\begin{equation}\label{pENS1}
\left\{
\begin{aligned}
& \partial_{t}\rho^\tau+\operatorname{div} (\rho^\tau u^\tau)=0,\\
& \partial_{t}u^\tau+u^\tau\cdot\nabla u^\tau=-\frac{1}{\tau}(u^\tau-v^\tau),\\
&\partial_{t}a^\tau+\operatorname{div} v^\tau+\operatorname{div}(a^\tau v^\tau)=0,\\
&\partial_{t}v^\tau+v^\tau\cdot\nabla v^\tau+\nabla a^\tau-\mu\Delta v^\tau-(\mu+\lambda)\nabla\operatorname{div} v^\tau\\
&\quad\quad\quad\quad=\frac{1}{\tau}\rho^\tau(u^\tau-v^\tau)+F_{1}(a^\tau,v^\tau)+F_{2}(\rho^\tau,u^\tau,a^\tau,v^\tau),\\
&(\rho^\tau,u^\tau,a^\tau,v^\tau)(0,x)=(\rho^\tau_{0},u^\tau_{0},a^\tau_{0},v^\tau_{0})(x),
\end{aligned}
\right.
\end{equation}
where the nonlinear terms $F_{1}$ and $F_{2}$ satisfy
\begin{equation}\label{nonltm}
\left\{
\begin{aligned}
&F_{1}(a^\tau,v^\tau):=g(a^\tau)\nabla a^\tau+\mu f(a^\tau)\Delta v^\tau+(\mu+\lambda)f(a^\tau)\nabla\operatorname{div} v^\tau,\\
&F_{2}(\rho^\tau,u^\tau,a^\tau,v^\tau):=\frac{1}{\tau}f(a^\tau)\rho^\tau(u^\tau-v^\tau),\\
&g(a^\tau):=-\frac{P'(1+a^\tau)}{1+a^\tau}+ 1,\quad f(a^\tau):=-\frac{a^\tau}{1+a^\tau}.
\end{aligned}
\right.
\end{equation}
We note that the equation $\eqref{pENS1}_1$ can be treated with the pure mass equation, while $\eqref{pENS1}_3$-$\eqref{pENS1}_4$ forms a hyperbolic-parabolic system; in contrast, the equation $\eqref{pENS1}_2$ exhibits a coupling structure between the two velocities $u^\tau$ and $v^\tau$.

Regarding the global existence result for the Euler-NS system, we focus on the reformulated problem \eqref{pENS1} and establish the uniform {\emph{a priori}} estimates.
The main difficulty arises from the nonlinear term $\frac{1}{\tau}\rho^\tau(u^\tau-v^\tau)$ in $\eqref{pENS1}_4$, which, compared to other nonlinear terms, is of lower-order and highly singular for small $\tau$. Therefore, one requires the $L^1$ time integrability for the relative velocity $u^\tau-v^\tau$ since the density $\rho$ does not exhibit dissipation or time integrability due to the pressureless structure.
On the other hand, the presence of the linear term $\frac{1}{\tau}v^\tau$ in $\eqref{pENS1}_2$ makes it nontrivial to derive the expected estimates for $u$ directly from $\eqref{pENS1}_2$.
To overcome these difficulties, we shall derive new estimates for the relative velocity $u^\tau-v^\tau$. Our analysis 
carefully tracks the influence of the friction damping associated with the parameter $\tau$. The strategies can be summarized as follows.

\begin{itemize}
\item Since $\eqref{pENS1}_{3}$-$\eqref{pENS1}_4$ can be viewed as the isentropic compressible Navier-Stokes equations, after a careful analysis of the lower-order nonlinear term caused by the drag force, one can first derive the uniform regularity estimates of $(a^\tau,v^\tau)$ in \eqref{pENS1}.

\item It is observed that the relative velocity $u^\tau-v^\tau$ satisfies the equation
\begin{equation}\label{eqrv0}
\begin{aligned}
\partial_{t}(u^\tau-v^\tau)+\frac{1}{\tau}(u^\tau-v^\tau)&=\nabla a^\tau-\mu\Delta v^\tau-(\mu+\lambda)\nabla\operatorname{div} v^\tau\\
&\quad\quad-u^\tau\cdot\nabla u^\tau+v^\tau\cdot\nabla v^\tau-\frac{1}{\tau}\rho^\tau(u^\tau-v^\tau)-F_{1}-F_{2}.
\end{aligned}
\end{equation} 
Using the strong damping effect in \eqref{eqrv0} and applying the dissipation estimates for the higher-order terms $\nabla a^\tau$ and $\nabla^2 v^\tau$, we can establish the uniform convergence estimates for $u^\tau-v^\tau$ in \eqref{pENS1}.

\item The lower-order uniform estimate of $u^\tau$ in $\widetilde{L}^{\infty}_t(\dot{B}^{\frac{d}{2}-1}_{2,1})$ is recovered from the bounds of $u^\tau-v^\tau$ and $v^\tau$.

\item Treating $\eqref{pENS1}_2$ as a strongly damped transport equation of $u^\tau$ with a source term $\frac{1}{\tau}v^\tau$, we derive the uniform bounds for $\tau u^\tau$ in $\widetilde{L}^{\infty}_t(\dot{B}^{\frac{d}{2}+1}_{2,1})$ and $u^\tau$ in $L^1_t(\dot{B}^{\frac{d}{2}+1}_{2,1})$.

\item In view of the transport structure of $\eqref{pENS1}_1$, together with the uniform $L^1_t(\dot{B}^{\frac{d}{2}+1}_{2,1})$-regularity of $u^\tau$, we establish the uniform $\widetilde{L}^{\infty}_t(\dot{B}^{\frac{d}{2}-1}_{2,1}\cap \dot{B}^{\frac{d}{2}}_{2,1})$-estimate of $\rho^\tau$.
\end{itemize}
Combining the above estimates of $(\rho^\tau,u^\tau,a^\tau,v^\tau)$ and $u^\tau-v^\tau$, we can enclose the uniform {\emph{a priori}} estimates and derive the global existence of strong solutions (cf. Section \ref{sectionexistence}).

Next, we explain the ideas for obtaining the error estimates stated in Theorem \ref{theorem3}. For notational simplicity, we denote the error variables $(\widetilde{\rho}^\tau,\widetilde{u}^\tau,\widetilde{n}^\tau,\widetilde{v}^\tau):=(\rho^\tau-\rho,u^\tau-v,n^\tau-n,v^\tau-v)$, where $(\rho^\tau,u^\tau,n^\tau,v^\tau)$ and $(\rho,n,v)$ are the global solutions to the Euler-NS system \eqref{pENS} given by Theorem \ref{theorem1} and the DF model \eqref{DF} given by Corollary \ref{corollary1}, respectively. To overcome the singularity that appears in $\frac{1}{\tau} \rho^\tau (u^\tau - v^\tau)$, we add $\eqref{pENS}_{2}$ and $\eqref{pENS}_{4}$ together to obtain
\begin{equation}\label{rhousum}
\begin{aligned}
\rho^{\tau} \partial_{t}u^{\tau}+n^{\tau}\partial_{t}v^{\tau}+\rho^{\tau} u^{\tau}\cdot\nabla u^{\tau}+n^{\tau} v^{\tau}\cdot \nabla v^{\tau}+\nabla P(n^{\tau})= \mu\Delta v^{\tau}+(\mu+\lambda)\nabla\operatorname{div} v^{\tau}.
\end{aligned}
\end{equation}
However, a new difficulty arises from how to address the mixed time derivatives and convective terms in \eqref{rhousum}. To this end, the introduction of the effective mixed velocity $V^\tau$ in \eqref{V0} allows us to reduce the two-velocity formulation to a single-velocity equation:
\begin{equation*}
\partial_{t}V^{\tau}+v^\tau\cdot \nabla V^\tau+\frac{1}{\rho^\tau+n^\tau}\nabla P(n^{\tau})= \frac{\mu}{\rho^\tau+n^\tau}\Delta v^{\tau}+\frac{\mu+\lambda}{\rho^\tau+n^\tau}\nabla\operatorname{div} v^{\tau}+\mathbf{h},
\end{equation*}
where $\mathbf{h}$ can be controlled in $L^1_t(\dot{B}^{\frac{d}{2}-2}_{2,1})$ via the bounds of the relative velocity $u^\tau-v^\tau$, exhibiting a rate of order $\sqrt{\tau}$. We introduce the {\emph{effective mixed error}} as
\begin{align}
&\widetilde{V}^\tau:=\frac{\rho^{\tau}}{\rho^{\tau}+n^{\tau}}\widetilde{u}^\tau+\frac{n^{\tau}}{\rho^{\tau}+n^{\tau}}\widetilde{v}^\tau=V^{\tau}-v.\label{errV0}
\end{align}
This enables us to rewrite the equations in terms of $\widetilde{n}^\tau$ and $\widetilde{V}^\tau$ as a coupled hyperbolic-parabolic system:
\begin{equation}\notag
\left\{
\begin{aligned}
&\partial_{t}\widetilde{n}^\tau+v^{\tau}\cdot\nabla\widetilde{n}^\tau+ \operatorname{div} \widetilde{V}^\tau=\widetilde{F}_{1}+\widetilde{F}_{2},\\
&\partial_{t}\widetilde{V}^\tau+v^{\tau}\cdot \nabla \widetilde{V}^\tau+ \nabla \widetilde{n}^\tau- \mu\Delta \widetilde{V}^\tau-(\mu+\lambda)\nabla\operatorname{div} \widetilde{V}^\tau =\widetilde{F}_{3}+\widetilde{F}_{4},
\end{aligned}
\right.
\end{equation}
where $\widetilde{F}_i$ ($i=1,2,3,4)$ are given in \eqref{F1F2F3F4}.
Making use of the energy estimates on the coupled hyperbolic-parabolic system, together with the uniform regularity estimates obtained in Theorem \ref{theorem1} and Corollary \ref{corollary1}, we establish the error estimates for $(\widetilde{n}^\tau,\widetilde{V}^\tau)$. Consequently, we recover the expected estimates of $\widetilde{u}^\tau$ and $\widetilde{v}^\tau$ by the convergence of the relative velocity $u^\tau-v^\tau$ and the observation that
\begin{equation}\label{erruvV}
\left\{
\begin{aligned}	&\widetilde{u}^\tau=\widetilde{V}^\tau+\frac{n^{\tau}}{\rho^{\tau}+n^{\tau}}(u^{\tau}-v^{\tau}),\\
&\widetilde{v}^\tau=\widetilde{V}^\tau- \frac{\rho^{\tau}}{\rho^{\tau}+n^{\tau}}(u^{\tau}-v^{\tau}).
\end{aligned}
\right.
\end{equation} 
Note that by adopting $\widetilde{V}^\tau$, we eliminate the additional condition on $u^\tau-v^\tau$, thereby dealing with the {\emph{ill-prepared}} initial data (cf. Section \ref{sectionlimit}).

Then, we return to the perturbed problem \eqref{pENS1} to investigate the uniform large-time behavior of the global solution. The uniform time-decay estimates for $(u^\tau,a^\tau,v^\tau)$ can be established by adapting the time-weighted method (cf. \cite{danchinxu17}). The key difficulties are the singular term $\frac{1}{\tau}\rho^\tau(u^\tau-v^\tau)$ and the nondecaying property of the density $\rho^\tau$. To overcome these issues, we establish improved decay estimates for the relative velocity $u^\tau-v^\tau$ with explicit convergence rates depending on $\tau$, which rely on the damping structure in \eqref{eqrv0} together with the fast decay of the higher-order terms $\nabla a^\tau$ and $\nabla^2 v^\tau$. Furthermore, inspired by \cite{danchin2026elementary}, we observe that
\begin{align}
\rho^\tau(t,x)-\rho^\tau_{\infty}(x)=\int_{t}^{\infty} \operatorname{div} (\rho^\tau u^\tau)(s,x){\rm d}s. \label{rhominusrhoinfty} 
\end{align}
Utilizing the time-decay rate of $\operatorname{div}(\rho^\tau u^\tau)$ for $\sigma\in(\sigma_1+1,\frac{d}{2}-1]$ with $\sigma_1\in[-\frac{d}{2},\frac{d}{2}-2)$, we further derive the convergence rate of $\rho^\tau$ to the asymptotic profile $\rho^\tau_{\infty}$ (cf. Section \ref{sectiondecay}).

Finally, we give some comments on the combined large-friction and incompressible limit for the Euler-NS system given in Theorem \ref{Theoremincompressible2}. It is observed that equation $\eqref{TNS}_2$ can be derived from the sum of $\eqref{pENSvar}_2$ and $\eqref{pENSvar}_4$, analogous to the incompressible limit of the isentropic viscous flows (e.g., see \cite{desjardinsJMPA99,lionsmasmoudi98,desjardins99,danchin02,danchinMA16}). Thanks to the advantage of critical norms, the uniform regularity estimates \eqref{Incompress:r21}-\eqref{Incompress:r22} can be obtained by Theorem \ref{theorem1} with $\tau=\var=1$ and a rescaling argument. However, there are two main challenges to justify the convergence process for our goals. On one hand, due to the presence of density $\rho^\var$, some singular terms (e.g., $\frac{1}{\var}\frac{P'(1+\var a^\var)}{1+\var a^\var}\nabla a^\var$) cannot be eliminated by using the incompressible projector on the equation of velocity, compared with the incompressible limit for single-phase compressible viscous fluids (refer to \cite{danchinMA16}). On the other hand, after adding $\eqref{pENSvar}_2$ and $\eqref{pENSvar}_4$ together, one must deal with the mixed momentum $\var \rho^\var u^\var+(1+\var a^\var) v^\var$. To resolve these, we apply the rescaled effective mixed velocity $V^\var$ in \eqref{V000} and establish the convergence rates for $(a^\var,\mathcal{Q}V^\var)$. Combining the Strichartz estimates for the acoustic wave equations of $(a^\var,\mathcal{Q}V^\var)$ with the convergence property \eqref{Incompress:r22} for the relative velocity $u^\var-v^\var$, we are able to recover the convergence rates in \eqref{Incompress:limit2} for $(a^\var,\mathcal{Q}u^\var,\mathcal{Q}v^\var)$, which enable us to rigorously justify the combined singular limit (cf. Section \ref{section:incompressible}).

\section{Global existence and uniform regularity for Euler-NS system}\label{sectionexistence} 

In this section, we prove Theorem \ref{theorem1} on the global existence and uniqueness of solutions to the Cauchy problem \eqref{pENS1}. For simplicity, we omit the superscript $^\tau$ from the solution and initial data throughout this section.

\subsection{Uniform regularity estimates}
To prove the global existence result, the key points are to derive the {\emph{a priori}} estimates for any strong solution  uniformly with respect to the time and the friction parameter $\tau$, and then to apply a bootstrap argument. We will carry out a uniform energy argument to recover the optimal dissipative properties as expected in \hyperref[sectionlinear]{Appendix A}. To this end, it is convenient to define the dissipation functional
\begin{equation}\label{priorieg}
\begin{aligned}
\mathcal{D}(t)&:=\|u\|_{L^{1}_{t}(\dot{B}^{\frac{d}{2}+1}_{2,1})}+\|a\|_{L^{1}_{t}(\dot{B}^{\frac{d}{2}+1}_{2,1})}^{\ell}+\|a\|_{L^{1}_{t}(\dot{B}^{\frac{d}{2}}_{2,1})}^{h}+\|v\|_{L^{1}_{t}(\dot{B}^{\frac{d}{2}+1}_{2,1})} \\
&\qquad+\frac{1}{\sqrt{\tau}}\|u-v\|_{\widetilde{L}^{2}_{t}(\dot{B}^{\frac{d}{2}-1}_{2,1})}+\frac{1}{\tau}\|u-v\|_{L^{1}_{t}(\dot{B}^{\frac{d}{2}}_{2,1})}^{\ell}+\frac{1}{\tau}\|u-v\|_{L^{1}_{t}(\dot{B}^{\frac{d}{2}-1}_{2,1})}^{h},
\end{aligned}
\end{equation}
and the norm of the initial data
\begin{equation}\label{priorieg0}
\mathcal{X}_{0}:=\|u_{0}\|_{\dot{B}^{\frac{d}{2}-1}_{2,1}}+\tau\|u_{0}\|_{\dot{B}^{\frac{d}{2}+1}_{2,1}}+\|a_{0}\|_{\dot{B}^{\frac{d}{2}-1}_{2,1}\cap\dot{B}^{\frac{d}{2}}_{2,1}}+\|v_{0}\|_{\dot{B}^{\frac{d}{2}-1}_{2,1}}.
\end{equation}

\begin{prop}\label{proppriori}
Let $\tau\in(0,1)$. Suppose that $(\rho,u,a,v)$ is a smooth solution to the Cauchy problem \eqref{pENS1} defined on $[0,T)\times\mathbb{R}^d$ with a given time $T>0$. For any $t\in(0,T)$, if, for some small constant $\beta\in(0,1)$ {\rm{(}}to be chosen later{\rm{)}}, it holds that
\begin{equation}\label{ainfty}
\begin{aligned}
\|(\rho,a)\|_{\widetilde{L}^{\infty}_{t}(\dot{B}^{\frac{d}{2}-1}_{2,1}\cap \dot{B}^{\frac{d}{2}}_{2,1})}+\|(u,v)\|_{\widetilde{L}^{\infty}_{t}(\dot{B}^{\frac{d}{2}-1}_{2,1})}+\tau \|u\|_{\widetilde{L}^{\infty}_{t}(\dot{B}^{\frac{d}{2}+1}_{2,1})}\leq \beta,
\end{aligned}
\end{equation}
then the solution $(\rho,u,a,v)$ satisfies
\begin{align}\label{3.40}
\|\rho\|_{\widetilde{L}^{\infty}_{t}(\dot{B}^{\frac{d}{2}-1}_{2,1})}\leq C_{0}\|\rho_0\|_{\dot{B}^{\frac{d}{2}-1}_{2,1}} ,\quad \|\rho\|_{\widetilde{L}^{\infty}_{t}(\dot{B}^{\frac{d}{2}}_{2,1})}\leq C_{0}\|\rho_0\|_{\dot{B}^{\frac{d}{2}}_{2,1}},
\end{align}
and
\begin{align}
&\|(u,v)\|_{\widetilde{L}^{\infty}_{t}(\dot{B}^{\frac{d}{2}-1}_{2,1})}+\tau\|u\|_{\widetilde{L}^{\infty}_{t}(\dot{B}^{\frac{d}{2}+1}_{2,1})}+\|a\|_{\widetilde{L}^{\infty}_{t}(\dot{B}^{\frac{d}{2}-1}_{2,1}\cap\dot{B}^{\frac{d}{2}}_{2,1})}+\mathcal{D}(t)\leq C_{0}		\mathcal{X}_{0},\label{3.4}
\end{align}
where $C_{0}>0$ is a constant independent of the time $T$ and the parameter $\tau$.
\end{prop}

The proof of Proposition \ref{proppriori} consists of Lemmas \ref{lemav}-\ref{lemrho} as follows. 

\subsubsection{Estimates of \texorpdfstring{$(a,v)$}{(a,v)}}\label{sub1}
First, we analyze the coupled hyperbolic-parabolic structure of $\eqref{pENS1}_{3}$-$\eqref{pENS1}_{4}$ and establish the uniform regularity estimates of $a$ and $v$.
	
\begin{lemma}\label{lemav}
Let $(\rho,u,a,v)$ be a smooth solution to $\eqref{pENS1}$ on $[0,T)\times\mathbb{R}^d$. Then, under the condition \eqref{ainfty}, it holds
\begin{equation}\label{aves}
\begin{aligned}
&\|(a,\nabla a,v)\|_{\widetilde{L}^{\infty}_{t}(\dot{B}^{\frac{d}{2}-1}_{2,1})}+\|a\|_{L^{1}_{t}(\dot{B}^{\frac{d}{2}+1}_{2,1})}^{\ell}+\|a\|_{L^{1}_{t}(\dot{B}^{\frac{d}{2}}_{2,1})}^{h}+\|v\|_{L^{1}_{t}(\dot{B}^{\frac{d}{2}+1}_{2,1})}\\
&\quad\lesssim \|(a_{0},\nabla a_{0}, v_{0})\|_{\dot{B}^{\frac{d}{2}-1}_{2,1}}+\beta \mathcal{D}(t).
\end{aligned}
\end{equation}
\end{lemma}

\begin{proof}
Define the Lyapunov functional 
\begin{align*}
&\mathcal{E}_{j}(t):=\frac{1}{2}\|\dot{\Delta}_{j}(a,v)\|_{L^{2}}^{2}+\eta_{*}\Big( \int_{\mathbb{R}^{d}}\dot{\Delta}_{j}v\cdot\nabla\dot{\Delta}_{j}a \,{\rm{d}}x+\frac{2\mu+\lambda}{2}\|\nabla\dot{\Delta}_{j}a\|_{L^{2}}^{2}\Big),\\
&\mathcal{D}_{j}(t):=\mu\|\nabla\dot{\Delta}_{j}v\|_{L^{2}}^{2}+(\mu+\lambda)\|\operatorname{div}\dot{\Delta}_{j}v\|_{L^{2}}^{2}+\eta_{*} \Big(\|\nabla\dot{\Delta}_{j}a\|_{L^{2}}^{2}-\|\operatorname{div}\dot{\Delta}_{j}v\|_{L^{2}}^{2}\Big),
\end{align*}
with positive constant $\eta_{*}>0$. Making use of $\eqref{pENS1}_{3}$-$\eqref{pENS1}_{4}$, we can show
\begin{equation}\label{ns6}
\begin{aligned}
\frac{{\rm{d}}}{{\rm{d}}t}\mathcal{E}_{j}(t)+\mathcal{D}_{j}(t)&\lesssim\Big(\|\operatorname{div}v\|_{L^{\infty}}\|\dot{\Delta}_{j}(a,\nabla a, v)\|_{L^{2}}+\|[v\cdot\nabla,\dot{\Delta}_{j}]a\|_{L^{2}}+\|[v\cdot\nabla,\dot{\Delta}_{j}]v\|_{L^{2}}\\
&\qquad+\sum_{k=1}^{d}\|[v\cdot\nabla,\partial_{x_k}\dot{\Delta}_{j}]a\|_{L^{2}}+\|\dot{\Delta}_{j}(a\operatorname{div} v)\|_{L^{2}}+\|\nabla\dot{\Delta}_{j}(a\operatorname{div} v)\|_{L^{2}}\\
&\qquad+\frac{1}{\tau}\|\dot{\Delta}_{j}(\rho(u-v))\|_{L^{2}}+\|\dot{\Delta}_{j}(F_{1},F_{2})\|_{L^{2}}\Big)\|\dot{\Delta}_{j}(a,\nabla a, v)\|_{L^{2}}.
\end{aligned}
\end{equation}
Then, choosing a sufficiently small constant $\eta_{*}>0$, we deduce that
\begin{equation*}
\begin{aligned}
&\mathcal{E}_{j}(t)\sim \|\dot{\Delta}_{j}(a,\nabla a, v)\|_{L^{2}}^{2}\quad \text{and}\quad \mathcal{D}_{j}(t)\gtrsim \min\{1,2^{2j}\}\|\dot{\Delta}_{j}(a,\nabla a, v)\|_{L^{2}}^{2}.
\end{aligned}
\end{equation*}
Consequently, it holds
\begin{equation}\label{ns7}
\begin{aligned}
&\|\dot{\Delta}_{j}(a,\nabla a, v)\|_{L^{\infty}_{t}(L^{2})}+\min\{1,2^{2j}\} \|\dot{\Delta}_{j}(a,\nabla a, v)\|_{L^{1}_{t}(L^{2})}\\[2mm]
&\quad\lesssim \|\dot{\Delta}_{j}(a_{0},\nabla a_{0}, v_{0})\|_{L^{2}}+\|\operatorname{div}v\|_{L_{t}^{1}(L^{\infty})}\|\dot{\Delta}_{j}(a,\nabla a, v)\|_{L_{t}^{\infty}(L^{2})}+\|[v\cdot\nabla,\dot{\Delta}_{j}]a\|_{L_{t}^{1}(L^{2})}\\
&\qquad+\|[v\cdot\nabla,\dot{\Delta}_{j}]v\|_{L_{t}^{1}(L^{2})}+\sum_{k=1}^{d}\|[v\cdot\nabla,\partial_{x_k}\dot{\Delta}_{j}]a\|_{L_{t}^{1}(L^{2})}+\|\dot{\Delta}_{j}(a\operatorname{div}v)\|_{L_{t}^{1}(L^{2})}\\
&\qquad+\|\nabla\dot{\Delta}_{j}(a\operatorname{div}v)\|_{L_{t}^{1}(L^{2})}+\frac{1}{\tau}\|\dot{\Delta}_{j}(\rho(u-v))\|_{L_{t}^{1}(L^{2})}+\|\dot{\Delta}_{j}(F_{1},F_{2})\|_{L_{t}^{1}(L^{2})}.
\end{aligned}
\end{equation}
Note that $\dot{\Delta}_{j}(a,\nabla a,v)$ behaves like a heat equation in low frequencies and exhibits damping in high frequencies.  One can refer to \cite{bahourietal11} for similar computations on the compressible Navier-Stokes equations.

The summation of \eqref{ns7} with the weight $2^{j(\frac{d}{2}-1)}$ over $j\in\mathbb{Z}$ leads to
\begin{equation}\label{nse1}
\begin{aligned}
&\|(a,\nabla a,v)\|_{\widetilde{L}^{\infty}_{t}(\dot{B}^{\frac{d}{2}-1}_{2,1})}+\|(a,\nabla a,v)\|_{L^{1}_{t}(\dot{B}^{\frac{d}{2}+1}_{2,1})}^{\ell}+\|(a,\nabla a,v)\|_{L^{1}_{t}(\dot{B}^{\frac{d}{2}-1}_{2,1})}^{h}\\
&\quad\lesssim \|(a_{0},\nabla a_{0}, v_{0})\|_{\dot{B}^{\frac{d}{2}-1}_{2,1}}+\|v\|_{L_{t}^{1}(\dot{B}^{\frac{d}{2}+1}_{2,1})}\|(a,\nabla a,v)\|_{L_{t}^{\infty}(\dot{B}^{\frac{d}{2}-1}_{2,1})}\\
&\qquad+\frac{1}{\tau}\|\rho(u-v)\|_{L^{1}_{t}(\dot{B}^{\frac{d}{2}-1}_{2,1})}+\|(F_{1},F_{2})\|_{L^{1}_{t}(\dot{B}^{\frac{d}{2}-1}_{2,1})}.
\end{aligned}
\end{equation}
Here we have applied the product law in \eqref{uv2}, the commutator estimates in Lemma \ref{commutator}, and the embedding property $\dot{B}^{\frac{d}{2}}_{2,1}\hookrightarrow L^{\infty}$.  For the higher-order estimates of the parabolic component $v$, we consider the following equation
\begin{equation}\label{vheat}
\begin{aligned}
\partial_{t}v-\mu\Delta v-(\mu+\lambda)\nabla\operatorname{div} v=-\nabla a-v\cdot\nabla v+\frac{1}{\tau}\rho(u-v)+F_{1}+F_{2}.
\end{aligned}
\end{equation}
Applying the maximal regularity estimates in Lemma \ref{esheat} to \eqref{vheat}, together with the product law \eqref{uv2}, we obtain
\begin{equation}\label{nse2}
\begin{aligned}
&\|v\|_{\widetilde{L}^{\infty}_{t}(\dot{B}^{\frac{d}{2}-1}_{2,1})}^{h}+\|v\|_{L^{1}_{t}(\dot{B}^{\frac{d}{2}+1}_{2,1})}^{h}\\
&\quad\lesssim \|v_{0}\|_{\dot{B}^{\frac{d}{2}-1}_{2,1}}^{h}+\|\nabla a\|_{L^{1}_{t}(\dot{B}^{\frac{d}{2}-1}_{2,1})}^{h}+\|v\|_{\widetilde{L}^{\infty}_{t}(\dot{B}^{\frac{d}{2}-1}_{2,1})}\|\nabla v\|_{L^{1}_{t}(\dot{B}^{\frac{d}{2}}_{2,1})}\\
&\qquad+\frac{1}{\tau}\|\rho(u-v)\|_{L^{1}_{t}(\dot{B}^{\frac{d}{2}-1}_{2,1})}+\|(F_{1},F_{2})\|_{L^{1}_{t}(\dot{B}^{\frac{d}{2}-1}_{2,1})}.
\end{aligned}
\end{equation}

The key difficulty lies in analyzing the lower-order nonlinear term $\frac{1}{\tau}\rho(u-v)$. To overcome the singular coefficient $\frac{1}{\tau}$ and the lack of dissipation of $\rho$, we shall make use of the convergence rate estimates for the relative velocity $u-v$ in $L^1_t(\dot{B}^{\frac{d}{2}}_{2,1}+\dot{B}^{\frac{d}{2}-1}_{2,1})$ with respect to the parameter $\tau$. By estimates for the product in \eqref{uv2}, we arrive at
\begin{align}\label{ns9}
&\begin{aligned}
&\frac{1}{\tau}\|\rho(u-v)\|_{L^{1}_{t}(\dot{B}^{\frac{d}{2}-1}_{2,1})}\\
&\quad\lesssim \frac{1}{\tau}\|\rho(u-v)^{\ell}\|_{L^{1}_{t}(\dot{B}^{\frac{d}{2}-1}_{2,1})}+\frac{1}{\tau}\|\rho(u-v)^h\|_{L^{1}_{t}(\dot{B}^{\frac{d}{2}-1}_{2,1})}\\
&\quad\lesssim \|\rho\|_{L^{\infty}_{t}(\dot{B}^{\frac{d}{2}-1}_{2,1})} \frac{1}{\tau}\|(u-v)^{\ell}\|_{L^{1}_{t}(\dot{B}^{\frac{d}{2}}_{2,1})}+\|\rho\|_{L^{\infty}_{t}(\dot{B}^{\frac{d}{2}}_{2,1})} \frac{1}{\tau} \|(u-v)^{h}\|_{L^{1}_{t}(\dot{B}^{\frac{d}{2}-1}_{2,1})}.
\end{aligned}
\end{align}      
Then we address the remaining nonlinear terms.  By virtue of Lemmas \ref{lemmaB5}-\ref{multicompo}, it is straightforward to obtain
\begin{equation}\label{ns8}
\begin{aligned}
&\|(F_{1},F_{2})\|_{L^{1}_{t}(\dot{B}^{\frac{d}{2}-1}_{2,1})}\\
&\quad\lesssim \|g(a)\nabla a^\ell\|_{L^{1}_{t}(\dot{B}^{\frac{d}{2}-1}_{2,1})}+\|g(a)\nabla a^h\|_{L^{1}_{t}(\dot{B}^{\frac{d}{2}-1}_{2,1})}+\|f(a)\Delta v\|_{L^{1}_{t}(\dot{B}^{\frac{d}{2}-1}_{2,1})}+\|f(a)\nabla\dive v\|_{L^{1}_{t}(\dot{B}^{\frac{d}{2}-1}_{2,1})}\\
&\qquad+\frac{1}{\tau}\|f(a)\rho(u-v)\|_{L^{1}_{t}(\dot{B}_{2,1}^{\frac{d}{2}-1})}\\
&\quad\lesssim \|a\|_{L^{\infty}_{t}(\dot{B}^{\frac{d}{2}-1}_{2,1})}\|a\|_{L^{1}_{t}(\dot{B}^{\frac{d}{2}+1}_{2,1})}^{\ell}+\|a\|_{L^{\infty}_{t}(\dot{B}^{\frac{d}{2}}_{2,1})}\|a\|_{L^{1}_{t}(\dot{B}^{\frac{d}{2}}_{2,1})}^{h}+\|a\|_{L^{\infty}_{t}(\dot{B}^{\frac{d}{2}}_{2,1})}\|v\|_{L^{1}_{t}(\dot{B}^{\frac{d}{2}+1}_{2,1})}\\
&\qquad+\|a\|_{L^{\infty}_{t}(\dot{B}^{\frac{d}{2}}_{2,1})}\Big(\|\rho\|_{L^{\infty}_{t}(\dot{B}^{\frac{d}{2}-1}_{2,1})} \frac{1}{\tau}\|(u-v)^{\ell}\|_{L^{1}_{t}(\dot{B}^{\frac{d}{2}}_{2,1})}+\|\rho\|_{L^{\infty}_{t}(\dot{B}^{\frac{d}{2}}_{2,1})} \frac{1}{\tau} \|(u-v)^{h}\|_{L^{1}_{t}(\dot{B}^{\frac{d}{2}-1}_{2,1})}\Big).
\end{aligned}
\end{equation}
Combining \eqref{ainfty}, \eqref{nse1}, and   \eqref{nse2}-\eqref{ns8}, we derive the estimate \eqref{aves} and complete the proof of Lemma \ref{lemav}.
\end{proof}

\subsubsection{Estimates of \texorpdfstring{$u-v$}{u-v}}
Then, we prove some key convergence estimates of the relative velocity $u-v$, which play a crucial role in handling the difficult nonlinear terms involving the friction drag force.
\begin{lemma}\label{lemrv}
Let $(\rho,u,a,v)$ be a smooth solution to $\eqref{pENS1}$ on $[0,T)\times\mathbb{R}^d$. Then, under the condition \eqref{ainfty}, it holds
\begin{equation}\label{uves}
\begin{aligned}
&\|u-v\|_{\widetilde{L}^{\infty}_{t}(\dot{B}^{\frac{d}{2}-1}_{2,1})}+\frac{1}{\sqrt{\tau}}\|u-v\|_{\widetilde{L}^{2}_{t}(\dot{B}^{\frac{d}{2}-1}_{2,1})}+\frac{1}{\tau}\|u-v\|_{L^{1}_{t}(\dot{B}^{\frac{d}{2}}_{2,1})}^{\ell}+\frac{1}{\tau}\|u-v\|_{L^{1}_{t}(\dot{B}^{\frac{d}{2}-1}_{2,1})}^{h}\\
&\quad\lesssim \|(u_{0},a_{0},\nabla a_{0}, v_{0})\|_{\dot{B}^{\frac{d}{2}-1}_{2,1}}+\beta \mathcal{D}(t).
\end{aligned}
\end{equation}
\end{lemma}

\begin{proof}
We consider the equation of the relative velocity $u-v$, which satisfies
\begin{equation}\label{eqrv}
\partial_{t}(u-v)+\frac{1}{\tau}(u-v)=\nabla a-\mu\Delta v-(\mu+\lambda)\nabla\operatorname{div} v-u\cdot\nabla u+v\cdot\nabla v-\frac{1}{\tau}\rho(u-v)-F_{1}-F_{2}.
\end{equation}
A direct computation yields that
\begin{equation}\label{rv1}
\begin{aligned}
&\frac{1}{2}\frac{{\rm{d}}}{{\rm{d}}t}\|\dot{\Delta}_{j}(u-v)\|_{L^{2}}^{2}+\frac{1}{\tau}\|\dot{\Delta}_{j}(u-v)\|_{L^{2}}^{2}\\
&\quad\lesssim \Big(2^j\| \dot{\Delta}_{j}a\|_{L^{2}}+2^{2j}\|\dot{\Delta}_{j}v\|_{L^{2}}+\|\dot{\Delta}_{j}(u\cdot\nabla u,v\cdot\nabla v,\tau^{-1}\rho(u-v),F_{1},F_{2})\|_{L^{2}}\Big) \|\dot{\Delta}_{j}(u-v)\|_{L^{2}},
\end{aligned}
\end{equation}
from which we infer
\begin{equation}\label{rv2}
\begin{aligned}
&\|\dot{\Delta}_{j}(u-v)\|_{L^{\infty}_{t}(L^{2})}+\frac{1}{\tau}\|\dot{\Delta}_{j}(u-v)\|_{L^{1}_{t}(L^{2})}\\
&\quad\lesssim \|\dot{\Delta}_{j}(u_{0}-v_{0})\|_{L^{2}}+2^j\|\dot{\Delta}_{j} a\|_{L^1_t(L^2)}+2^{2j}\|\dot{\Delta}_{j}v\|_{L^{1}_{t}(L^{2})}\\
&\qquad+\|\dot{\Delta}_{j}(u\cdot\nabla u,v\cdot\nabla v)\|_{L^{1}_{t}(L^{2})}+\frac{1}{\tau}\|\dot{\Delta}_{j}(\rho(u-v))\|_{L^{1}_{t}(L^{2})}+\|(F_{1},F_{2})\|_{L^{1}_{t}(L^{2})}.
\end{aligned}
\end{equation}
This, together with Lemma \ref{lemav}, \eqref{ainfty}, \eqref{ns9}, \eqref{ns8}, and \eqref{lh}, yields for low frequencies ($j\leq 0$) that
\begin{equation}\label{rve1}
\begin{aligned}
&\|u-v\|_{\widetilde{L}^{\infty}_{t}(\dot{B}^{\frac{d}{2}}_{2,1})}^{\ell}+\frac{1}{\tau}\|u-v\|_{L^{1}_{t}(\dot{B}^{\frac{d}{2}}_{2,1})}^{\ell}\\
&\quad\lesssim \|u_{0}-v_{0}\|_{\dot{B}^{\frac{d}{2}}_{2,1}}^{\ell}+\|a\|_{L^{1}_{t}(\dot{B}^{\frac{d}{2}+1}_{2,1})}^{\ell}+\|v\|_{L^{1}_{t}(\dot{B}^{\frac{d}{2}+1}_{2,1})}^{\ell}+\|(u,v)\|_{L^{\infty}_{t}(\dot{B}^{\frac{d}{2}-1}_{2,1})}\|(\nabla u,\nabla v)\|_{L^{1}_{t}(\dot{B}^{\frac{d}{2}}_{2,1})}\\
&\qquad+\frac{1}{\tau}\|\rho(u-v)\|_{L^{1}_{t}(\dot{B}^{\frac{d}{2}-1}_{2,1})}+\|(F_{1},F_{2})\|_{L^{1}_{t}(\dot{B}^{\frac{d}{2}-1}_{2,1})}\\
&\quad\lesssim \|(u_{0},a_{0},\nabla a_{0},v_{0})\|_{\dot{B}^{\frac{d}{2}-1}_{2,1}}+\beta \mathcal{D}(t),
\end{aligned}
\end{equation}
and for high frequencies ($j\geq -1$) that
\begin{equation}\label{rve2}
\begin{aligned}
&\|u-v\|_{\widetilde{L}^{\infty}_{t}(\dot{B}^{\frac{d}{2}-1}_{2,1})}^{h}+\frac{1}{\tau}\|u-v\|_{L^{1}_{t}(\dot{B}^{\frac{d}{2}-1}_{2,1})}^{h}\\
&\quad\lesssim \|u_{0}-v_{0}\|_{\dot{B}^{\frac{d}{2}-1}_{2,1}}^{h}+\|a\|_{L^{1}_{t}(\dot{B}^{\frac{d}{2}}_{2,1})}^{h}+\|v\|_{L^{1}_{t}(\dot{B}^{\frac{d}{2}+1}_{2,1})}^{h}+\beta\mathcal{D}(t)\\
&\quad\lesssim \|(u_{0},a_{0},\nabla a_{0},v_{0})\|_{\dot{B}^{\frac{d}{2}-1}_{2,1}}+\beta \mathcal{D}(t).
\end{aligned}
\end{equation}

Moreover, we devote ourselves to obtaining the estimates of $u-v$ in $\widetilde{L}^2_t(\dot{B}^{\frac{d}{2}-1}_{2,1})$. Employing Young's inequality to \eqref{rv1} and integrating the resultant over $[0,t]$, we derive
\begin{equation*}
\begin{aligned}
&\|\dot{\Delta}_{j}(u-v)\|_{L^{\infty}_{t}(L^{2})}^{2}+\frac{1}{\tau}\|\dot{\Delta}_{j}(u-v)\|_{L^{2}_{t}(L^{2})}^{2}	\\
&\quad\lesssim\|\dot{\Delta}_{j}(u_{0}-v_{0})\|_{L^{2}}^{2}+\tau2^{2j}\|\dot{\Delta}_{j}a\|_{L^{2}_{t}(L^{2})}^{2}+\tau2^{2j}\|\dot{\Delta}_{j}v\|_{L^{2}_{t}(L^{2})}^{2}\\
&\quad+\|\dot{\Delta}_{j}(u\cdot\nabla u,v\cdot\nabla v)\|_{L^{1}_{t}(L^{2})}\|\dot{\Delta}_{j}(u-v)\|_{L^{\infty}_{t}(L^{2})}+\frac{1}{\tau}\|\dot{\Delta}_{j}(\rho(u-v))\|_{L^{2}_{t}(L^{2})}^{2}\\
&\quad+\|\dot{\Delta}_{j}F_{1}\|_{L^{1}_{t}(L^{2})}\|\dot{\Delta}_{j}(u-v)\|_{L^{\infty}_{t}(L^{2})}+\tau\|\dot{\Delta}_{j}F_{2}\|_{L^{2}_{t}(L^{2})}^{2},
\end{aligned}
\end{equation*}
and this gives rise to
\begin{equation}\label{rv3}
\begin{aligned}
&\|\dot{\Delta}_{j}(u-v)\|_{L^{\infty}_{t}(L^{2})}^{2}+\frac{1}{\tau}\|\dot{\Delta}_{j}(u-v)\|_{L^{2}_{t}(L^{2})}^{2}\\
&\quad\lesssim\|\dot{\Delta}_{j}(u_{0},v_{0})\|_{L^{2}}^{2}+\tau2^{2j}\|\dot{\Delta}_{j}a\|_{L^{2}_{t}(L^{2})}^{2}+\tau2^{2j}\|\dot{\Delta}_{j}v\|_{L^{2}_{t}(L^{2})}^{2}+\|\dot{\Delta}_{j}(u\cdot\nabla u,v\cdot\nabla v)\|_{L^{1}_{t}(L^{2})}^{2}\\
&\qquad+\frac{1}{\tau}\|\dot{\Delta}_{j}(\rho(u-v))\|_{L^{2}_{t}(L^{2})}^{2}+\|\dot{\Delta}_{j}F_{1}\|_{L^{1}_{t}(L^{2})}^{2}+\tau\|\dot{\Delta}_{j}F_{2}\|_{L^{2}_{t}(L^{2})}^{2}.
\end{aligned}
\end{equation}
Taking the square root of both sides of \eqref{rv3} and summing it over $j\leq 0$ with the weight $2^{j(\frac{d}{2}-1)}$, we deduce that
\begin{equation}\label{rv4}
\begin{aligned}
&\|u-v\|_{\widetilde{L}^{\infty}_{t}(\dot{B}^{\frac{d}{2}-1}_{2,1})}^{\ell}+\frac{1}{\sqrt{\tau}}\|u-v\|_{\widetilde{L}^{2}_{t}(\dot{B}^{\frac{d}{2}-1}_{2,1})}^{\ell}\\
&\quad\lesssim \|(u_{0},v_{0})\|_{\dot{B}^{\frac{d}{2}-1}_{2,1}}^{\ell}+\sqrt{\tau}\|(a,v)\|_{\widetilde{L}^{2}_{t}(\dot{B}^{\frac{d}{2}}_{2,1})}^{\ell}+\|(u,v)\|_{L^{\infty}_{t}(\dot{B}^{\frac{d}{2}-1}_{2,1})}\|(\nabla u,\nabla v)\|_{L^{1}_{t}(\dot{B}^{\frac{d}{2}}_{2,1})}\\
&\qquad+\frac{1}{\sqrt{\tau}}\|\rho(u-v)\|_{\widetilde{L}^{2}_{t}(\dot{B}^{\frac{d}{2}-1}_{2,1})}^{\ell}+\|F_{1}\|_{L^{1}_{t}(\dot{B}^{\frac{d}{2}-1}_{2,1})}^{\ell}+\sqrt{\tau}\|F_{2}\|_{\widetilde{L}^{2}_{t}(\dot{B}^{\frac{d}{2}-1}_{2,1})}^{\ell}.
\end{aligned}
\end{equation}
It follows from \eqref{inter} and the fact $\tau\in(0,1)$ that 
\begin{equation}\label{mmfff}
\sqrt{\tau}\|(a,v)\|_{\widetilde{L}^{2}_{t}(\dot{B}^{\frac{d}{2}}_{2,1})}^{\ell}\lesssim \|(a,v)\|_{\widetilde{L}^{\infty}_{t}(\dot{B}^{\frac{d}{2}-1}_{2,1})}^{\ell}+\|(a,v)\|_{L^{1}_{t}(\dot{B}^{\frac{d}{2}+1}_{2,1})}^{\ell}.
\end{equation}
Making use of \eqref{ainfty} and the product law in \eqref{uv2}, we get
\begin{equation}
\frac{1}{\sqrt{\tau}}\|\rho(u-v)\|^{\ell}_{\widetilde{L}^{2}_{t}(\dot{B}^{\frac{d}{2}-1}_{2,1})}\lesssim\|\rho\|_{\widetilde{L}^{\infty}_{t}(\dot{B}^{\frac{d}{2}}_{2,1})}\frac{1}{\sqrt{\tau}}\|u-v\|_{\widetilde{L}^{2}_{t}(\dot{B}^{\frac{d}{2}-1}_{2,1})}\lesssim\beta\mathcal{D}(t),
\end{equation}
and similarly, one has
\begin{equation}\label{mmfff1}
\sqrt{\tau}\|F_{2}\|^{\ell}_{\widetilde{L}^{2}_{t}(\dot{B}^{\frac{d}{2}-1}_{2,1})}\lesssim\|a\|_{\widetilde{L}^{\infty}_{t}(\dot{B}^{\frac{d}{2}}_{2,1})}\|\rho\|_{\widetilde{L}^{\infty}_{t}(\dot{B}^{\frac{d}{2}}_{2,1})}\frac{1}{\sqrt{\tau}}\|u-v\|_{\widetilde{L}^{2}_{t}(\dot{B}^{\frac{d}{2}-1}_{2,1})}\lesssim\beta\mathcal{D}(t).
\end{equation}
Combining \eqref{rv4} with \eqref{ns8}, \eqref{mmfff}-\eqref{mmfff1}, and Lemma \ref{lemav}, we verify that
\begin{equation}\label{rve3}
\|u-v\|_{\widetilde{L}^{\infty}_{t}(\dot{B}^{\frac{d}{2}-1}_{2,1})}^{\ell}+\frac{1}{\sqrt{\tau}}\|u-v\|_{\widetilde{L}^{2}_{t}(\dot{B}^{\frac{d}{2}-1}_{2,1})}^{\ell}\lesssim \|(u_{0},a_{0},\nabla a_{0}, v_{0})\|_{\dot{B}^{\frac{d}{2}-1}_{2,1}}+\beta \mathcal{D}(t),
\end{equation}
and this, together with the high-frequency estimate \eqref{rve2} and H\"{o}lder's inequality, leads to
\begin{equation}\label{rve4}
\begin{aligned}
&\|u-v\|_{\widetilde{L}^{\infty}_{t}(\dot{B}^{\frac{d}{2}-1}_{2,1})}+\frac{1}{\sqrt{\tau}}\|u-v\|_{\widetilde{L}^{2}_{t}(\dot{B}^{\frac{d}{2}-1}_{2,1})}\\
&\quad\lesssim \|u-v\|_{\widetilde{L}^{\infty}_{t}(\dot{B}^{\frac{d}{2}-1}_{2,1})}^{\ell}+\frac{1}{\sqrt{\tau}}\|u-v\|_{\widetilde{L}^{2}_{t}(\dot{B}^{\frac{d}{2}-1}_{2,1})}^{\ell}+\|u-v\|_{\widetilde{L}^{\infty}_{t}(\dot{B}^{\frac{d}{2}-1}_{2,1})}^{h}+\frac{1}{\tau}\|u-v\|_{L^{1}_{t}(\dot{B}^{\frac{d}{2}-1}_{2,1})}^{h}\\
&\quad\lesssim \|(u_{0},a_{0},\nabla a_{0}, v_{0})\|_{\dot{B}^{\frac{d}{2}-1}_{2,1}}+\beta \mathcal{D}(t).
\end{aligned}
\end{equation}
Therefore, we complete the proof of Lemma \ref{lemrv}.
\end{proof}

\subsubsection{Estimates of \texorpdfstring{$u$}{u}}
We now recover the expected bounds of $u$ based on the estimates for $v$ and $u-v$. In particular, we establish the uniform $L^{1}$ time integrability of the Lipschitz bound of the velocity $u$ required in \eqref{L1Lip}. 

\begin{lemma}\label{lemau}
Let $(\rho,u,a,v)$ be a smooth solution to $\eqref{pENS1}$ on $[0,T)\times\mathbb{R}^d$. Then, under the condition \eqref{ainfty}, it holds
\begin{equation}\label{ues}
\begin{aligned}
&\|u\|_{\widetilde{L}^{\infty}_{t}(\dot{B}^{\frac{d}{2}-1}_{2,1})}+ \tau\|u\|_{\widetilde{L}^{\infty}_{t}(\dot{B}^{\frac{d}{2}+1}_{2,1})}+\|u\|_{L^{1}_{t}(\dot{B}^{\frac{d}{2}+1}_{2,1})}\\
&\quad\lesssim\|(u_{0},a_{0},\nabla a_{0}, v_{0})\|_{\dot{B}^{\frac{d}{2}-1}_{2,1}}+\tau\|u_{0}\|_{\dot{B}^{\frac{d}{2}+1}_{2,1}}+\beta \mathcal{D}(t).
\end{aligned}
\end{equation}
\end{lemma}

\begin{proof}
It is easy to verify from \eqref{aves} and \eqref{uves} that
\begin{equation}
\begin{aligned}
\|u\|_{\widetilde{L}^{\infty}_{t}(\dot{B}^{\frac{d}{2}-1}_{2,1})}&\leq \|v\|_{\widetilde{L}^{\infty}_{t}(\dot{B}^{\frac{d}{2}-1}_{2,1})}+\|u-v\|_{\widetilde{L}^{\infty}_{t}(\dot{B}^{\frac{d}{2}-1}_{2,1})}\\
&\lesssim\|(u_{0},a_{0},\nabla a_{0}, v_{0})\|_{\dot{B}^{\frac{d}{2}-1}_{2,1}}+\beta\mathcal{D}(t).
\end{aligned}
\end{equation}
Applying the operator $\dot{\Delta}_{j}$ to $\eqref{pENS1}_{2}$, we have
\begin{equation}\label{veu1}
\partial_{t}\dot{\Delta}_{j}u+\frac{1}{\tau}\dot{\Delta}_{j}u=\frac{1}{\tau}\dot{\Delta}_{j}v-u\cdot\nabla\dot{\Delta}_{j}u+[u\cdot\nabla,\dot{\Delta}_{j}]u.
\end{equation}
Multiplying \eqref{veu1} by $\dot{\Delta}_{j}u$ and integrating the resulting equation over $\mathbb{R}^{d}$, we obtain
\begin{equation}\label{veu2}
\begin{aligned}
&\frac{1}{2}\frac{{\rm{d}}}{{\rm{d}}t}\|\dot{\Delta}_{j}u\|_{L^{2}}^{2}+\frac{1}{\tau}\|\dot{\Delta}_{j}u\|_{L^{2}}^{2}\\
&\quad\leq\Big(\frac{1}{\tau}\|\dot{\Delta}_{j}v\|_{L^{2}}+\frac{1}{2}\|{\rm div}\,u\|_{L^\infty}\|\dot{\Delta}_{j}u\|_{L^{2}}+\|[u\cdot\nabla,\dot{\Delta}_{j}]u\|_{L^{2}}\Big)\|\dot{\Delta}_{j}u\|_{L^{2}},
\end{aligned}
\end{equation}
which yields that
\begin{equation}\label{veu3}
\begin{aligned}
&\|\dot{\Delta}_{j}u\|_{L_{t}^{\infty}(L^{2})}+\frac{1}{\tau}\|\dot{\Delta}_{j}u\|_{L_{t}^{1}(L^{2})}\\
&\quad\lesssim\|\dot{\Delta}_{j}u_{0}\|_{L^{2}}+\frac{1}{\tau}\|\dot{\Delta}_{j}v\|_{L_{t}^{1}(L^{2})}+\|{\rm div}\,u\|_{L_{t}^{1}(L^{\infty})}\|\dot{\Delta}_{j}u\|_{L_{t}^{\infty}(L^{2})}+\|[u\cdot\nabla,\dot{\Delta}_{j}]u\|_{L_{t}^{1}(L^{2})}.
\end{aligned}
\end{equation}
The summation of \eqref{veu3} with the weight $\tau 2^{j(\frac{d}{2}+1)}$ over $j\in\mathbb{Z}$ shows that
\begin{equation}\label{veue1}
\begin{aligned}
&\tau\|u\|_{\widetilde{L}^{\infty}_{t}(\dot{B}^{\frac{d}{2}+1}_{2,1})}+\|u\|_{L^{1}_{t}(\dot{B}^{\frac{d}{2}+1}_{2,1})}\\
&\quad\lesssim \tau \|u_{0}\|_{\dot{B}^{\frac{d}{2}+1}_{2,1}}+\|v\|_{L^{1}_{t}(\dot{B}^{\frac{d}{2}+1}_{2,1})}+\tau\|u\|_{\widetilde{L}_{t}^{\infty}(\dot{B}^{\frac{d}{2}+1}_{2,1})}\|u\|_{L_{t}^{1}(\dot{B}_{2,1}^{\frac{d}{2}+1})}\\
&\quad\lesssim\tau \|u_{0}\|_{\dot{B}^{\frac{d}{2}+1}_{2,1}}+\|(a_{0},\nabla a_{0}, v_{0})\|_{\dot{B}^{\frac{d}{2}-1}_{2,1}}+\beta \mathcal{D}(t),
\end{aligned}
\end{equation}
where we have applied Lemma \ref{lemav}. Thus, we derive the desired estimate \eqref{ues} and complete the proof of Lemma \ref{lemau}.
\end{proof}

\subsubsection{Estimates of \texorpdfstring{$\rho$}{rho}}\label{sub2}
Finally, we establish the estimate of the density $\rho$, making use of the transport nature of $\eqref{pENS1}_{1}$ and the $L^{1}$ time integrability of the Lipschitz bound for $u$.
\begin{lemma}\label{lemrho}
Let $(\rho,u,a,v)$ be a smooth solution to $\eqref{pENS1}$ on $[0,T)\times\mathbb{R}^d$. Then, if $\rho_0(x)>0$ for $x\in\mathbb{R}^d$, we have $\rho(t,x)>0$ for $(t,x)\in [0,T]\times\mathbb{R}^d$ and
\begin{equation}\label{rhoes}
\begin{aligned}
\|\rho\|_{\widetilde{L}^{\infty}_{t}(\dot{B}^{\frac{d}{2}-1}_{2,1})}\leq e^{C\mathcal{D}(t)} \|\rho_0\|_{\dot{B}^{\frac{d}{2}-1}_{2,1}},\quad \|\rho\|_{\widetilde{L}^{\infty}_{t}(\dot{B}^{\frac{d}{2}}_{2,1})}\leq e^{C\mathcal{D}(t)} \|\rho_0\|_{\dot{B}^{\frac{d}{2}}_{2,1}}.
\end{aligned}
\end{equation}
\end{lemma}

\begin{proof}
The positivity of $\rho(t,x)$ for $(t,x)\in [0,T]\times\mathbb{R}^{d}$ can be directly justified by the transport nature of $\eqref{pENS1}_{1}$. For any $j\in\mathbb{Z}$, applying the operator $\dot{\Delta}_{j}$ to $\eqref{pENS1}_{1}$, we have
\begin{equation}\label{rho1}
\partial_{t}\dot{\Delta}_{j}\rho+u\cdot\nabla\dot{\Delta}_{j}\rho-[u\cdot\nabla,\dot{\Delta}_{j}]\rho+\dot{\Delta}_{j}(\rho\operatorname{div}u)=0.
\end{equation}
Multiplying \eqref{rho1} by $\dot{\Delta}_{j}\rho$ and integrating the resulting equation over $\mathbb{R}^d$, we deduce that
\begin{equation}\label{rrr1}
\begin{aligned}
\frac{1}{2}\frac{{\rm{d}}}{{\rm{d}}t}\|\dot{\Delta}_{j}\rho\|_{L^{2}}^{2}&=\frac{1}{2}\int_{\mathbb{R}^d}{\rm div}\,u|\dot{\Delta}_{j}\rho|^{2}\,{\rm{d}}x+\int_{\mathbb{R}^d}[u\cdot\nabla,\dot{\Delta}_{j}]\rho\dot{\Delta}_{j}\rho \,{\rm{d}}x-\int_{\mathbb{R}^d}\dot{\Delta}_{j}(\rho{\rm div}\,u)\dot{\Delta}_{j}\rho \,{\rm{d}}x\\
&\leq\Big(\frac{1}{2}\|\operatorname{div} u\|_{L^{\infty}}\|\dot{\Delta}_{j}\rho\|_{L^{2}}+\|[u\cdot\nabla,\dot{\Delta}_{j}]\rho\|_{L^{2}}+\|\dot{\Delta}_{j}(\rho\operatorname{div} u)\|_{L^{2}}\Big)\|\dot{\Delta}_{j}\rho\|_{L^{2}}.
\end{aligned}
\end{equation}
This, combined with the embedding property $\dot{B}^{\frac{d}{2}}_{2,1}\hookrightarrow L^{\infty}$, the product law in \eqref{uv2}, and the commutator estimates in Lemma \ref{commutator}, gives rise to
\begin{equation}\label{sfgmgg}
\begin{aligned}
\|\rho\|_{\widetilde{L}^{\infty}_{t}(\dot{B}^{s'}_{2,1})}&\lesssim \|\rho_{0}\|_{\dot{B}^{s'}_{2,1}}+\int_0^t\|u\|_{\dot{B}^{\frac{d}{2}+1}_{2,1}}\|\rho\|_{\dot{B}^{s'}_{2,1}} {\rm{d}}s,\quad \text{for}~~s'=\frac{d}{2}-1,\, \frac{d}{2}.
\end{aligned}
\end{equation}
Since $\|\rho(t)\|_{\dot{B}^{\frac{d}{2}-1}_{2,1}\cap\dot{B}^{\frac{d}{2}}_{2,1}}$ is continuous on $t\in [0,T)$,  employing Gr\"onwall's inequality for \eqref{sfgmgg}, we arrive at \eqref{rhoes} and the proof is completed.
\end{proof}
	
\begin{proof}[\indent\textbf{Proof of Proposition \ref{proppriori}}]
It follows from Lemmas \ref{lemav}-\ref{lemau} that there exists a uniform positive constant $C_*>0$ such that 
\begin{align*}
&\|(u,v)\|_{\widetilde{L}^{\infty}_{t}(\dot{B}^{\frac{d}{2}-1}_{2,1})}+\tau\|u\|_{\widetilde{L}^{\infty}_{t}(\dot{B}^{\frac{d}{2}+1}_{2,1})}+\|a\|_{\widetilde{L}^{\infty}_{t}(\dot{B}^{\frac{d}{2}-1}_{2,1}\cap\dot{B}^{\frac{d}{2}}_{2,1})}+\mathcal{D}(t)\leq C_*	\mathcal{X}_{0}+C_*\beta \mathcal{D}(t).
\end{align*}
Choosing $\beta\leq\frac{1}{2C_*}$, we have
\begin{align*}
&\|(u,v)\|_{\widetilde{L}^{\infty}_{t}(\dot{B}^{\frac{d}{2}-1}_{2,1})}+\tau\|u\|_{\widetilde{L}^{\infty}_{t}(\dot{B}^{\frac{d}{2}+1}_{2,1})}+\|a\|_{\widetilde{L}^{\infty}_{t}(\dot{B}^{\frac{d}{2}-1}_{2,1}\cap\dot{B}^{\frac{d}{2}}_{2,1})}+\mathcal{D}(t)\leq 2 C_*\mathcal{X}_0,
\end{align*}
and this, together with Lemma \ref{lemrho}, implies
\begin{equation*}
\begin{aligned}
\|\rho\|_{\widetilde{L}^{\infty}_{t}(\dot{B}^{\frac{d}{2}-1}_{2,1})}\leq e^{2C C_*\mathcal{X}_0} \|\rho_0\|_{\dot{B}^{\frac{d}{2}-1}_{2,1}}\quad\text{and}\quad \|\rho\|_{\widetilde{L}^{\infty}_{t}(\dot{B}^{\frac{d}{2}}_{2,1})}\leq e^{2 CC_*\mathcal{X}_0} \|\rho_0\|_{\dot{B}^{\frac{d}{2}}_{2,1}}.
\end{aligned}
\end{equation*}
We thus obtain the desired estimates \eqref{3.40}-\eqref{3.4} and finish the proof of Proposition \ref{proppriori}.
\end{proof}

\subsection{Proof of Theorem {\rm\ref{theorem1}}}\label{subsection32}
At the end of this section, we construct a local Friedrichs approximation (refer to \cite{bahourietal11}). With the uniform-in-time {\emph{a priori}} estimates established in Sections \ref{sub1}-\ref{sub2} at hand, we extend the local approximate sequence to a global one and show its convergence to the expected global solution of the Cauchy problem \eqref{pENS1}.

Let $\dot{\mathbb{E}}_{q}:L^{2}\rightarrow L^{2}_{q}$ be the Friedrichs projector, defined by $\dot{\mathbb{E}}_{q}f:=\mathcal{F}^{-1}\big({\rm\bf{1}}_{\mathcal{C}_{q}}(\xi)\mathcal{F}f(\xi) \big)$. Here $L^{2}_{q}$ is the set of $L^{2}$ functions spectrally supported in the annulus $\mathcal{C}_{q}=\{\xi\in\mathbb{R}^{d}~|~\frac{1}{q}\leq|\xi|\leq q\}$ endowed with the standard $L^{2}$ topology, and ${\rm\bf{1}}_{\mathcal{C}_{q}}$ is the characteristic function on the annulus $\mathcal{C}_{q}$. We aim to solve the following approximate problem for $q\geq 1$:
\begin{equation}\label{apENS}
\left\{
\begin{aligned}
& \partial_{t}\rho^{q}+\operatorname{div} \dot{\mathbb{E}}_{q}(\rho^{q} u^{q})=0,\\
& \partial_{t}u^{q}+\dot{\mathbb{E}}_{q}(u^{q}\cdot\nabla u^{q})+\frac{1}{\tau}\dot{\mathbb{E}}_{q}(u^{q}-v^{q})=0,\\
&\partial_{t}a^{q}+\operatorname{div} \dot{\mathbb{E}}_{q}v^{q}+\operatorname{div}\dot{\mathbb{E}}_{q}(a^{q}v^{q})=0,\\
&\partial_{t}v^{q}+\dot{\mathbb{E}}_{q}(v^{q}\cdot\nabla v^{q})+\nabla \dot{\mathbb{E}}_{q}a^{q}-\mu\Delta \dot{\mathbb{E}}_{q}v^{q}-(\mu+\lambda)\nabla\operatorname{div} \dot{\mathbb{E}}_{q}v^{q}\\
&\quad\quad=\frac{1}{\tau} \dot{\mathbb{E}}_{q}\big(\rho^{q}(u^{q}-v^{q})\big)+\dot{\mathbb{E}}_{q}F_{1}(a^{q},v^{q})+\dot{\mathbb{E}}_{q}F_{2}(\rho^{q},u^{q},a^{q},v^{q}),\\
&(\rho^{q},u^{q},a^{q},v^{q})(0,x)=(\dot{\mathbb{E}}_{q}\rho_{0},\dot{\mathbb{E}}_{q}u_{0},\dot{\mathbb{E}}_{q}a_{0},\dot{\mathbb{E}}_{q}v_{0})(x).
\end{aligned}
\right.
\end{equation}
Note that the initial data $(\dot{\mathbb{E}}_{q}\rho_{0},\dot{\mathbb{E}}_{q}u_{0},\dot{\mathbb{E}}_{q}a_{0},\dot{\mathbb{E}}_{q}v_{0})$ uniformly satisfy \eqref{idcd1} with respect to $q\geq 1$ and strongly converge to $(\rho_{0},u_{0},a_{0},v_{0})$ in the sense \eqref{idcd1} as $q\rightarrow \infty$. Due to the Bernstein inequality, all the Sobolev norms are equivalent and hence the system \eqref{apENS} is in $L^{2}_{q}\times L^{2}_{q}\times L^{2}_{q}\times L^{2}_{q}$ and locally Lipschitz with respect to the variable $(\rho^{q},u^{q},a^{q},v^{q})$ for every $q\geq 1$. It follows from the Cauchy-Lipschitz theorem (refer to \cite{bahourietal11}) that there exists a maximal time $T_{q}^{*}>0$ such that the problem \eqref{apENS} has a unique solution $(\rho^{q},u^{q},a^{q},v^{q})$ in the space $\mathcal{C}([0,T_{q}^{*});L^{2}_{q})$.

Define the maximal time
\begin{equation*}
\begin{aligned}
T_{q}:=\sup\Big\{t\geq 0~|~\|(\rho^{q},a^{q})\|_{\widetilde{L}^{\infty}_{t}(\dot{B}^{\frac{d}{2}-1}_{2,1})}+\|(u^{q},v^{q})\|_{\widetilde{L}^{\infty}_{t}(\dot{B}^{\frac{d}{2}-1}_{2,1})}+\tau \|u^{q}\|_{\widetilde{L}^{\infty}_{t}(\dot{B}^{\frac{d}{2}+1}_{2,1})}\leq \beta\Big\}.
\end{aligned}
\end{equation*}
It is easy to verify that $0<T_{q}\leq T_{q}^{*}$. We claim $T_{q}=T_{q}^{*}$ and prove it by a contradiction argument. Assuming that $T_{q}<T_{q}^{*}$, one can observe that the $L^{2}$ orthogonal projector $\dot{\mathbb{E}}_{q}$ has no effect on the energy estimates in the proof of Proposition \ref{proppriori}. Denote $\mathcal{D}^{q}(t)=\mathcal{D}(u^{q},a^{q},v^{q})$ by \eqref{priorieg}, and let $\mathcal{X}_0$ be given by \eqref{priorieg0}. It follows from Proposition \ref{proppriori} that
\begin{align}
\|\rho^{q}\|_{\widetilde{L}^{\infty}_{t}(\dot{B}^{\frac{d}{2}-1}_{2,1})}\leq C_{0}\|\rho_0\|_{\dot{B}^{\frac{d}{2}-1}_{2,1}} ,\quad \|\rho^{q}\|_{\widetilde{L}^{\infty}_{t}(\dot{B}^{\frac{d}{2}}_{2,1})}\leq C_{0}\|\rho_0\|_{\dot{B}^{\frac{d}{2}}_{2,1}},
\end{align}
and
\begin{align}
&\|(u^{q},v^{q})\|_{\widetilde{L}^{\infty}_{t}(\dot{B}^{\frac{d}{2}-1}_{2,1})}+\tau\|u^{q}\|_{\widetilde{L}^{\infty}_{t}(\dot{B}^{\frac{d}{2}+1}_{2,1})}+\|a^{q}\|_{\widetilde{L}^{\infty}_{t}(\dot{B}^{\frac{d}{2}-1}_{2,1}\cap\dot{B}^{\frac{d}{2}}_{2,1})}+\mathcal{D}^{q}(t)\leq C_{0}\mathcal{X}_{0}.
\end{align}
Then, we obtain for any $t\in(0,T_{q})$ that
\begin{equation}\label{eschi}
\|(\rho^{q},a^{q})\|_{\widetilde{L}^{\infty}_{t}(\dot{B}^{\frac{d}{2}-1}_{2,1})}+\|(u^{q},v^{q})\|_{\widetilde{L}^{\infty}_{t}(\dot{B}^{\frac{d}{2}-1}_{2,1})}+\tau \|u^{q}\|_{\widetilde{L}^{\infty}_{t}(\dot{B}^{\frac{d}{2}+1}_{2,1})}\leq \frac{1}{2}\beta,
\end{equation}
provided that
\begin{equation}\notag
\|\rho_0\|_{\dot{B}^{\frac{d}{2}-1}_{2,1}\cap \dot{B}^{\frac{d}{2}}_{2,1}}+\mathcal{X}_{0}\leq\delta_{q}:=\min\Big\{1,\frac{\beta}{4C_0}\Big\}.
\end{equation}

By virtue of \eqref{eschi} and the time continuity property, $T_{q}$ is not the maximal time, which contradicts the definition of $T_{q}$. Thus, one has $T_{q}=T_{q}^{*}$. Furthermore, using the uniform estimate \eqref{eschi}, the Cauchy-Lipschitz theorem and a similar contradiction argument on the maximal time $T_{q}^{*}$, we can further deduce $T_{q}^{*}=\infty$, which implies that $(\rho^{q},u^{q},a^{q},v^{q})$ is a global solution to the Cauchy problem \eqref{apENS}. 

It follows from the equations \eqref{apENS} and the uniform estimates \eqref{eschi} that the time derivatives $(\partial_{t}\rho^{q},\partial_{t}u^{q},\partial_{t}a^{q},\partial_{t}v^{q})$ can be uniformly bounded in a suitable sense with respect to $q\geq 1$. By a standard compactness process \cite[pages 442--444]{bahourietal11}, one can get a limit $(\rho,u,a,v)$ of a subsequence $(\rho^{q_{k}},u^{q_{k}},a^{q_{k}},v^{q_{k}})$ as $q_{k}\rightarrow\infty$, such that the approximate solutions to \eqref{apENS} with $q=q_k$ converge to a solution of the Cauchy problem \eqref{pENS1} in the sense of distributions. Therefore, $(\rho,u,n,v)$ with $n=1+a$ is a global solution to the problem \eqref{pENS}-\eqref{indt}. In view of the uniform estimates \eqref{eschi} and the Fatou property, $(\rho,u,n,v)$ with $\rho>0$ and $n\geq \frac{1}{2}$ is indeed a global strong solution to the problem \eqref{pENS}-\eqref{indt} and satisfies the properties \eqref{soluspce1}-\eqref{uvcon}, where we have noticed that
\begin{equation*}
\begin{aligned}
\|u-v\|_{L^{1}_{t}(\dot{B}^{\frac{d}{2}}_{2,1})}&\lesssim \|u-v\|_{L^{1}_{t}(\dot{B}^{\frac{d}{2}}_{2,1})}^{\ell}+\Big(\|u-v\|_{L^{1}_{t}(\dot{B}^{\frac{d}{2}-1}_{2,1})}^{h}\Big)^{\frac{1}{2}}\Big(\|(u,v)\|_{L^{1}_{t}(\dot{B}^{\frac{d}{2}+1}_{2,1})}^{h}\Big)^{\frac{1}{2}}\lesssim \sqrt{\tau}\mathcal{X}_{0}.
\end{aligned}
\end{equation*}

To complete the proof of Theorem \ref{theorem1}, it remains to show the uniqueness of the solution obtained above. Here we only deal with the case $d\geq 3$, and one can refer to \cite{bahourietal11,danchinCPDE07} with a suitable modification in the two-dimensional case. Note that the uniqueness holds for arbitrarily large initial data. Assume that $(\rho_{1},u_{1},n_{1},v_{1})$ ($n_1=1+a_1$) and $(\rho_{2},u_{2},n_{2},v_{2})$ ($n_2=1+a_2$) are two solutions of the Cauchy problem \eqref{pENS}-\eqref{indt} satisfying \eqref{soluspce1} with the same initial data on $[0,T_{0}]\times\mathbb{R}^d$ for a given time $T_{0}>0$. We shall prove that these two solutions coincide, and hence we aim to analyze  
\begin{equation*}(\widetilde{\rho},\widetilde{u},\widetilde{a},\widetilde{v}):=(\rho_{1}-\rho_{2},u_{1}-u_{2},a_{1}-a_{2},v_{1}-v_{2}),
\end{equation*}
by considering the following system
\begin{equation}\label{difpENS}
\left\{
\begin{aligned}
& \partial_{t}\widetilde{\rho}+u_{1}\cdot\nabla\widetilde{\rho}=-\widetilde{u}\cdot\nabla\rho_{2}-\widetilde{\rho}\operatorname{div}u_{1}-\rho_{2}\operatorname{div}\widetilde{u},\\
& \partial_{t}\widetilde{u}+u_{1}\cdot\nabla\widetilde{u}+\frac{1}{\tau}\widetilde{u}=\frac{1}{\tau}\widetilde{v}-\widetilde{u}\cdot\nabla u_{2},\\
&\partial_{t}\widetilde{a}+v_{1}\cdot\nabla\widetilde{a}=-\widetilde{v}\cdot\nabla a_{2}-\widetilde{a}\operatorname{div}v_{1}-(1+a_{2})\operatorname{div}\widetilde{v},\\
&\partial_{t}\widetilde{v}+v_{1}\cdot\nabla\widetilde{v}+\widetilde{v}\cdot\nabla v_{2}-(1+f(a_{1}))\big(\mu\Delta\widetilde{v}+(\mu+\lambda)\nabla\operatorname{div}\widetilde{v}\big)=\widetilde{F}_{1}+\widetilde{F}_{2},
\end{aligned}
\right.
\end{equation}
where $\widetilde{F}_{1}$ and $\widetilde{F}_{2}$ satisfy
\begin{align*}
&\widetilde{F}_{1}:=-(1-g(a_{1}))\nabla\widetilde{a}+\big(g(a_{1})-g(a_{2})\big)\nabla a_{2}+\big(f(a_{1})-f(a_{2})\big)\big(\mu\Delta v_{2}+(\mu+\lambda)\nabla\operatorname{div}v_{2}\big),\\
&\widetilde{F}_{2}:=\frac{1}{\tau}(1+f(a_{1}))\rho_{1}(\widetilde{u}-\widetilde{v})+\frac{1}{\tau}(1+f(a_{1}))\widetilde{\rho}(u_{2}-v_{2})+\frac{1}{\tau}\big(f(a_{1})-f(a_{2})\big)\rho_{2}(u_{2}-v_{2}).
\end{align*}

First, applying the regularity estimates in Lemma \ref{estrans} to $\eqref{difpENS}_{1}$ and $\eqref{difpENS}_{3}$, respectively, one can verify that
\begin{equation*}
\begin{aligned}			\|\widetilde{\rho}\|_{\dot{B}^{\frac{d}{2}-2}_{2,1}\cap\dot{B}^{\frac{d}{2}-1}_{2,1}}&\lesssim \exp\Big(C\|u_{1}\|_{L^{1}_{t}(\dot{B}^{\frac{d}{2}+1}_{2,1})}\Big)\Big(\int_{0}^{t}\|\widetilde{\rho}\|_{\dot{B}^{\frac{d}{2}-2}_{2,1}\cap\dot{B}^{\frac{d}{2}-1}_{2,1}}\|u_{1}\|_{\dot{B}^{\frac{d}{2}+1}_{2,1}}{\rm d}s+\int_{0}^{t}\|\widetilde{u}\|_{\dot{B}^{\frac{d}{2}}_{2,1}}{\rm d}s\Big),
\end{aligned}
\end{equation*}
and 
\begin{equation*}
\begin{aligned}
\|\widetilde{a}\|_{\dot{B}^{\frac{d}{2}-1}_{2,1}}&\lesssim \exp\Big(C\|v_{1}\|_{L^{1}_{t}(\dot{B}^{\frac{d}{2}+1}_{2,1})}\Big)\Big(\int_{0}^{t}\|\widetilde{a}\|_{\dot{B}^{\frac{d}{2}-1}_{2,1}}\|v_{1}\|_{\dot{B}^{\frac{d}{2}+1}_{2,1}}{\rm d}s+\int_{0}^{t}\|\widetilde{v}\|_{\dot{B}^{\frac{d}{2}}_{2,1}}{\rm d}s\Big).
\end{aligned}
\end{equation*}
These, together with the Gr\"{o}nwall inequality, give rise to
\begin{align}\label{difrho}
&\|\widetilde{\rho}\|_{\dot{B}^{\frac{d}{2}-2}_{2,1}\cap\dot{B}^{\frac{d}{2}-1}_{2,1}}+\|\widetilde{a}\|_{\dot{B}^{\frac{d}{2}-1}_{2,1}}\leq C_{T_{0}}\int_{0}^{t}\|(\widetilde{u},\widetilde{v})\|_{\dot{B}^{\frac{d}{2}}_{2,1}}{\rm d}s,
\end{align}
for $t\in[0,T_{0}]$, where $C_{T_{0}}>0$ is a positive constant depending on $T_{0}$ and the regularities of two solutions.

Next, we apply the operator $\dot{\Delta}_{j}$ to $\eqref{difpENS}_{2}$ and take the $L^{2}$ inner product of the resulting equation with $\dot{\Delta}_{j}\widetilde{u}$ to arrive at
\begin{equation*}
\begin{aligned}
&\frac{1}{2}\frac{{\rm{d}}}{{\rm{d}}t}\|\dot{\Delta}_{j}\widetilde{u}\|_{L^{2}}^{2}+\frac{1}{\tau}\|\dot{\Delta}_{j}\widetilde{u}\|_{L^{2}}^{2}\\
&\quad\leq\Big(\frac{1}{\tau}\|\dot{\Delta}_{j}\widetilde{v}\|_{L^{2}}+\frac{1}{2}\|{\rm div}\,u_{1}\|_{L^\infty}\|\dot{\Delta}_{j}\widetilde{u}\|_{L^{2}}+\|[u_{1}\cdot\nabla,\dot{\Delta}_{j}]\widetilde{u}\|_{L^{2}}+\|\dot{\Delta}_{j}(\widetilde{u}\cdot\nabla u_{2})\|_{L^{2}}\Big)\|\dot{\Delta}_{j}\widetilde{u}\|_{L^{2}},
\end{aligned}
\end{equation*}
which, combined with \eqref{uv2},  Lemma \ref{commutator} and the Gr\"{o}nwall inequality, leads to
\begin{align}\label{difu}
&\tau\|\widetilde{u}\|_{\dot{B}^{\frac{d}{2}-2}_{2,1}\cap \dot{B}^{\frac{d}{2}}_{2,1}}+\|\widetilde{u}\|_{L^{1}_{t}(\dot{B}^{\frac{d}{2}-2}_{2,1}\cap \dot{B}^{\frac{d}{2}}_{2,1})}\leq C_{T_{0}}\int_{0}^{t}\|\widetilde{v}\|_{\dot{B}^{\frac{d}{2}-2}_{2,1}\cap \dot{B}^{\frac{d}{2}}_{2,1}}{\rm d}s,\quad \text{for}~~t\in[0,T_{0}].
\end{align}

Then, applying the regularity estimates for the parabolic equations with variable coefficients (see \cite[Proposition 6]{danchinCPDE07}), we have
\begin{equation}\label{difv}
\begin{aligned}
&\|\widetilde{v}\|_{\dot{B}^{\frac{d}{2}-2}_{2,1}}+\|\widetilde{v}\|_{L^{1}_{t}(\dot{B}^{\frac{d}{2}}_{2,1})}\\
&\quad\leq\exp\Big(C\|(v_{1},v_{2})\|_{L^{1}_{t}(\dot{B}^{\frac{d}{2}+1}_{2,1})}+C\|f(a_{1})\|_{\widetilde{L}^{2}_{t}(\dot{B}^{\frac{d}{2}}_{2,1})}^{2}\Big)\int_{0}^{t}\|(\widetilde{F}_{1},\widetilde{F}_{2})\|_{\dot{B}_{2,1}^{\frac{d}{2}-2}} {\rm d}s\\
&\quad\leq C_{\tau,T_{0}}\int_{0}^{t}\|(\widetilde{u},\widetilde{v})\|_{\dot{B}_{2,1}^{\frac{d}{2}-2}}\Big(1+\|v_{2}\|_{\dot{B}_{2,1}^{\frac{d}{2}+1}}+\|(u_{2},v_{2})\|_{\dot{B}_{2,1}^{\frac{d}{2}}}\Big){\rm d}s+C_{\tau}\int_0^t  \|(\widetilde{\rho},\widetilde{a})\|_{\dot{B}^{\frac{d}{2}-1}_{2,1}}\,{\rm d}s.
\end{aligned}
\end{equation}
Hence, combining  the above estimate \eqref{difv} with \eqref{difrho} and \eqref{difu} and making use of the Gr\"{o}nwall inequality, we finally prove the uniqueness and complete the proof of Theorem \ref{theorem1}.

\section{Large-friction limit of the Euler-NS system to the DF system}\label{sectionlimit}

This section is devoted to the error estimates between global solutions for the Euler-NS system \eqref{pENS} and the DF system \eqref{DF}.

\subsection{Global existence of the DF system}\label{sectionlimit:DFweak}
To begin with, we construct a global solution of the DF system as the large-friction limit of approximate solutions of the Euler-NS system. For $\tau\in(0,1)$, we define the approximate initial data
\begin{equation}\label{app:data}
\left\{
\begin{aligned}
&\rho^\tau_0(x)=\rho_0(x)+\tau e^{-|x|^2}>0,\quad  u_0^\tau(x)=\mathcal{F}^{-1}\Big( \mathbf{I}_{|\xi|\leq \frac{1}{\sqrt{\tau}}}\mathcal{F}(v_0)\Big)(x)\\
& n_0^\tau(x)=n_0(x),\quad v_0^\tau(x)=v_0(x).
\end{aligned}
\right.
\end{equation}
Then one has 
\begin{align}
&\|\rho^\tau_0\|_{\dot{B}^{\frac{d}{2}-1}_{2,1}}\lesssim \|\rho_0\|_{\dot{B}^{\frac{d}{2}-1}_{2,1}}+\tau,\quad \|\rho^\tau_0\|_{\dot{B}^{\frac{d}{2}}_{2,1}}\lesssim \|\rho_0\|_{\dot{B}^{\frac{d}{2}}_{2,1}}+\tau, \label{data1}\\
&\|n_0^\tau-1\|_{\dot{B}^{\frac{d}{2}-1}_{2,1}\cap\dot{B}^{\frac{d}{2}}_{2,1}}\lesssim \|n_0-1\|_{\dot{B}^{\frac{d}{2}-1}_{2,1}\cap\dot{B}^{\frac{d}{2}}_{2,1}},\quad  \|(u_0^\tau, v_0^\tau)\|_{\dot{B}^{\frac{d}{2}-1}_{2,1}}\lesssim  \|v_0\|_{\dot{B}^{\frac{d}{2}-1}_{2,1}},\label{data2}
\end{align}
and
\begin{align}
\tau\|u_0^\tau\|_{\dot{B}^{\frac{d}{2}+1}_{2,1}}\lesssim \tau\sum_{2^j\lesssim \frac{1}{\sqrt{\tau}}} 2^{(\frac{d}{2}+1)j}\|\dot{\Delta}_j u_0^\tau\|_{L^2}\lesssim \|u^\tau_0\|_{\dot{B}^{\frac{d}{2}-1}_{2,1}}\lesssim \|v_0\|_{\dot{B}^{\frac{d}{2}-1}_{2,1}}.\label{data3}
\end{align}
Therefore, in view of Theorem \ref{theorem1}, if $(\rho_0,n_0,v_0)$ satisfies \eqref{idcd4} with $\delta_0^*$ sufficiently small and let $\tau\in(0, \tau_0]$ for some parameter $\tau_0\in(0,1)$, we can obtain an approximate sequence $(\rho^\tau, u^\tau, n^\tau, v^\tau)$ that are global solutions to the Cauchy problem \eqref{pENS}-\eqref{indt} subject to the initial data $(\rho^\tau_0,u_0^\tau,n^\tau_0,v^\tau_0)$ given in \eqref{app:data}, satisfying the uniform regularity estimates \eqref{es:rho}-\eqref{uvcon} with respect to $\tau\in (0,\tau_0)$.

According to \eqref{uvcon}, as $\tau\rightarrow 0$, we have
\begin{align}\label{con1}
&u^\tau-v^\tau\rightarrow 0\quad\text{in}\quad L^1(\mathbb{R}_+;\dot{B}^{\frac{d}{2}}_{2,1})\cap L^2(\mathbb{R}_+;\dot{B}^{\frac{d}{2}-1}_{2,1}).
\end{align}
For any fixed time $T>0$, according to  $\eqref{pENS}_1$ and $\eqref{pENS}_3$, it follows from \eqref{es:rho}-\eqref{result1}, \eqref{idcd4}, and \eqref{data1}-\eqref{data3} that
\begin{equation*}
\begin{aligned}
&\|\partial_t \rho^\tau\|_{L^{\infty}_T(\dot{B}^{\frac{d}{2}-2}_{2,1})}\lesssim \|\rho^\tau u^\tau\|_{L^{\infty}_T(\dot{B}^{\frac{d}{2}-1}_{2,1})}\lesssim \|\rho^\tau\|_{L^{\infty}_T(\dot{B}^{\frac{d}{2}}_{2,1})} \|u^\tau\|_{L^{\infty}_T(\dot{B}^{\frac{d}{2}-1}_{2,1})}\lesssim 1,\\
&\|\partial_t n^\tau\|_{L^{\infty}_T(\dot{B}^{\frac{d}{2}-2}_{2,1})}\lesssim (1+\|n^\tau-1\|_{L^{\infty}_T(\dot{B}^{\frac{d}{2}}_{2,1})})\|v^\tau\|_{L^{\infty}_T(\dot{B}^{\frac{d}{2}-1}_{2,1})}\lesssim 1.
\end{aligned}
\end{equation*}
Similarly, we derive from $\eqref{pENS}_2$, $\eqref{pENS}_4$, \eqref{es:rho}-\eqref{uvcon}, and \eqref{data1}-\eqref{data3} that
\begin{align*}
&\begin{aligned}
\|\partial_t u^\tau\|_{L^1_T(\dot{B}^{\frac{d}{2}}_{2,1})}^{\ell}&\lesssim \|u^\tau\cdot \nabla u^\tau \|_{L^1_T(\dot{B}^{\frac{d}{2}-1}_{2,1})}^{\ell}+\frac{1}{\tau} \|u^\tau-v^\tau\|_{L^1_T(\dot{B}^{\frac{d}{2}}_{2,1})}^{\ell}\\
&\lesssim \|u^\tau\|_{L^{\infty}_T(\dot{B}^{\frac{d}{2}-1}_{2,1})}\| u^\tau\|_{L^1_T(\dot{B}^{\frac{d}{2}+1}_{2,1})}+\frac{1}{\tau} \|u^\tau-v^\tau\|_{L^1_T(\dot{B}^{\frac{d}{2}}_{2,1})}^{\ell}\lesssim 1, 
\end{aligned}\\
&\|\partial_t u^\tau\|_{L^1_T(\dot{B}^{\frac{d}{2}-1}_{2,1})}^{h}\lesssim \|u^\tau\|_{L^{\infty}_T(\dot{B}^{\frac{d}{2}-1}_{2,1})}\| u^\tau\|_{L^1_T(\dot{B}^{\frac{d}{2}+1}_{2,1})}+\frac{1}{\tau} \|u^\tau-v^\tau\|_{L^1_T(\dot{B}^{\frac{d}{2}-1}_{2,1})}^{h}\lesssim 1,
\end{align*}
and
\begin{equation*}
\begin{aligned}
\|\partial_t v^\tau\|_{L^1_T(\dot{B}^{\frac{d}{2}-1}_{2,1})}&\lesssim (1+\|v^\tau\|_{L^{\infty}_T(\dot{B}^{\frac{d}{2}-1}_{2,1})})\| v^\tau\|_{L^1_T(\dot{B}^{\frac{d}{2}+1}_{2,1})}+T\|n^\tau-1\|_{L^{\infty}_T(\dot{B}^{\frac{d}{2}}_{2,1})}\\
&\quad+\|\rho^\tau\|_{L^{\infty}_T(\dot{B}^{\frac{d}{2}-1}_{2,1})}\frac{1}{\tau}\|u^\tau-v^\tau\|_{L^1_T(\dot{B}^{\frac{d}{2}}_{2,1})}^{\ell}+\|\rho^\tau\|_{L^{\infty}_T(\dot{B}^{\frac{d}{2}}_{2,1})}\frac{1}{\tau}\|u^\tau-v^\tau\|_{L^1_T(\dot{B}^{\frac{d}{2}-1}_{2,1})}^{h}\\
&\lesssim 1+T.
\end{aligned}
\end{equation*}
In view of a standard compactness argument based on the Aubin-Lions lemma (cf. \cite{simon87}), we conclude that there exists a limit $(\rho,u,n,v)$ such that, up to a subsequence,
\begin{equation}\label{ccccc}
\left\{
\begin{aligned}
&(\rho^\tau, n^\tau-1)\rightarrow (\rho,n-1) \quad &&\text{strongly in}~~\mathcal{C}([0,T];H_{\rm{loc}}^{\frac{d}{2}-\eta}),\\
&(u^\tau,v^\tau)\rightarrow (u,v) &&\text{strongly in}~~ L^2(0,T;H_{\rm{loc}}^{\frac{d}{2}-\eta}),
\end{aligned}
\right.
\end{equation}
for any $\eta\in(0,1)$ and $T>0$. Consequently, one can show the convergence of the equations $\eqref{pENS}_1$ and $\eqref{pENS}_3$ to $\eqref{DF}_1$ and $\eqref{DF}_2$, respectively, in the distributional sense as $\tau\rightarrow 0$. Combining \eqref{con1} and \eqref{ccccc}, we have $u=v$.  Adding $\eqref{pENS}_2$ and  $\eqref{pENS}_4$ together, we have
\begin{equation*}
\begin{aligned}
&\partial_t(\rho^\tau u^\tau+n^\tau v^\tau)+\operatorname{div} (\rho^\tau u^\tau\otimes u^\tau+n^\tau v^\tau\otimes v^\tau)+\nabla P(n^\tau)=\mu \Delta v^\tau+(\mu+\lambda) \nabla\operatorname{div} v^\tau,
\end{aligned}
\end{equation*}
which converges to the momentum equation $\eqref{DF}_3$ in the distributional sense as $\tau\rightarrow0$. 

Therefore, $(\rho,n,v)$ is indeed a global solution to the Cauchy problem \eqref{DF}-\eqref{DFindt} satisfying \eqref{soluspce2}. Due to the uniform estimates and Fatou's property (see \cite{bahourietal11}), we discover that $(\rho,n,v)$ satisfies \eqref{rhoDF} and \eqref{result5}, and one can prove the uniqueness by similar arguments as in Subsection \ref{subsection32}; the details are omitted. The proof of Corollary \ref{corollary1} is complete.

\subsection{Global-in-time convergence rate}

Let $(\rho^{\tau},u^{\tau},n^{\tau},v^{\tau})$ be the solution to the Euler-NS system \eqref{pENS} with the initial data $(\rho_{0}^{\tau},u_{0}^{\tau},n_{0}^{\tau},v_{0}^{\tau})$, and $(\rho,n,v)$ be the solution to the DF system \eqref{DF} with the initial data $(\rho_{0},n_{0},v_{0})$. We first present the equations of the error variables, defined by
\begin{equation*}
(\widetilde{\rho}^{\tau},\widetilde{u}^{\tau},\widetilde{n}^{\tau},\widetilde{v}^{\tau}):=(\rho^{\tau}-\rho,u^{\tau}-v,n^{\tau}-n,v^{\tau}-v).
\end{equation*}
Then, the convergence estimates with respect to the friction parameter $\tau$ are established, and consequently, we give the proof of Theorem \ref{theorem3}.

\subsubsection{Reformulations}\label{section61}
\noindent \textbf{Effective mixed velocity}. Due to the appearance of the drag force $\frac{1}{\tau}\rho^{\tau}(u^{\tau}-v^{\tau})$, one cannot directly estimate the error equations of $(\widetilde{n}^{\tau},\widetilde{v}^{\tau})$ directly. To this end, we add $\eqref{pENS}_{2}$ and $\eqref{pENS}_{4}$ together to get \eqref{rhousum}, i.e.,
\begin{equation}\tag{\ref{rhousum}}
\begin{aligned}
\rho^{\tau} \partial_{t}u^{\tau}+n^{\tau}\partial_{t}v^{\tau}+\rho^{\tau} u^{\tau}\cdot\nabla u^{\tau}+n^{\tau} v^{\tau}\cdot \nabla v^{\tau}+\nabla P(n^{\tau})= \mu\Delta v^{\tau}+(\mu+\lambda)\nabla\operatorname{div} v^{\tau}.
\end{aligned}
\end{equation}
As analyzed before, the mixed time derivatives and convective terms in \eqref{rhousum} cause the main challenges. Our key observation is the introduction of the {\emph{effective mixed velocity}}:
\begin{equation*}
V^{\tau}:=\frac{\rho^{\tau}}{\rho^{\tau}+n^{\tau}}u^{\tau}+\frac{n^{\tau}}{\rho^{\tau}+n^{\tau}}v^{\tau},
\end{equation*}
which satisfies
\begin{align*}
\frac{\rho^{\tau}}{\rho^{\tau}+n^{\tau}}\partial_{t}u^{\tau}+\frac{n^{\tau}}{\rho^{\tau}+n^{\tau}}\partial_{t}v^{\tau}&=\partial_t v^\tau+\frac{\rho^{\tau}}{\rho^{\tau}+n^{\tau}}\partial_{t}(u^{\tau}-v^{\tau})=\partial_{t} V^\tau-\partial_{t} \Big( \frac{\rho^{\tau}}{\rho^{\tau}+n^{\tau}}\Big) (u^{\tau}-v^{\tau}),
\end{align*}
and
\begin{equation}\nonumber
\begin{aligned}
&\frac{\rho^{\tau}}{\rho^{\tau}+n^{\tau}}u^{\tau}\cdot\nabla u^{\tau}+\frac{n^{\tau}}{\rho^{\tau}+n^{\tau}}v^{\tau}\cdot\nabla v^{\tau}\\
&\quad=v^\tau\cdot \nabla v^\tau+\frac{\rho^{\tau}}{\rho^{\tau}+n^{\tau}}(u^\tau \cdot\nabla u^\tau-v^\tau\cdot \nabla v^\tau)\\
&\quad=v^\tau\cdot \nabla V^\tau+\frac{\rho^{\tau}}{\rho^{\tau}+n^{\tau}}(u^\tau-v^\tau)\cdot \nabla u^{\tau}-v^{\tau}\cdot \nabla\Big(\frac{\rho^{\tau}}{\rho^{\tau}+n^{\tau}}\Big) (u^{\tau}-v^{\tau}).
\end{aligned}
\end{equation}
Therefore, we are able to rewrite \eqref{rhousum} in terms of $V^\tau$ as  
\begin{equation}\label{rhousum1}
\begin{aligned}
&\partial_{t}V^{\tau}+v^\tau\cdot \nabla V^\tau+\frac{1}{\rho^\tau+n^\tau}\nabla P(n^{\tau})\\
&\quad= \frac{\mu}{\rho^\tau+n^\tau}\Delta v^{\tau}+\frac{\mu+\lambda}{\rho^\tau+n^\tau}\nabla\operatorname{div} v^{\tau}\\
&\qquad+\partial_{t} \Big( \frac{\rho^{\tau}}{\rho^{\tau}+n^{\tau}}\Big) (u^{\tau}-v^{\tau})-\frac{\rho^{\tau}}{\rho^{\tau}+n^{\tau}}(u^\tau-v^\tau)\cdot \nabla u^{\tau}+v^{\tau}\cdot \nabla\Big(\frac{\rho^{\tau}}{\rho^{\tau}+n^{\tau}}\Big) (u^{\tau}-v^{\tau}).
\end{aligned}
\end{equation}
Note that the sources involve $u^{\tau}-v^{\tau}$, in which the $L^1_t(\dot{B}^{\frac{d}{2}}_{2,1})$-rate of order $\sqrt{\tau}$ is crucial to establish the $\mathcal{O}(\sqrt{\tau})$ error estimates.

\noindent \textbf{Error equations for $(\widetilde{n}^{\tau},\widetilde{V}^{\tau})$}. 
The above observations motivate us to introduce the {\emph{effective mixed error}} $\widetilde{V}^{\tau}$ as
\begin{align*}
\widetilde{V}^\tau:=\frac{\rho^{\tau}}{\rho^{\tau}+n^{\tau}}\widetilde{u}^\tau+\frac{n^{\tau}}{\rho^{\tau}+n^{\tau}}\widetilde{v}^\tau=V^{\tau}-v,
\end{align*}
and then $(\widetilde{u}^\tau,\widetilde{v}^\tau)$ can be represented by $(\widetilde{V}^\tau,u^\tau-v^\tau)$ as in \eqref{erruvV}. It follows from \eqref{pENS}, \eqref{DF}, and \eqref{rhousum1} that $(\widetilde{n}^{\tau},\widetilde{V}^{\tau})$ satisfies a coupled hyperbolic-parabolic system
\begin{equation}\label{errnVe}
\left\{
\begin{aligned}
&\partial_{t}\widetilde{n}^\tau+v^{\tau}\cdot\nabla\widetilde{n}^\tau+ \operatorname{div} \widetilde{V}^\tau=\widetilde{F}_{1}+\widetilde{F}_{2},\\
&\partial_{t}\widetilde{V}^\tau+v^{\tau}\cdot \nabla \widetilde{V}^\tau+ \nabla \widetilde{n}^\tau- \mu\Delta \widetilde{V}^\tau-(\mu+\lambda)\nabla\operatorname{div} \widetilde{V}^\tau =\widetilde{F}_{3}+\widetilde{F}_{4},
\end{aligned}
\right.
\end{equation}
where the nonlinear terms $\widetilde{F}_{i}$ ($i=1,2,3,4$) are given by
\begin{equation}\label{F1F2F3F4}
\left\{
\begin{aligned}
\widetilde{F}_{1}&:=-\widetilde{V}^{\tau}\cdot\nabla n-(n-1)\operatorname{div} \widetilde{V}^{\tau}-\widetilde{n}^{\tau} \operatorname{div} v^{\tau},\\
\widetilde{F}_{2}&:= \frac{\rho^{\tau}}{\rho^{\tau}+n^{\tau}} (u^\tau-v^\tau)\cdot \nabla n+n\operatorname{div}\Big(\frac{\rho^{\tau}}{\rho^{\tau}+n^{\tau}}(u^{\tau}-v^{\tau})\Big),\\
\widetilde{F}_{3}&:=-\widetilde{V}^\tau\cdot\nabla v-\Big(\frac{P'(n^{\tau})}{\rho^{\tau}+n^{\tau}}-1\Big)\nabla\widetilde{n}^{\tau}-\Big(\frac{P'(n^{\tau})}{\rho^{\tau}+n^{\tau}}-\frac{P'(n)}{\rho+n}\Big) \nabla n\\
&\qquad+\Big(\frac{1}{\rho^\tau+n^\tau}-1\Big)(\mu\Delta\widetilde{V}^\tau+(\mu+\lambda)\nabla\operatorname{div}\widetilde{V}^\tau)\\
&\qquad+\Big(\frac{1}{\rho^{\tau}+n^{\tau}}-\frac{1}{\rho+n}\Big) (\mu\Delta v+(\mu+\lambda)\nabla\operatorname{div}v),\\
\widetilde{F}_{4}&:=\frac{\rho^{\tau}}{\rho^{\tau}+n^{\tau}} (u^\tau-v^\tau)\cdot \nabla v+\partial_{t} \Big(\frac{\rho^{\tau}}{\rho^{\tau}+n^{\tau}}\Big) (u^{\tau}-v^{\tau})\\
&\qquad-\frac{\rho^{\tau}}{\rho^{\tau}+n^{\tau}}(u^\tau-v^\tau)\cdot \nabla u^{\tau}+v^{\tau}\cdot \nabla\Big(\frac{\rho^{\tau}}{\rho^{\tau}+n^{\tau}}\Big) (u^{\tau}-v^{\tau})\\
&\qquad-\frac{\mu}{\rho^{\tau}+n^{\tau}}\Delta\Big( \frac{\rho^{\tau}}{\rho^{\tau}+n^{\tau}}(u^{\tau}-v^{\tau})\Big)-\frac{\mu+\lambda}{\rho^{\tau}+n^{\tau}}\nabla\operatorname{div} \Big( \frac{\rho^{\tau}}{\rho^{\tau}+n^{\tau}}(u^{\tau}-v^{\tau})\Big).
\end{aligned}
\right.
\end{equation}
Given the bounds for the relative velocity $u^{\tau}-v^{\tau}$, we will treat \eqref{errnVe} as a compressible Navier--Stokes type system with $\mathcal{O}(\sqrt{\tau})$ source terms at the $\dot{B}^{\frac{d}{2}-2}_{2,1}$ regularity level for both low and high frequencies. One can obtain the desired estimates for $(\widetilde{u}^{\tau},\widetilde{v}^{\tau})$ via the bounds for $\widetilde{V}^{\tau}$ and $u^{\tau}-v^{\tau}$.

\noindent \textbf{Error equations for $\widetilde{\rho}^{\tau}$}. It follows from $\eqref{pENS}_{1}$, $\eqref{DF}_{1}$, and $\eqref{erruvV}_{1}$ that
\begin{align}\label{errrho}	\partial_{t}\widetilde{\rho}^{\tau}+u^{\tau}\cdot\nabla\widetilde{\rho}^{\tau}=-\widetilde{\rho}^{\tau} \operatorname{div} u^{\tau}-\operatorname{div} (\rho\widetilde{V}^{\tau})+\operatorname{div} \Big(\frac{\rho n^{\tau}}{\rho^{\tau}+n^{\tau}}(u^\tau-v^\tau)\Big).
\end{align}
We mention that the estimate of $\widetilde{\rho}^{\tau}$ is based on the standard transport estimates for \eqref{errrho} and the $L^{1}$ time integrability of $\widetilde{V}^{\tau}$ and $u^\tau-v^\tau$.

\subsubsection{Estimates for \texorpdfstring{$(\widetilde{n}^{\tau},\widetilde{V}^{\tau})$}{(n,V)}}
Next, we establish the convergence estimates of $(\widetilde{n}^{\tau},\widetilde{V}^{\tau})$ for the system \eqref{errnVe}, where the effective unknown $\widetilde{V}^{\tau}$ is defined in \eqref{errV0}. 

\begin{lemma}\label{lemmanV}
Let $(\rho^{\tau},u^{\tau},n^{\tau},v^{\tau})$ and $(\rho,n,v)$ be the global solutions to \eqref{pENS} and \eqref{DF} obtained in Theorem \ref{theorem1} and Corollary \ref{corollary1}, respectively. Then, under the condition \eqref{idcd5}, it holds that
\begin{equation}\label{deltanV}
\begin{aligned}
&\|\widetilde{n}^{\tau}\|_{\widetilde{L}^{\infty}_{t}(\dot{B}^{\frac{d}{2}-2}_{2,1}\cap\dot{B}^{\frac{d}{2}-1}_{2,1} )}+\|\widetilde{V}^{\tau}\|_{\widetilde{L}^{\infty}_{t}(\dot{B}^{\frac{d}{2}-2}_{2,1})}+\|\widetilde{n}^{\tau}\|_{L^{1}_{t}(\dot{B}^{\frac{d}{2}}_{2,1})}^{\ell}+\|\widetilde{n}^{\tau}\|_{L^{1}_{t}(\dot{B}^{\frac{d}{2}-1}_{2,1})}^{h}+\|\widetilde{V}^{\tau}\|_{L^{1}_{t}(\dot{B}^{\frac{d}{2}}_{2,1})}\\
&\quad\lesssim \|(\widetilde{n}^{\tau},\nabla \widetilde{n}^{\tau}, \widetilde{V}^{\tau})(0)\|_{\dot{B}^{\frac{d}{2}-2}_{2,1}}+\delta_{0}^{*}\|\widetilde{\rho}^{\tau}\|_{\widetilde{L}^{\infty}_{t}(\dot{B}^{\frac{d}{2}-2}_{2,1}\cap \dot{B}^{\frac{d}{2}-1}_{2,1})}+\sqrt{\tau}.
\end{aligned}
\end{equation}
\end{lemma}

\begin{proof}
Arguing as in the proof of Lemma \ref{lemav}, we can verify that
\begin{equation}\nonumber
\begin{aligned}
&\|\dot{\Delta}_{j}(\widetilde{n}^{\tau},\nabla \widetilde{n}^{\tau}, \widetilde{V}^{\tau})\|_{L^{\infty}_{t}(L^{2})}+\min\{1,2^{2j}\} \|\dot{\Delta}_{j}(\widetilde{n}^{\tau},\nabla \widetilde{n}^{\tau}, \widetilde{V}^{\tau})\|_{L^{1}_{t}(L^{2})}\\[2mm]
&\quad\lesssim \|\dot{\Delta}_{j}(\widetilde{n}^{\tau},\nabla \widetilde{n}^{\tau}, \widetilde{V}^{\tau})(0)\|_{L^{2}}+\|\operatorname{div}v^{\tau}\|_{L_{t}^{1}(L^{\infty})}\|\dot{\Delta}_{j}(\widetilde{n}^{\tau},\nabla \widetilde{n}^{\tau}, \widetilde{V}^{\tau})\|_{L^{\infty}_{t}(L^{2})}+\|[v^{\tau}\cdot\nabla,\dot{\Delta}_{j}]\widetilde{n}^{\tau}\|_{L_{t}^{1}(L^{2})}\\
&\qquad+\|[v^{\tau}\cdot\nabla,\dot{\Delta}_{j}]\widetilde{V}^{\tau}\|_{L_{t}^{1}(L^{2})}+\sum_{k=1}^{d}\|[v^{\tau}\cdot\nabla,\partial_{x_k}\dot{\Delta}_{j}]\widetilde{n}^{\tau}\|_{L_{t}^{1}(L^{2})}+\sum_{i=1}^{4}\|\dot{\Delta}_{j}\widetilde{F}_{i}\|_{L_{t}^{1}(L^{2})}+\sum_{i=1}^{2}\|\dot{\Delta}_{j}\nabla\widetilde{F}_{i}\|_{L_{t}^{1}(L^{2})},
\end{aligned}
\end{equation}
from which and the commutator estimates in Lemma \ref{commutator}, we have
\begin{equation}\label{nVeineq}
\begin{aligned}
&\|(\widetilde{n}^{\tau},\nabla \widetilde{n}^{\tau}, \widetilde{V}^{\tau})\|_{\widetilde{L}^{\infty}_{t}(\dot{B}^{\frac{d}{2}-2}_{2,1})}+\|(\widetilde{n}^{\tau},\nabla \widetilde{n}^{\tau}, \widetilde{V}^{\tau})\|_{L^{1}_{t}(\dot{B}^{\frac{d}{2}}_{2,1})}^{\ell}+\|(\widetilde{n}^{\tau},\nabla \widetilde{n}^{\tau}, \widetilde{V}^{\tau})\|_{L^{1}_{t}(\dot{B}^{\frac{d}{2}-2}_{2,1})}^{h}\\
&\quad\lesssim \|(\widetilde{n}^{\tau},\nabla \widetilde{n}^{\tau}, \widetilde{V}^{\tau})(0)\|_{\dot{B}^{\frac{d}{2}-2}_{2,1}}+\| v^{\tau}\|_{L^{1}_{t}(\dot{B}^{\frac{d}{2}+1}_{2,1})}\|(\widetilde{n}^{\tau},\nabla \widetilde{n}^{\tau}, \widetilde{V}^{\tau})\|_{L^{\infty}_{t}(\dot{B}^{\frac{d}{2}-2}_{2,1})}\\
&\qquad+\|(\widetilde{F}_{1},\widetilde{F}_{2})\|_{L^{1}_{t}(\dot{B}^{\frac{d}{2}-2}_{2,1}\cap\dot{B}^{\frac{d}{2}-1}_{2,1})}+\|(\widetilde{F}_{3},\widetilde{F}_{4})\|_{L^{1}_{t}(\dot{B}^{\frac{d}{2}-2}_{2,1})}.
\end{aligned}
\end{equation}  
We are now in a position to analyze the terms on the right-hand side of \eqref{nVeineq}. Since $-\frac{d}{2}<\frac{d}{2}-2$ for $d\geq 3$, it follows from the uniform estimates \eqref{result1} and \eqref{result5} as well as the product law \eqref{uv2} that
\begin{equation}\label{nVF1}       
\begin{aligned}
&\|\widetilde{F}_{1}\|_{L^{1}_{t}(\dot{B}^{\frac{d}{2}-2}_{2,1}\cap \dot{B}^{\frac{d}{2}-1}_{2,1})}\\
&\quad\lesssim \|\widetilde{V}^{\tau}\|_{L^{1}_{t}(\dot{B}^{\frac{d}{2}}_{2,1})}\|(n-1,\nabla n)\|_{L^{\infty}_{t}(\dot{B}^{\frac{d}{2}-2}_{2,1}\cap\dot{B}^{\frac{d}{2}-1}_{2,1})}+\|\widetilde{n}^{\tau}\|_{L^{\infty}_{t}(\dot{B}^{\frac{d}{2}-2}_{2,1}\cap\dot{B}^{\frac{d}{2}-1}_{2,1})}\|\operatorname{div} v^{\tau}\|_{L^{1}_{t}(\dot{B}^{\frac{d}{2}}_{2,1})}\\
&\quad\lesssim \delta_{0}\|\widetilde{n}^{\tau}\|_{L^{\infty}_{t}(\dot{B}^{\frac{d}{2}-2}_{2,1}\cap\dot{B}^{\frac{d}{2}-1}_{2,1})}+\delta_{0}^{*}\|\widetilde{V}^{\tau}\|_{L^{1}_{t}(\dot{B}^{\frac{d}{2}}_{2,1})}.
\end{aligned}
\end{equation} 
Similarly, owing to Lemma \ref{multicompo} for the multi-component function $\frac{\rho^{\tau}}{\rho^{\tau}+n^{\tau}}$ and the $\mathcal{O}(\sqrt{\tau})$-bound of $u^{\tau}-v^{\tau}$ in $L^{1}_{t}(\dot{B}^{\frac{d}{2}}_{2,1})$ obtained in \eqref{uvcon}, one has
\begin{equation}\label{nVF20}
\begin{aligned}
&\|\widetilde{F}_{2}\|_{L^{1}_{t}(\dot{B}^{\frac{d}{2}-2}_{2,1}\cap \dot{B}^{\frac{d}{2}-1}_{2,1})}\\
&\quad\lesssim \Big\|\frac{\rho^\tau}{\rho^\tau+n^\tau}\Big\|_{L^{\infty}_{t}(\dot{B}^{\frac{d}{2}-1}_{2,1}\cap \dot{B}^{\frac{d}{2}}_{2,1})}\|\nabla n\|_{L^{\infty}_{t}(\dot{B}^{\frac{d}{2}-1}_{2,1})}\|u^\tau-v^\tau\|_{L^{1}_{t}(\dot{B}^{\frac{d}{2}}_{2,1})}\\
&\qquad+\Big(1+\|n-1\|_{L^{\infty}_{t}(\dot{B}^{\frac{d}{2}}_{2,1})}\Big)\Big\|\frac{\rho^\tau}{\rho^\tau+n^\tau}\Big\|_{L^{\infty}_{t}(\dot{B}^{\frac{d}{2}-1}_{2,1}\cap \dot{B}^{\frac{d}{2}}_{2,1})} \|u^\tau-v^\tau\|_{L^{1}_{t}(\dot{B}^{\frac{d}{2}}_{2,1})}
\lesssim \sqrt{\tau}.
\end{aligned}
\end{equation}
Regarding $\widetilde{F}_3$, noticing the following equality:
\begin{equation*}
\frac{P'(n^{\tau})}{\rho^{\tau}+n^{\tau}}-\frac{P'(n)}{\rho+n}=\frac{P'(n^{\tau})-P'(n)}{\rho^{\tau}+n^{\tau}}-\frac{P'(n)}{(\rho^{\tau}+n^{\tau})(\rho+n)}(\widetilde{\rho}^{\tau}+\widetilde{n}^{\tau}),
\end{equation*}
from which we easily deduce that
\begin{equation}\label{F3wide}
\begin{aligned}
&\|\widetilde{F}_3\|_{L^{1}_{t}(\dot{B}^{\frac{d}{2}-2}_{2,1})}\\
&\quad\lesssim \|\widetilde{V}^\tau\|_{L^{1}_{t}(\dot{B}^{\frac{d}{2}}_{2,1})}\|\nabla v\|_{L^{\infty}_{t}(\dot{B}^{\frac{d}{2}-2}_{2,1})}+\Big\|\frac{P'(n^{\tau})}{\rho^{\tau}+n^{\tau}}-1 \Big\|_{L^{\infty}_{t}(\dot{B}^{\frac{d}{2}-1}_{2,1})} \|(\nabla \widetilde{n}^{\tau})^{\ell}\|_{L^{1}_{t}(\dot{B}^{\frac{d}{2}-1}_{2,1})}\\
&\qquad+\Big\|\frac{P'(n^{\tau})}{\rho^{\tau}+n^{\tau}}-1 \Big\|_{L^{\infty}_{t}(\dot{B}^{\frac{d}{2}}_{2,1})} \|(\nabla \widetilde{n}^{\tau})^h\|_{L^{1}_{t}(\dot{B}^{\frac{d}{2}-2}_{2,1})} \\
&\qquad+\Big\|\frac{P'(n^{\tau})-P'(n)}{\rho^{\tau}+n^{\tau}}-\frac{P'(n)}{(\rho^{\tau}+n^{\tau})(\rho+n)}(\widetilde{\rho}^{\tau}+\widetilde{n}^{\tau})\Big\|_{L^{\infty}_{t}(\dot{B}^{\frac{d}{2}-2}_{2,1})}\|(\nabla n)^{\ell}\|_{L^1_t(\dot{B}^{\frac{d}{2}}_{2,1})}\\
&\qquad+\Big\|\frac{P'(n^{\tau})-P'(n)}{\rho^{\tau}+n^{\tau}}-\frac{P'(n)}{(\rho^{\tau}+n^{\tau})(\rho+n)}(\widetilde{\rho}^{\tau}+\widetilde{n}^{\tau})\Big\|_{L^{\infty}_{t}(\dot{B}^{\frac{d}{2}-1}_{2,1})}\|(\nabla n)^{h}\|_{L^1_t(\dot{B}^{\frac{d}{2}-1}_{2,1})}\\
&\qquad+\Big\|\frac{1}{\rho^{\tau}+n^{\tau}}-1\Big\|_{L^{\infty}_{t}(\dot{B}^{\frac{d}{2}}_{2,1})}\|\nabla^{2}\widetilde{V}^\tau\|_{L^{1}_{t}(\dot{B}^{\frac{d}{2}-2}_{2,1})}+\Big\|\frac{1}{\rho^{\tau}+n^{\tau}}-\frac{1}{\rho+n}\Big\|_{L^{\infty}_{t}(\dot{B}^{\frac{d}{2}-1}_{2,1})}\|\nabla^2 v\|_{L^{1}_{t}(\dot{B}^{\frac{d}{2}-1}_{2,1})}\\
&\quad\lesssim \delta_0^* \|(\widetilde{\rho}^\tau,\widetilde{n}^\tau)\|_{L^{\infty}_t(\dot{B}^{\frac{d}{2}-2}_{2,1}\cap\dot{B}^{\frac{d}{2}-1}_{2,1})}+\delta_0\Big(\|\widetilde{n}^\tau\|_{L^1_t(\dot{B}^{\frac{d}{2}}_{2,1})}^{\ell}+\|\widetilde{n}^\tau\|_{L^1_t(\dot{B}^{\frac{d}{2}-1}_{2,1})}^{h}\Big)+(\delta_{0}+\delta_0^*)\|\widetilde{V}^\tau\|_{L^{1}_{t}(\dot{B}^{\frac{d}{2}}_{2,1})}.
\end{aligned}
\end{equation}
Before bounding $\widetilde{F}_{4}$, we need to handle the time derivatives for $\rho^{\tau}$ and $n^{\tau}$. In fact, according to $\eqref{pENS}_{1}$, $\eqref{pENS}_{3}$, and the embedding property $L^{\infty}_t(\dot{B}^{\frac{d}{2}-1}_{2,1})\cap L^1_t(\dot{B}^{\frac{d}{2}+1}_{2,1})\hookrightarrow L^2_t(\dot{B}^{\frac{d}{2}}_{2,1})$, we infer
\begin{equation}\nonumber
\begin{aligned}
&\|\partial_{t}\rho^{\tau}\|_{L^{2}_{t}(\dot{B}^{\frac{d}{2}-1}_{2,1})}\lesssim \|\rho^{\tau}\|_{L^{\infty}_{t}(\dot{B}^{\frac{d}{2}}_{2,1})} \|u^{\tau}\|_{L^{2}_{t}(\dot{B}^{\frac{d}{2}}_{2,1})}\lesssim \delta_{0}^{2},\\
&\|\partial_{t}n^{\tau}\|_{L^{2}_{t}(\dot{B}^{\frac{d}{2}-1}_{2,1})}\lesssim (1+\|n^{\tau}-1\|_{L^{\infty}_{t}(\dot{B}^{\frac{d}{2}}_{2,1})})\|v^{\tau}\|_{L^{2}_{t}(\dot{B}^{\frac{d}{2}}_{2,1})}\lesssim (1+\delta_{0})\delta_{0},
\end{aligned}
\end{equation}
which, together with \eqref{uv2}, \eqref{F1}, and the $\mathcal{O}(\sqrt{\tau})$-bound of $u^{\tau}-v^{\tau}$ in \eqref{uvcon}, gives rise to
\begin{equation}\label{nVF3}
\begin{aligned}
&\|\widetilde{F}_{4}\|_{L^{1}_{t}(\dot{B}^{\frac{d}{2}-2}_{2,1})}\\
&\quad\lesssim(1+\delta_{0})\|u^{\tau}-v^{\tau}\|_{L^{2}_{t}(\dot{B}^{\frac{d}{2}-1}_{2,1})}\Big(\|(\partial_{t}\rho^{\tau},\partial_{t}n^{\tau})\|_{L^{2}_{t}(\dot{B}^{\frac{d}{2}-1}_{2,1})}+\|(v,u^\tau,v^{\tau})\|_{L^{2}_{t}(\dot{B}^{\frac{d}{2}}_{2,1})}\Big) \\
&\qquad+(1+\delta_{0})^2\|u^\tau-v^\tau\|_{L^1_t(\dot{B}^{\frac{d}{2}}_{2,1})}\lesssim \sqrt{\tau}.
\end{aligned}
\end{equation} 
Under the condition \eqref{idcd5}, one can check that $\|\widetilde{V}^{\tau}(0)\|_{\dot{B}^{\frac{d}{2}-2}_{2,1}}\lesssim \sqrt{\tau}$. Substituting the above estimates \eqref{nVF1}-\eqref{nVF3} into \eqref{nVeineq}, we arrive at
\begin{equation}\label{nVineq}
\begin{aligned}
&\|(\widetilde{n}^{\tau},\nabla \widetilde{n}^{\tau}, \widetilde{V}^{\tau})\|_{\widetilde{L}^{\infty}_{t}(\dot{B}^{\frac{d}{2}-2}_{2,1})}+\|(\widetilde{n}^{\tau},\nabla \widetilde{n}^{\tau}, \widetilde{V}^{\tau})\|_{L^{1}_{t}(\dot{B}^{\frac{d}{2}}_{2,1})}^{\ell}+\|(\widetilde{n}^{\tau},\nabla \widetilde{n}^{\tau}, \widetilde{V}^{\tau})\|_{L^{1}_{t}(\dot{B}^{\frac{d}{2}-2}_{2,1})}^{h}\\
&\quad\lesssim\|(\widetilde{n}^{\tau},\nabla \widetilde{n}^{\tau}, \widetilde{V}^{\tau})(0)\|_{\dot{B}^{\frac{d}{2}-2}_{2,1}}+\delta_{0}^{*}\|\widetilde{\rho}^{\tau}\|_{L^{\infty}_{t}(\dot{B}^{\frac{d}{2}-2}_{2,1}\cap\dot{B}^{\frac{d}{2}-1}_{2,1})}+(\delta_{0}+\delta_{0}^{*})\|(\widetilde{n}^{\tau},\nabla \widetilde{n}^{\tau}, \widetilde{V}^{\tau})\|_{\widetilde{L}^{\infty}_{t}(\dot{B}^{\frac{d}{2}-2}_{2,1})}\\
&\qquad+\delta_0\Big(\|\widetilde{n}^\tau\|_{L^1_t(\dot{B}^{\frac{d}{2}}_{2,1})}^{\ell}+\|\widetilde{n}^\tau\|_{L^1_t(\dot{B}^{\frac{d}{2}-1}_{2,1})}^{h}\Big)+(\delta_{0}+\delta_{0}^{*})\|\widetilde{V}^{\tau}\|_{L^{1}_{t}(\dot{B}^{\frac{d}{2}}_{2,1})}+\sqrt{\tau}.
\end{aligned}
\end{equation}  
Furthermore, applying the maximal regularity estimates in Lemma \ref{esheat} to the equation $\eqref{errnVe}_{2}$ and making use of  \eqref{F3wide}-\eqref{nVF3}, we have
\begin{equation}\label{nVhigh}
\begin{aligned}
\|\widetilde{V}^{\tau}\|_{L^{1}_{t}(\dot{B}^{\frac{d}{2}}_{2,1})}^{h}&\lesssim \|\widetilde{V}^{\tau}(0)\|_{\dot{B}^{\frac{d}{2}-2}_{2,1}}^{h}+\|\nabla \widetilde{n}^{\tau}\|_{L^{1}_{t}(\dot{B}^{\frac{d}{2}-2}_{2,1})}^{h}+\|v^{\tau}\cdot\nabla\widetilde{V}^{\tau}\|_{L^{1}_{t}(\dot{B}^{\frac{d}{2}-2}_{2,1})}^{h}+\|(\widetilde{F}_{3},\widetilde{F}_{4})\|_{L^{1}_{t}(\dot{B}^{\frac{d}{2}-2}_{2,1})}^{h}\\
&\lesssim \|\widetilde{V}^{\tau}(0)\|_{\dot{B}^{\frac{d}{2}-2}_{2,1}}+\|\nabla \widetilde{n}^{\tau}\|_{L^{1}_{t}(\dot{B}^{\frac{d}{2}-2}_{2,1})}^{h}+\delta_{0}^{*}\|(\widetilde{\rho}^{\tau},\widetilde{n}^{\tau})\|_{L^{\infty}_{t}(\dot{B}^{\frac{d}{2}-2}_{2,1}\cap\dot{B}^{\frac{d}{2}-1}_{2,1})}\\
&\quad+\delta_0\Big(\|\widetilde{n}^\tau\|_{L^1_t(\dot{B}^{\frac{d}{2}}_{2,1})}^{\ell}+\|\widetilde{n}^\tau\|_{L^1_t(\dot{B}^{\frac{d}{2}-1}_{2,1})}^{h}\Big)+(\delta_{0}+\delta_{0}^{*})\|\widetilde{V}^{\tau}\|_{L^{1}_{t}(\dot{B}^{\frac{d}{2}}_{2,1})}+\sqrt{\tau}.
\end{aligned}
\end{equation}
Combining \eqref{nVineq} and \eqref{nVhigh}, together with the smallness of $\delta_{0}$ and $\delta_{0}^{*}$, we obtain the desired convergence estimate \eqref{deltanV}, and hence we complete the proof of Lemma \ref{lemmanV}.
\end{proof}

\subsubsection{Estimates for \texorpdfstring{$\widetilde{\rho}^{\tau}$}{rho-tau} }
Then, we state the following lemma about the estimates of the error variable $\widetilde{\rho}^{\tau}$ corresponding to the non-dissipated variables. 

\begin{lemma}\label{lemmawiderho}
Let $(\rho^{\tau},u^{\tau},n^{\tau},v^{\tau})$ and $(\rho,n,v)$ be the global solutions to \eqref{pENS} and \eqref{DF} obtained in Theorem \ref{theorem1} and Corollary \ref{corollary1}, respectively. Then, under the condition \eqref{idcd5}, it holds that
\begin{equation}\label{deltarho}
\begin{aligned}
&\|\widetilde{\rho}^{\tau}\|_{\widetilde{L}^{\infty}_{t}(\dot{B}^{\frac{d}{2}-2}_{2,1}\cap \dot{B}^{\frac{d}{2}-1}_{2,1})}\lesssim \|\widetilde{\rho}^{\tau}(0)\|_{\dot{B}^{\frac{d}{2}-2}_{2,1}\cap \dot{B}^{\frac{d}{2}-1}_{2,1}}+\delta_{0}^{*}\|\widetilde{V}^{\tau}\|_{L^{1}_{t}(\dot{B}^{\frac{d}{2}}_{2,1})}+\sqrt{\tau}.
\end{aligned}
\end{equation}
\end{lemma}

\begin{proof}
Applying Lemma \ref{estrans} to the transport equation \eqref{errrho}, we have
\begin{equation}\nonumber
\begin{aligned}
&\|\widetilde{\rho}^{\tau}\|_{\widetilde{L}^{\infty}_{t}(\dot{B}^{\frac{d}{2}-2}_{2,1}\cap \dot{B}^{\frac{d}{2}-1}_{2,1})}\\
&\quad\leq\exp\Big(C\|u^{\tau}\|_{L^{1}_{t}(\dot{B}^{\frac{d}{2}+1}_{2,1})}\Big)\Big( \|\widetilde{\rho}^{\tau}(0)\|_{\dot{B}^{\frac{d}{2}-2}_{2,1}\cap \dot{B}^{\frac{d}{2}-1}_{2,1}}+\|\widetilde{\rho}^{\tau}\operatorname{div} u^{\tau}\|_{L^{1}_{t}(\dot{B}^{\frac{d}{2}-2}_{2,1}\cap\dot{B}^{\frac{d}{2}-1}_{2,1})}\\
&\qquad\qquad\qquad\qquad\qquad\qquad\quad+ \|\rho\widetilde{V}^{\tau}\|_{L^{1}_{t}(\dot{B}^{\frac{d}{2}-1}_{2,1}\cap\dot{B}^{\frac{d}{2}}_{2,1})}+ \Big\|\frac{\rho n^{\tau}}{\rho^{\tau}+n^{\tau}}(u^\tau-v^\tau)\Big\|_{L^{1}_{t}(\dot{B}^{\frac{d}{2}-1}_{2,1}\cap\dot{B}^{\frac{d}{2}}_{2,1})}\Big).
\end{aligned}
\end{equation}
Noticing that $-\frac{d}{2}<\frac{d}{2}-2$ for $d\geq3$, we make use of \eqref{result1} and the product estimates in \eqref{uv2} to derive
\begin{equation}\nonumber
\begin{aligned}
\|\widetilde{\rho}^{\tau}\operatorname{div} u^{\tau}\|_{L^{1}_{t}(\dot{B}^{\frac{d}{2}-2}_{2,1}\cap\dot{B}^{\frac{d}{2}-1}_{2,1})}\lesssim \|\widetilde{\rho}^{\tau}\|_{L^{\infty}_{t}(\dot{B}^{\frac{d}{2}-2}_{2,1}\cap\dot{B}^{\frac{d}{2}-1}_{2,1})}\|u^{\tau}\|_{L^{1}_{t}(\dot{B}^{\frac{d}{2}+1}_{2,1})}\lesssim\delta_{0}\|\widetilde{\rho}^{\tau}\|_{L^{\infty}_{t}(\dot{B}^{\frac{d}{2}-2}_{2,1}\cap\dot{B}^{\frac{d}{2}-1}_{2,1})}.
\end{aligned}
\end{equation}
Similarly, by \eqref{result1}, \eqref{result5},  \eqref{uv2}, and estimates for composite functions, we have
\begin{equation}\nonumber
\begin{aligned}
\|\rho\widetilde{V}^{\tau}\|_{L^{1}_{t}(\dot{B}^{\frac{d}{2}-1}_{2,1}\cap\dot{B}^{\frac{d}{2}}_{2,1})}&\lesssim \|\rho\|_{L^{\infty}_{t}(\dot{B}^{\frac{d}{2}-1}_{2,1}\cap\dot{B}^{\frac{d}{2}}_{2,1})}\|\widetilde{V}^{\tau}\|_{L^{1}_{t}(\dot{B}^{\frac{d}{2}}_{2,1})}\lesssim\delta_{0}^{*}\|\widetilde{V}^{\tau}\|_{L^{1}_{t}(\dot{B}^{\frac{d}{2}}_{2,1})},
\end{aligned}
\end{equation}
and
\begin{equation}\nonumber
\begin{aligned}
&\Big\|\frac{\rho n^{\tau}}{\rho^{\tau}+n^{\tau}}(u^\tau-v^\tau)\Big\|_{L^{1}_{t}(\dot{B}^{\frac{d}{2}-1}_{2,1}\cap\dot{B}^{\frac{d}{2}}_{2,1})}\\
&\quad\lesssim \Big(1+\|(\rho^{\tau},n^\tau-1)\|_{L^{\infty}_{t}(\dot{B}^{\frac{d}{2}}_{2,1})}\Big)\|\rho\|_{L^{\infty}_{t}(\dot{B}^{\frac{d}{2}-1}_{2,1}\cap\dot{B}^{\frac{d}{2}}_{2,1})} \|u^\tau-v^\tau\|_{L^{1}_{t}(\dot{B}^{\frac{d}{2}}_{2,1})}\lesssim \sqrt{\tau}.
\end{aligned}
\end{equation}
Under the condition \eqref{idcd5}, the combination of the above estimates as well as the smallness of $\delta_{0}$ leads to the desired estimate \eqref{deltarho}, and hence the proof is completed.
\end{proof}

\subsection{Proof of Theorem \ref{theorem3}}
We are in a position to show Theorem \ref{theorem3} based on Lemmas \ref{lemmanV} and \ref{lemmawiderho}. Since $\delta_0^*$ can be uniformly small, it follows from \eqref{deltanV} and \eqref{deltarho} that
\begin{equation}\label{ddghhb}
\begin{aligned}		&\|\widetilde{\rho}^{\tau}\|_{\widetilde{L}^{\infty}_{t}(\dot{B}^{\frac{d}{2}-2}_{2,1}\cap \dot{B}^{\frac{d}{2}-1}_{2,1})}+\|\widetilde{n}^{\tau}\|_{\widetilde{L}^{\infty}_{t}(\dot{B}^{\frac{d}{2}-2}_{2,1}\cap\dot{B}^{\frac{d}{2}-1}_{2,1} )}+\|\widetilde{V}^{\tau}\|_{\widetilde{L}^{\infty}_{t}(\dot{B}^{\frac{d}{2}-2}_{2,1})}\\
&\quad+\|\widetilde{n}^{\tau}\|_{L^{1}_{t}(\dot{B}^{\frac{d}{2}}_{2,1})}^{\ell}+\|\widetilde{n}^{\tau}\|_{L^{1}_{t}(\dot{B}^{\frac{d}{2}-1}_{2,1})}^{h}+\|\widetilde{V}^{\tau}\|_{L^{1}_{t}(\dot{B}^{\frac{d}{2}}_{2,1})}\lesssim \sqrt{\tau}.
\end{aligned}
\end{equation}
The convergence rates of $(\widetilde{u}^{\tau},\widetilde{v}^{\tau})$ can be derived from the estimates of $\widetilde{V}^{\tau}$ and $u^{\tau}-v^{\tau}$. Indeed, based on \eqref{erruvV} and \eqref{ddghhb}, we arrive at
\begin{equation}\label{deltauv}
\begin{aligned}
&\|(\widetilde{u}^{\tau},\widetilde{v}^{\tau})\|_{L^{1}_{t}(\dot{B}^{\frac{d}{2}}_{2,1})}\lesssim\|\widetilde{V}^{\tau}\|_{L^{1}_{t}(\dot{B}^{\frac{d}{2}}_{2,1})}+\Big(1+\|(\rho^{\tau},n^{\tau}-1)\|_{L^{\infty}_{t}(\dot{B}^{\frac{d}{2}}_{2,1})}\Big)\|u^{\tau}-v^{\tau}\|_{L^{1}_{t}(\dot{B}^{\frac{d}{2}}_{2,1})}\lesssim \sqrt{\tau}.
\end{aligned}
\end{equation}
Combining \eqref{deltanV}, \eqref{deltarho}, and \eqref{deltauv}, we conclude the convergence estimates \eqref{result6}. Furthermore, for any $\eta\in(0,1]$, taking full advantage of the uniform bounds and the interpolation inequality \eqref{inter}, we deduce that
\begin{equation*}
\begin{aligned}
&\|(\widetilde{\rho}^{\tau},\widetilde{n}^{\tau})\|_{L^{\infty}_{t}(\dot{B}^{\frac{d}{2}-\eta}_{2,1})}+\|(\widetilde{u}^{\tau},\widetilde{v}^{\tau})\|_{ L^1_t(\dot{B}^{\frac{d}{2}+1-\eta}_{2,1})}\lesssim \tau^{\frac{\eta}{2}},
\end{aligned}
\end{equation*}
which implies \eqref{conver} and the convergence of the Euler-NS system \eqref{pENS} to the limiting DF system \eqref{DF} in $\mathcal{S}'(\mathbb{R}_+\times\mathbb{R}^d)$. Therefore, we complete the proof of Theorem \ref{theorem3}.

\section{Optimal time-decay estimates}\label{sectiondecay}

In this section, we shall establish the uniform time-decay rates of the strong solution to the Cauchy problem \eqref{pENS1} so as to prove Theorem \ref{theorem2}. As in Section \ref{sectionexistence}, we omit the superscript $^\tau$ from the solution and initial data throughout this section.

We refine the time-weighted inequalities in \cite{danchinxu17,lishouSIAMJMA23} to handle the presence of the singular parameter $\tau$. To this end, we define the initial energy functional by
\begin{align}\label{Z_0}
&\mathcal{Y}_0:=\|(\rho_{0},u_0,a_{0},v_{0})\|_{\dot{B}^{\sigma_{1}}_{2,\infty}}^{\ell}+\|(\rho_0,a_{0})\|_{\dot{B}^{\frac{d}{2}-1}_{2,1}\cap\dot{B}^{\frac{d}{2}}_{2,1}}+\|(u_{0},v_{0})\|_{\dot{B}^{\frac{d}{2}-1}_{2,1}}+\tau\|u_{0}\|_{\dot{B}^{\frac{d}{2}+1}_{2,1}},
\end{align}
and the time-weighted energy functional as
\begin{equation}\label{prioridc}
\begin{aligned}
\mathcal{Y}(t)&:=\sup_{\sigma_{1}<\sigma\leq \frac{d}{2}+1}\|\langle s\rangle^{\frac{1}{2}(\sigma-\sigma_{1})}(u,a,v)\|_{L^{\infty}_{t}(\dot{B}^{\sigma}_{2,1})}^{\ell}+\frac{1}{\tau}\|\langle s\rangle^{\frac{1}{2}}(u-v)\|_{L^{\infty}_{t}(1,t;\,\dot{B}^{\sigma_{1}}_{2,\infty})}^{\ell}\\
&\qquad+\frac{1}{\tau}\sup_{\sigma_{1}<\sigma\leq\frac{d}{2}}\|\langle s\rangle^{\frac{1}{2}(\sigma-\sigma_{1}+1)}(u-v)\|_{L^{\infty}_{t}(1,t;\,\dot{B}^{\sigma}_{2,1})}^{\ell}\\
&\qquad+\|\langle s\rangle^{\alpha}(u,v)\|_{\widetilde{L}^{\infty}_{t}(1,t;\dot{B}^{\frac{d}{2}+1}_{2,1})}^{h}+\|\langle s\rangle^{\alpha}a\|_{\widetilde{L}^{\infty}_{t}(\dot{B}^{\frac{d}{2}}_{2,1})}^{h}+\frac{1}{\tau}\|\langle s\rangle^{\alpha}(u-v)\|_{\widetilde{L}^{\infty}_{t}(1,t;\dot{B}^{\frac{d}{2}-1}_{2,1})}^{h},
\end{aligned}
\end{equation}
for $t>1$, where $\langle s\rangle=\sqrt{1+s^2}$, $\sigma_{1}\in[-\frac{d}{2},\frac{d}{2}-1)$, and $\alpha:=\frac{1}{2}(\frac{d}{2}+1-\sigma_{1})$. The above weighted norms on $[1,t]$ enable us to effectively capture the maximal decay properties derived from the $L^1$ time integrability estimates in \eqref{result1}. 

Compared with the decay for the compressible Navier-Stokes equations \cite{danchinxu17} or the Euler-NS system with pressure \cite{lishouSIAMJMA23}, we obtain the uniform-in-$\tau$ decay estimates; however, the rate $\langle t\rangle^{-\alpha}$ in the high-frequency regime is slower, due to the influence of the non-decaying component $\rho$ on the drag force term $\frac{1}{\tau}\rho(u-v)$. In this case, the improved decay estimates for the relative velocity $u-v$ are crucial to ensure sufficient decay and overcome the singularity arising from the factor $\frac{1}{\tau}$.

Before showing the uniform-in-time bound of $\mathcal{Y}(t)$, we derive the uniform regularity evolution of $\rho$ under the additional condition $\rho_0^\ell\in\dot{B}^{\sigma_1}_{2,\infty}$. Recalling $\mathcal{X}_0$ given in \eqref{priorieg0}, it follows from the regularity estimates in Lemma \ref{estrans} for $\eqref{pENS1}_{1}$ that
\begin{equation*}
\begin{aligned}
\|\rho\|_{\widetilde{L}^{\infty}_{t}(\dot{B}^{\sigma_{1}}_{2,\infty})}
&\leq \exp\Big(C\|u\|_{L^{1}_{t}(\dot{B}^{\frac{d}{2}+1}_{2,1})}\Big)\Big(\|\rho_{0}\|_{\dot{B}^{\sigma_{1}}_{2,\infty}}+\|\rho\operatorname{div}u\|_{L^{1}_{t}(\dot{B}^{\sigma_{1}}_{2,\infty})}\Big)\\
&\lesssim \|\rho_{0}\|_{\dot{B}^{\sigma_{1}}_{2,\infty}}+\|\rho\|_{\widetilde{L}^{\infty}_{t}(\dot{B}^{\sigma_{1}}_{2,\infty})}\|\operatorname{div}u\|_{L^{1}_{t}(\dot{B}^{\frac{d}{2}}_{2,1})}\\
&\lesssim\|\rho_{0}^{\ell}\|_{\dot{B}^{\sigma_{1}}_{2,\infty}}+\|\rho_{0}^{h}\|_{\dot{B}^{\frac{d}{2}-1}_{2,1}}+\mathcal{X}_{0}\|\rho\|_{\widetilde{L}^{\infty}_{t}(\dot{B}^{\sigma_{1}}_{2,\infty})}.
\end{aligned}
\end{equation*}
Then, in terms of the smallness of $\mathcal{X}_{0}$, we can verify for $\sigma_{1}\in[-\frac{d}{2},\frac{d}{2}-1)$ that
\begin{equation}\label{rhosgm1}
\|\rho\|_{\widetilde{L}^{\infty}_{t}(\dot{B}^{\sigma_{1}}_{2,\infty})}\lesssim\mathcal{Y}_0,\quad t\in\mathbb{R}_{+}.
\end{equation}
This, together with the interpolation property \eqref{inter}, implies for $\sigma\in(\sigma_{1},\frac{d}{2}]$ that
\begin{equation}\label{rhosgm2}
\|\rho\|_{\widetilde{L}^{\infty}_{t}(\dot{B}^{\sigma}_{2,1})}\lesssim \|\rho\|_{\widetilde{L}^{\infty}_{t}(\dot{B}^{\sigma_{1}}_{2,\infty})}+\|\rho\|_{\widetilde{L}^{\infty}_{t}(\dot{B}^{\frac{d}{2}}_{2,1})}\lesssim \mathcal{Y}_0,\quad t\in\mathbb{R}_{+}.
\end{equation}

\subsection{Low-frequency analysis}
Following the same line for the proof of the functional inequality \eqref{ns6} but without using commutators, for $j\leq0$ we have
\begin{equation}\label{ns6:low}
\begin{aligned}
&\frac{{\rm{d}}}{{\rm{d}}t}\mathcal{E}_{j}(t)+c 2^{2j}\mathcal{E}_{j}(t)\lesssim \|\dot{\Delta}_{j}(\operatorname{div}(av),v\cdot\nabla v,\tau^{-1}\rho(u-v),F_{1},F_{2})\|_{L^{2}} \|\dot{\Delta}_{j}(a,\nabla a, v)\|_{L^{2}},
\end{aligned}
\end{equation}
for some uniform constant $c>0$. Applying the Gr\"onwall inequality to \eqref{ns6:low} and recalling the relation $\mathcal{E}_{j}(t)\sim \|\dot{\Delta}_j(a,\nabla a,v)\|_{L^2}^2$, we get
\begin{equation}\label{nsgm1}
\begin{aligned}
\|\dot{\Delta}_{j}(a,\nabla a, v)\|_{L^{2}}&\lesssim e^{-c 2^{2j} t}\|\dot{\Delta}_{j}(a_{0},\nabla a_{0},v_{0})\|_{L^{2}}\\
&\quad+\int_{0}^{t}e^{-c 2^{2j}(t-s')}\|\dot{\Delta}_{j}(\operatorname{div}(av),v\cdot\nabla v,\tau^{-1}\rho(u-v),F_{1},F_{2})\|_{L^{2}}{\rm{d}}s'.
\end{aligned}
\end{equation}
For any $\sigma>\sigma_{1}$, we notice that
\begin{equation*}
t^{\frac{\sigma-\sigma_{1}}{2}}\sum_{j\in\mathbb{Z}}2^{j(\sigma-\sigma_{1})}e^{-c 2^{2j}t}\lesssim 1,
\end{equation*}
which implies for any function $f$ that
\begin{equation*}
\sum_{j\leq 0} 2^{j\sigma} e^{-c 2^{2j}t}\|\dot{\Delta}_jf\|_{L^2} \lesssim t^{-\frac{\sigma-\sigma_{1}}{2}}\|f\|_{\dot{B}^{\sigma_{1}}_{2,\infty}}^{\ell}.
\end{equation*}
Consequently, due to the low-frequency cut-off property, we directly deduce from \eqref{nsgm1} that
\begin{equation}\label{nsgm2}
\begin{aligned}
\|(a,v)\|_{\dot{B}^{\sigma}_{2,1}}^{\ell}&\lesssim \langle t\rangle^{-\frac{\sigma-\sigma_{1}}{2}}\|(a_{0},v_{0})\|_{\dot{B}^{\sigma_{1}}_{2,\infty}}^{\ell}\\
&\quad+\int_{0}^{t}\langle t-s'\rangle^{-\frac{\sigma-\sigma_{1}}{2}}\|(\operatorname{div}(av),v\cdot\nabla v,\tau^{-1}\rho(u-v),F_{1},F_{2})\|_{\dot{B}^{\sigma_1}_{2,\infty}}^\ell{\rm{d}}s'.
\end{aligned}
\end{equation}
We start with the key estimates of the singular drag force term $\frac{1}{\tau}\rho(u-v)$ on the right-hand side of \eqref{nsgm2}. Note that the non-dissipative variable $\rho$ does not have any time-decay rate or time integrability. For $\sigma_{1}\in[-\frac{d}{2},\frac{d}{2}-1)$, the product law $\dot{B}^{\sigma_1}_{2,\infty}\times \dot{B}^{\frac{d}{2}}_{2,1}\hookrightarrow \dot{B}^{\sigma_1}_{2,\infty}$ is valid due to \eqref{uv3}, and then it holds
\begin{equation}\label{nsgm3}
\begin{aligned}
\frac{1}{\tau}\|\rho(u-v)\|_{\dot{B}^{\sigma_{1}}_{2,\infty}}^{\ell}&\lesssim\frac{1}{\tau}\Big(\|\rho(u-v)^{\ell}\|_{\dot{B}^{\sigma_{1}}_{2,\infty}}+\|\rho(u-v)^{h}\|_{\dot{B}^{\sigma_{1}}_{2,\infty}}\Big)\\
&\lesssim\|\rho\|_{\dot{B}^{\sigma_{1}}_{2,\infty}}\frac{1}{\tau}\|u-v\|_{\dot{B}^{\frac{d}{2}}_{2,1}}^{\ell}+\|\rho\|_{\dot{B}^{\sigma_{1}+1}_{2,\infty}}\frac{1}{\tau}\|u-v\|_{\dot{B}^{\frac{d}{2}-1}_{2,1}}^{h}.
\end{aligned}
\end{equation}
For any $t>1$, we split the integration into two parts:
\begin{equation*}
\int_{0}^{t}\langle t-s'\rangle^{-\frac{\sigma-\sigma_{1}}{2}}\frac{1}{\tau}\|\rho(u-v)\|_{\dot{B}^{\sigma_{1}}_{2,\infty}}^{\ell}{\rm{d}}s'=\bigg\{\int_{0}^{1}+\int_{1}^{t}\bigg\}\langle t-s'\rangle^{-\frac{\sigma-\sigma_{1}}{2}}\frac{1}{\tau}\|\rho(u-v)\|_{\dot{B}^{\sigma_{1}}_{2,\infty}}^{\ell}{\rm{d}}s'.
\end{equation*}
It follows from \eqref{uvcon}, \eqref{rhosgm1}-\eqref{rhosgm2}, \eqref{nsgm3}, and the fact $\langle t-s'\rangle\sim\langle t\rangle$ for $s'\in[0,1]$ that
\begin{equation}\label{nsgm5}
\begin{aligned}
&\int_{0}^{1}\langle t-s'\rangle^{-\frac{\sigma-\sigma_{1}}{2}}\frac{1}{\tau}\|\rho(u-v)\|_{\dot{B}^{\sigma_{1}}_{2,\infty}}^{\ell}{\rm{d}}s'\\
&\quad\lesssim \langle t\rangle^{-\frac{\sigma-\sigma_1}{2}}\Big(\|\rho\|_{L^{\infty}_{t}(\dot{B}^{\sigma_{1}}_{2,\infty})}\frac{1}{\tau}\|u-v\|_{L^{1}_{t}(\dot{B}^{\frac{d}{2}}_{2,1})}^{\ell}+\|\rho\|_{L^{\infty}_{t}(\dot{B}^{\sigma_{1}+1}_{2,\infty})}\frac{1}{\tau}\|u-v\|_{L^{1}_{t}(\dot{B}^{\frac{d}{2}-1}_{2,1})}^{h}\Big)\\
&\quad\lesssim\mathcal{Y}_0\langle t\rangle^{-\frac{\sigma-\sigma_1}{2}}\Big(\frac{1}{\tau}\|u-v\|_{L^{1}_{t}(\dot{B}^{\frac{d}{2}}_{2,1})}^{\ell}+\frac{1}{\tau}\|u-v\|_{L^{1}_{t}(\dot{B}^{\frac{d}{2}-1}_{2,1})}^{h}\Big)\lesssim \mathcal{Y}_0\langle t\rangle^{-\frac{\sigma-\sigma_{1}}{2}}.
\end{aligned}
\end{equation}
Regarding the time integration on $[1,t]$, we have
\begin{equation}\label{nsgm6}
\begin{aligned}
&\int_{1}^{t}\langle t-s'\rangle^{-\frac{\sigma-\sigma_{1}}{2}}\frac{1}{\tau}\|\rho(u-v)\|_{\dot{B}^{\sigma_{1}}_{2,\infty}}^{\ell}{\rm{d}}s'\\
&\quad\lesssim\int_{1}^{t}\langle t-s'\rangle^{-\frac{\sigma-\sigma_{1}}{2}}\Big(\|\rho\|_{\dot{B}^{\sigma_{1}}_{2,\infty}}\frac{1}{\tau}\|u-v\|_{\dot{B}^{\frac{d}{2}}_{2,1}}^{\ell}+\|\rho\|_{\dot{B}^{\sigma_{1}+1}_{2,\infty}}\frac{1}{\tau}\|u-v\|_{\dot{B}^{\frac{d}{2}-1}_{2,1}}^{h}\Big){\rm{d}}s'\\
&\quad\lesssim \mathcal{Y}_0\int_{1}^{t}\langle t-s'\rangle^{-\frac{\sigma-\sigma_{1}}{2}}\Big(\langle s'\rangle^{-\frac{1}{2}(\frac{d}{2}-\sigma_{1}+1)}+\langle s'\rangle^{-\alpha}\Big)\mathcal{Y}(t){\rm{d}}s'\\
&\quad\lesssim\mathcal{Y}_0\langle t \rangle^{-\frac{\sigma-\sigma_{1}}{2}}\mathcal{Y}(t),
\end{aligned}
\end{equation}
for $\sigma\in(\sigma_1,\frac{d}{2}+1]$ and $\frac{1}{2}(\frac{d}{2}-\sigma_1+1)>1$. Here, we used the fact that
\begin{equation*}
\int_0^t \langle t-s'\rangle^{-\alpha_1} \langle s'\rangle^{-\alpha_2}{\rm d}s'\lesssim \langle t\rangle^{-\alpha_1},\quad 0<\alpha_1\leq \alpha_2,\quad \alpha_2>1.
\end{equation*}
In a similar argument, one can verify that
\begin{equation}
\int_{0}^{t}\langle t-s'\rangle^{-\frac{\sigma-\sigma_{1}}{2}}\|F_{2}\|_{\dot{B}^{\sigma_{1}}_{2,\infty}}^{\ell}{\rm{d}}s'\lesssim(\mathcal{Y}_0+\mathcal{Y}_0\mathcal{Y}(t))\langle t \rangle^{-\frac{\sigma-\sigma_{1}}{2}}.
\end{equation}
The remaining terms on the right-hand side of \eqref{nsgm2} can be handled similarly to those in \cite{lishouSIAMJMA23}, and here we only provide the result:
\begin{equation}\label{nsgm7}
\int_{0}^{t}\langle t-s'\rangle^{-\frac{\sigma-\sigma_{1}}{2}}\|(\operatorname{div}(av),v\cdot\nabla v,F_{1})\|_{\dot{B}^{\sigma_{1}}_{2,\infty}}^{\ell}{\rm{d}}s'\lesssim(\mathcal{Y}_0+\mathcal{Y}(t)^{2})\langle t \rangle^{-\frac{\sigma-\sigma_{1}}{2}}.
\end{equation}
Thus, substituting \eqref{nsgm3}-\eqref{nsgm7} into \eqref{nsgm2}, we obtain for $\sigma\in(\sigma_{1},\frac{d}{2}+1]$ that
\begin{equation}\label{nsgm8}
\|\langle s \rangle^{\frac{\sigma-\sigma_{1}}{2}}(a,v)\|_{L^{\infty}_{t}(\dot{B}^{\sigma}_{2,1})}^{\ell}\lesssim \mathcal{Y}_{0}+\mathcal{Y}_0\mathcal{Y}(t)+\mathcal{Y}(t)^{2}.
\end{equation}

Next, we investigate the time-decay estimates for the relative velocity $u-v$, which are essential for enclosing the time-weighted estimates. It can be derived from the Gr\"onwall inequality to \eqref{rv1} that
\begin{equation}\label{rvgm1}
\begin{aligned}
\|\dot{\Delta}_{j}(u-v)\|_{L^{2}}&\lesssim e^{-\frac{1}{\tau}t}\|\dot{\Delta}_{j}(u_{0}-v_{0})\|_{L^{2}}+\int_{0}^{t}e^{-\frac{1}{\tau}(t-s')}\|\dot{\Delta}_{j}(\nabla a,\nabla^2 v)\|_{L^{2}}{\rm{d}}s'\\
&\quad+\int_{0}^{t}e^{-\frac{1}{\tau}(t-s')}\|\dot{\Delta}_{j}(u\cdot\nabla u,v\cdot\nabla v,\tau^{-1}\rho(u-v),F_{1},F_{2})\|_{L^{2}}{\rm{d}}s'.
\end{aligned}
\end{equation}
Taking the low-frequency $\dot{B}_{2,1}^{\sigma}$-norm for $\sigma\in(\sigma_{1},\frac{d}{2}]$, together with \eqref{lh} and \eqref{inter}, we have
\begin{equation}\label{rvgm2}
\begin{aligned}
\|u-v\|_{\dot{B}^{\sigma}_{2,1}}^{\ell}&\lesssim e^{-\frac{1}{\tau}t}\|u_{0}-v_{0}\|_{\dot{B}^{\sigma_{1}}_{2,\infty}}^{\ell}+\int_{0}^{t}e^{-\frac{1}{\tau}(t-s')}\|(a,v)\|_{\dot{B}^{\sigma+1}_{2,1}}^{\ell}{\rm{d}}s'\\
&\quad+\int_{0}^{t}e^{-\frac{1}{\tau}(t-s')}\|(u\cdot\nabla u,v\cdot\nabla v,\tau^{-1}\rho(u-v),F_{1},F_{2})\|_{\dot{B}^{\sigma_{1}}_{2,\infty}}^{\ell}{\rm{d}}s'.
\end{aligned}
\end{equation}
By an argument similar to that of \eqref{nsgm6} and noting the estimate $e^{-\frac{1}{\tau}t} \lesssim \langle t \rangle^{-\frac{1}{2}(\sigma-\sigma_1+1)}$ for $\tau\leq 1$, one can easily verify for $\sigma\in(\sigma_{1},\frac{d}{2}]$ that
\begin{equation}\label{uvtall}
\begin{aligned}
&\|\langle s\rangle^{\frac{1}{2}(\sigma-\sigma_{1}+1)}(u-v)\|_{L_{t}^{\infty}(\dot{B}^{\sigma}_{2,1})}^{\ell}\lesssim \mathcal{Y}_{0}+\mathcal{Y}_0\mathcal{Y}(t)+\mathcal{Y}(t)^{2},
\end{aligned}
\end{equation}
where we have used \eqref{nsgm8}. The combination of \eqref{nsgm8}, \eqref{uvtall}, and the low-frequency cut-off property leads to the expected bound of $u$ that holds for $\sigma_1<\sigma\leq \frac{d}{2}+1$,
\begin{equation}\label{veuegm}
\begin{aligned}
\|\langle s \rangle^{\frac{\sigma-\sigma_{1}}{2}}u\|_{L^{\infty}_{t}(\dot{B}^{\sigma}_{2,1})}^{\ell}&\lesssim \|\langle s \rangle^{\frac{1}{2}\big(\min\{\sigma,\frac{d}{2}\}-\sigma_{1}+1\big)}(u-v)\|_{L^{\infty}_{t}(\dot{B}^{\min\{\sigma,\frac{d}{2}\}}_{2,1})}^{\ell}+\|\langle s \rangle^{\frac{\sigma-\sigma_{1}}{2}}v\|_{L^{\infty}_{t}(\dot{B}^{\sigma}_{2,1})}^{\ell}\\
&\lesssim \mathcal{Y}_{0}+\mathcal{Y}_0\mathcal{Y}(t)+\mathcal{Y}(t)^{2}.
\end{aligned}
\end{equation}

In addition, we establish the improved uniform-in-$\tau$ time-decay estimates of the relative velocity $u-v$ on $[1,t]$ compared with \eqref{uvtall}. The key observation lies in that for $\alpha_3>1$ and $t\geq 1$, we have
\begin{equation}\label{epntep1}
e^{-\frac{1}{\tau}t}\lesssim\tau^{\alpha_3}t^{-\alpha_3}\lesssim\tau\langle t\rangle^{-\alpha_3},
\end{equation}
and thus it holds for $\alpha_{4}>0$ that
\begin{align}\label{expvar}
\int_0^t e^{-\frac{1}{\tau}(t-s')} \langle s'\rangle^{-\alpha_4}{\rm d}s'\lesssim e^{-\frac{t}{2\tau}}\int_0^{\frac{t}{2}}  \langle s'\rangle^{-\alpha_4}{\rm d}s'+\langle\frac{t}{2}\rangle^{-\alpha_4}\int_{\frac{t}{2}}^{t} e^{-\frac{1}{\tau}(t-s')}{\rm d}s'\lesssim \tau\langle t\rangle^{-\alpha_4}.
\end{align}
The above two estimates \eqref{epntep1} and \eqref{expvar} allow us to capture an additional rate of $\tau$ for $u$ and $u-v$. Using \eqref{epntep1} and \eqref{expvar} in the right-hand side of \eqref{rvgm2}, we deduce for $\sigma\in(\sigma_{1},\frac{d}{2}]$ that
\begin{equation}\label{rvgm3}
\frac{1}{\tau}\|\langle s\rangle^{\frac{1}{2}(\sigma-\sigma_{1}+1)}(u-v)\|_{L^{\infty}_{t}(1,t;\dot{B}^{\sigma}_{2,1})}^{\ell}\lesssim \mathcal{Y}_{0}+\mathcal{Y}_0\mathcal{Y}(t)+\mathcal{Y}(t)^{2}.
\end{equation}
Similarly, taking the low-frequency $\dot{B}_{2,\infty}^{\sigma_{1}}$-norm of \eqref{rvgm1} and using \eqref{epntep1} and \eqref{expvar}, we obtain
\begin{equation}\label{rvgm4}
\begin{aligned}
\frac{1}{\tau}\| \langle s\rangle^{\frac{1}{2}}(u-v)\|_{L^{\infty}_{t}(1,t;\dot{B}^{\sigma_{1}}_{2,\infty})}^{\ell}\lesssim \mathcal{Y}_{0}+\mathcal{Y}_0\mathcal{Y}(t)+\mathcal{Y}(t)^{2}.
\end{aligned}
\end{equation}
The combination of \eqref{nsgm8}, \eqref{veuegm}, \eqref{rvgm3}, and \eqref{rvgm4} gives rise to 
\begin{equation}\label{lowtime}
\begin{aligned}
&\sup_{\sigma_{1}<\sigma\leq \frac{d}{2}+1}\|\langle s\rangle^{\frac{1}{2}(\sigma-\sigma_{1})}(u,a,v)\|_{L^{\infty}_{t}(\dot{B}^{\sigma}_{2,1})}^{\ell}+\frac{1}{\tau}\|\langle s\rangle^{\frac{1}{2}}(u-v)\|_{L^{\infty}_{t}(1,t;\,\dot{B}^{\sigma_{1}}_{2,\infty})}^{\ell}\\
&\quad+\frac{1}{\tau}\sup_{\sigma_{1}<\sigma\leq\frac{d}{2}}\|\langle s\rangle^{\frac{1}{2}(\sigma-\sigma_{1}+1)}(u-v)\|_{L^{\infty}_{t}(1,t;\,\dot{B}^{\sigma}_{2,1})}^{\ell}\lesssim \mathcal{Y}_{0}+\mathcal{Y}_0\mathcal{Y}(t)+\mathcal{Y}(t)^{2}.
\end{aligned}
\end{equation}

\subsection{High-frequency analysis}
We further perform the uniform time-weighted estimates of solutions in high frequencies. It follows from \eqref{ns6} and Gr\"onwall's inequality that
\begin{equation*}
\begin{aligned}
&\|\dot{\Delta}_j(\nabla a,v)\|_{L^2}\\
&\quad\lesssim e^{-ct}\|\dot{\Delta}_j(\nabla a_0,v_0)\|_{L^2}+\int_0^t e^{-c(t-s')}\Big(\|\operatorname{div}v\|_{L^{\infty}}\|\dot{\Delta}_{j}(a,\nabla a,v)\|_{L^{2}}\\
&\qquad+\|[v\cdot\nabla,\dot{\Delta}_{j}]a\|_{L^{2}}+\|[v\cdot\nabla,\dot{\Delta}_{j}]v\|_{L^{2}}+\sum_{k=1}^{d}\|[v\cdot\nabla,\partial_{x_k}\dot{\Delta}_{j}]a\|_{L^{2}}\\
&\qquad+\|\nabla\dot{\Delta}_{j}(a\operatorname{div}v)\|_{L^{2}}+\frac{1}{\tau}\|\dot{\Delta}_{j}(\rho(u-v))\|_{L^{2}}+\|\dot{\Delta}_{j}(F_{1},F_{2})\|_{L^{2}}\Big){\rm d}s',
\end{aligned}
\end{equation*}
where $c>0$ is a uniform constant. As a direct consequence, it holds for any $j\geq -1$ that
\begin{equation}\label{nsgmm1}
\begin{aligned}
&\|\langle s\rangle^{\alpha}(\nabla a,v)\|_{\widetilde{L}^{\infty}_{t}(\dot{B}^{\frac{d}{2}-1}_{2,1})}^{h}\\
&\quad\lesssim \sum_{j\geq -1}\sup_{s\in[0,t]}\langle s\rangle^{\alpha}e^{-cs}2^{j(\frac{d}{2}-1)}\|\dot{\Delta}_{j}(\nabla a_{0},v_{0})\|_{L^{2}}\\
&\qquad+\sum_{j\geq -1}\sup_{s\in[0,t]}\langle s\rangle^{\alpha}\int_{0}^{s}e^{-c(s-s')}2^{j(\frac{d}{2}-1)}\Big(\|\operatorname{div}v\|_{L^{\infty}}\|\dot{\Delta}_{j}(a,\nabla a, v)\|_{L^{2}}\\
&\qquad+\|[v\cdot\nabla,\dot{\Delta}_{j}]a\|_{L^{2}}+\|[v\cdot\nabla,\dot{\Delta}_{j}]v\|_{L^{2}}+\sum_{k=1}^{d}\|[v\cdot\nabla,\partial_{x_k}\dot{\Delta}_{j}]a\|_{L^{2}}\\
&\qquad+\|\nabla\dot{\Delta}_{j}(a\operatorname{div}v)\|_{L^{2}}+\frac{1}{\tau}\|\dot{\Delta}_{j}(\rho(u-v))\|_{L^{2}}+\|\dot{\Delta}_{j}(F_{1},F_{2})\|_{L^{2}}\Big){\rm{d}}s'.
\end{aligned}
\end{equation}
One can easily verify that
\begin{equation}
\sum_{j\geq -1}\sup_{s\in[0,t]}\langle s\rangle^{\alpha}e^{-cs}2^{j(\frac{d}{2}-1)}\|\dot{\Delta}_{j}(\nabla a_{0},v_{0})\|_{L^{2}}\lesssim \|(\nabla a_{0},v_{0})\|_{\dot{B}^{\frac{d}{2}-1}_{2,1}}^{h}\lesssim\mathcal{Y}_{0}.
\end{equation}
Then, we analyze the key estimates of the singular term $\frac{1}{\tau}\rho(u-v)$ on the right-hand side of \eqref{nsgmm1}, which involves the non-dissipative variable $\rho$. For the supremum on $[0,1]$, we have
\begin{equation}\label{nsgmm2}
\begin{aligned}
&\sum_{j\geq -1}\sup_{s\in[0,1]}\langle s\rangle^{\alpha} \int_{0}^{s} e^{-c(s-s')}2^{j(\frac{d}{2}-1)}\frac{1}{\tau}\|\dot{\Delta}_{j}(\rho(u-v))\|_{L^{2}}{\rm{d}}s'\\
&\quad\lesssim\|\rho\|_{\widetilde{L}^{\infty}_{t}(\dot{B}^{\frac{d}{2}-1}_{2,1})}\frac{1}{\tau}\|u-v\|_{L^{1}_{t}(\dot{B}^{\frac{d}{2}}_{2,1})}^{\ell}+\|\rho\|_{\widetilde{L}^{\infty}_{t}(\dot{B}^{\frac{d}{2}}_{2,1})}\frac{1}{\tau}\|u-v\|_{L^{1}_{t}(\dot{B}^{\frac{d}{2}-1}_{2,1})}^{h}\lesssim\mathcal{Y}_{0}.
\end{aligned}
\end{equation}
The part of the supremum over $[1,t]$ is split into two parts: the integration on $[0,1]$ and $[1,t]$. For the former, we also have
\begin{equation}\label{nsgmm3}
\sum_{j\geq -1}\sup_{s\in[1,t]}\langle s\rangle^{\alpha} \int_{0}^{1}e^{-c(s-s')}2^{j(\frac{d}{2}-1)}\frac{1}{\tau}\|\dot{\Delta}_{j}(\rho(u-v))\|_{L^{2}}{\rm{d}}s'\lesssim\mathcal{Y}_{0}.
\end{equation}
It remains to control the integration on $[1,s]$. In terms of \eqref{result1} and \eqref{prioridc}, one can show that
\begin{equation}\label{nsgmm4}
\begin{aligned}
&\sum_{j\geq -1}\sup_{s\in[1,t]}\langle s\rangle^{\alpha} \int_{1}^{s} e^{-c(s-s')}2^{j(\frac{d}{2}-1)}\frac{1}{\tau}\|\dot{\Delta}_{j}(\rho(u-v))\|_{L^{2}}{\rm{d}}s'\\
&\quad\lesssim\sup_{s\in[1,t]}\langle s\rangle^{\alpha} \int_{1}^{s} e^{-c(s-s')}(s')^{-\alpha}{\rm{d}}s'\frac{1}{\tau}\Big(\|s^{\alpha}\rho(u-v)^{\ell}\|_{L^{\infty}_{t}(\dot{B}^{\frac{d}{2}}_{2,1})}^{h}+\|s^{\alpha}\rho(u-v)^{h}\|_{\widetilde{L}^{\infty}_{t}(\dot{B}^{\frac{d}{2}-1}_{2,1})}^{h}\Big)\\
&\quad\lesssim\|\rho\|_{\widetilde{L}^{\infty}_{t}(\dot{B}^{\frac{d}{2}}_{2,1})}\Big(\frac{1}{\tau}\|\langle s\rangle^{\alpha}(u-v)\|_{L^{\infty}_{t}(1,t;\,\dot{B}^{\frac{d}{2}}_{2,1})}^{\ell}+\frac{1}{\tau}\|\langle s\rangle^{\alpha}(u-v)\|_{\widetilde{L}^{\infty}_{t}(1,t;\dot{B}^{\frac{d}{2}-1}_{2,1})}^{h}\Big)\lesssim\mathcal{Y}_{0}\mathcal{Y}(t),
\end{aligned}
\end{equation}
where we have noticed that
\begin{equation*}
\|s^{\alpha}\rho(u-v)^{\ell}\|_{\widetilde{L}^{\infty}_{t}(\dot{B}^{\frac{d}{2}-1}_{2,1})}^{h}\lesssim\|s^{\alpha}\rho(u-v)^{\ell}\|_{L^{\infty}_{t}(\dot{B}^{\frac{d}{2}}_{2,\infty})}^{h}\sum_{j\geq-1}2^{-j}\lesssim\|s^{\alpha}\rho(u-v)^{\ell}\|_{L^{\infty}_{t}(\dot{B}^{\frac{d}{2}}_{2,1})}^{h}.
\end{equation*}
This is used to overcome the difficulty that in low frequencies we only have the norm  $L^{\infty}_t(\dot{B}^{\sigma}_{2,1})$, which is weaker than the Chemin-Lerner norm $\widetilde{L}^{\infty}_t(\dot{B}^{\sigma}_{2,1})$. Similarly, it holds
\begin{equation}
\sum_{j\geq -1}\sup_{s\in[0,t]}\langle s\rangle^{\alpha} \int_{0}^{s} e^{-c(s-s')}2^{j(\frac{d}{2}-1)}\|\dot{\Delta}_{j}F_{2}\|_{L^{2}}{\rm{d}}s'\lesssim\mathcal{Y}_{0}+\mathcal{Y}_{0}\mathcal{Y}(t).
\end{equation}
Since the other remaining terms on the right-hand side of \eqref{nsgmm1} can be treated similarly as in \cite{lishouSIAMJMA23}, we omit the details and give the result:
\begin{equation}\label{nsgmm5}
\begin{aligned}
&\sum_{j\geq -1}\sup_{s\in[0,t]}\langle s\rangle^{\alpha}\int_{0}^{s}e^{-c(s-s')}2^{j(\frac{d}{2}-1)}\Big(\|\operatorname{div}v\|_{L^{\infty}}\|\dot{\Delta}_{j}(a,\nabla a, v)\|_{L^{2}}+\|[v\cdot\nabla,\dot{\Delta}_{j}]a\|_{L^{2}}\\
&\quad+\|[v\cdot\nabla,\dot{\Delta}_{j}]v\|_{L^{2}}+\sum_{k=1}^{d}\|[v\cdot\nabla,\partial_{x_k}\dot{\Delta}_{j}]a\|_{L^{2}}+\|\nabla\dot{\Delta}_{j}(a\operatorname{div}v)\|_{L^{2}}+\|\dot{\Delta}_{j}F_{1}\|_{L^{2}}\Big){\rm{d}}s'\\
&\quad\lesssim\mathcal{Y}_{0}+\mathcal{Y}(t)^{2}.
\end{aligned}
\end{equation}
Combining \eqref{nsgmm1}-\eqref{nsgmm5}, we conclude that
\begin{equation}\label{nsgmm6}
\|\langle s \rangle^{\alpha}(\nabla a,v)\|_{\widetilde{L}^{\infty}_{t}(\dot{B}^{\frac{d}{2}-1}_{2,1})}^{h}\lesssim\mathcal{Y}_{0}+\mathcal{Y}_0\mathcal{Y}(t)+\mathcal{Y}(t)^{2}.
\end{equation}

Meanwhile, we establish the higher-order decay estimates of $v$. Multiplying $\eqref{vheat}$ by $t^{\alpha}$, we get
\begin{equation}\label{nsvgm1}
\left\{\begin{aligned}
&\partial_{t}(t^{\alpha}v)-\Delta (t^{\alpha}v)=\alpha t^{\alpha-1} v - t^{\alpha}\nabla a-t^{\alpha}v\cdot\nabla v+\frac{1}{\tau}t^{\alpha}\rho(u-v)+t^{\alpha}F_{1}+t^{\alpha}F_{2},\\
&(t^{\alpha}v)|_{t=0}=0.
\end{aligned}\right.
\end{equation}
The maximal regularity estimate in Lemma \ref{esheat} for \eqref{nsvgm1} leads to
\begin{equation}\label{mm}
\begin{aligned}
&\|s^{\alpha}v\|_{\widetilde{L}^{\infty}_{t}(\dot{B}^{\frac{d}{2}+1}_{2,1})}^{h}\\
&\quad\lesssim \|s^{\alpha-1}v\|_{\widetilde{L}^{\infty}_{t}(\dot{B}^{\frac{d}{2}-1}_{2,1})}^{h}+\|s^{\alpha}\nabla a\|_{\widetilde{L}^{\infty}_{t}(\dot{B}^{\frac{d}{2}-1}_{2,1})}^{h}+\|s^{\alpha}(v\cdot\nabla v,\tau^{-1}\rho(u-v),F_{1},F_{2})\|_{\widetilde{L}^{\infty}_{t}(\dot{B}^{\frac{d}{2}-1}_{2,1})}^{h}\\
&\quad\lesssim \|\langle s\rangle^{\alpha}(\nabla a,v)\|_{\widetilde{L}^{\infty}_{t}(\dot{B}^{\frac{d}{2}-1}_{2,1})}^{h}+\|s^{\alpha}(v\cdot\nabla v,\tau^{-1}\rho(u-v),F_{1},F_{2})\|_{\widetilde{L}^{\infty}_{t}(\dot{B}^{\frac{d}{2}-1}_{2,1})}^{h},
\end{aligned}
\end{equation}
where we have noticed the fact $\alpha>1$. Substituting \eqref{nsgmm6} into \eqref{mm}, we can verify that
\begin{equation}\label{nsvgm2}
\|s^{\alpha}v\|_{\widetilde{L}^{\infty}_{t}(\dot{B}^{\frac{d}{2}+1}_{2,1})}^{h}\lesssim\mathcal{Y}_{0}+\mathcal{Y}_0\mathcal{Y}(t)+\mathcal{Y}(t)^{2}.
\end{equation}

With the aid of \eqref{nsgmm6} and \eqref{nsvgm2}, we establish the decay estimates of $u$ in high frequencies. Employing the Gr\"onwall inequality to \eqref{veu2}, we obtain
\begin{equation}
\begin{aligned}\label{ujfff}
\|\dot{\Delta}_ju\|_{L^2}&\lesssim e^{-\frac{1}{\tau}t}\|\dot{\Delta}_j u_0\|_{L^2}\\
&\quad+\int_0^te^{-\frac{1}{\tau}(t-s')} \Big( \frac{1}{\tau}\|\dot{\Delta}_jv\|_{L^2}+\|\operatorname{div} u\|_{L^{\infty}}\|\dot{\Delta}_ju\|_{L^2}+\|[u\cdot \nabla,\dot{\Delta}_j]u\|_{L^2}\Big){\rm{d}}s'.
\end{aligned}
\end{equation}
Multiplying \eqref{ujfff} by $\langle s\rangle^{\alpha}$, taking the supremum over $[1,t]$ and then summing it over $j\geq -1$ with the weight $2^{j(\frac{d}{2}+1)}$, we arrive at
\begin{equation}
\begin{aligned}
&\|\langle s\rangle^{\alpha}u\|_{\widetilde{L}^{\infty}_{t}(1,t;\dot{B}^{\frac{d}{2}+1}_{2,1})}^h\\
&\quad\lesssim \sum_{j\geq-1}\sup_{s\in[1,t]}\langle s\rangle^{\alpha}e^{-\frac{1}{\tau}s}2^{j(\frac{d}{2}+1)}\|\dot{\Delta}_{j}u_{0}\|_{L^{2}}\\
&\qquad+\sum_{j\geq-1}\sup_{s\in[1,t]}\langle s\rangle^{\alpha}\int_{0}^{s}e^{-\frac{1}{\tau}(s-s')}2^{j(\frac{d}{2}+1)}\frac{1}{\tau}\|\dot{\Delta}_{j} v\|_{L^{2}}{\rm{d}}s'\\
&\qquad+\sum_{j\geq-1}\sup_{s\in[1,t]}\langle s\rangle^{\alpha}\int_{0}^{s}e^{-\frac{1}{\tau}(s-s')}2^{j(\frac{d}{2}+1)}\Big( \|\operatorname{div} u\|_{L^{\infty}}\|\dot{\Delta}_ju\|_{L^2}+\|[u\cdot \nabla,\dot{\Delta}_j]u\|_{L^2}\Big){\rm{d}}s'.
\end{aligned}
\end{equation}
Recalling the key observations in \eqref{epntep1}-\eqref{expvar}, we derive from the weighted estimate \eqref{nsvgm2} and the standard commutator estimate (see Lemma \ref{commutator}) that
\begin{equation}\label{uh}
\begin{aligned}
\|\langle s\rangle^{\alpha}u\|_{\widetilde{L}^{\infty}_{t}(1,t;\dot{B}^{\frac{d}{2}+1}_{2,1})}^h&\lesssim \tau\|u_0\|_{\dot{B}^{\frac{d}{2}+1}_{2,1}}^h+\|\langle s\rangle^{\alpha}v\|_{\widetilde{L}^{\infty}_{t}(1,t;\dot{B}^{\frac{d}{2}+1}_{2,1})}^h+\mathcal{Y}^2(t)\\
&\lesssim \mathcal{Y}_{0}+\mathcal{Y}_0\mathcal{Y}(t)+\mathcal{Y}(t)^{2}.
\end{aligned}
\end{equation}

Then, we address the relative velocity $u-v$ in high frequencies. By taking the $\dot{B}^{\frac{d}{2}-1}_{2,1}$-norm of \eqref{rvgm1} with the weight $\langle s\rangle^{\alpha}$ for $s\in[1,t]$, we have
\begin{align}\label{veugm1}
\begin{aligned}
&\|\langle s\rangle^{\alpha}(u-v)\|_{\widetilde{L}^{\infty}_{t}(1,t;\dot{B}^{\frac{d}{2}-1}_{2,1})}^{h}\\
&\quad \lesssim\sum_{j\geq-1}\sup_{s\in[1,t]}\langle s\rangle^{\alpha}e^{-\frac{1}{\tau}s}2^{j(\frac{d}{2}-1)}\|\dot{\Delta}_{j}(u_{0}-v_{0})\|_{L^{2}}\\
&\quad+\sum_{j\geq-1}\sup_{s\in[1,t]}\langle s\rangle^{\alpha}\int_{0}^{s}e^{-\frac{1}{\tau}(s-s')}2^{j(\frac{d}{2}-1)}\|\dot{\Delta}_{j}(\nabla a,\nabla^2 v)\|_{L^{2}}{\rm{d}}s'\\
&\quad+\sum_{j\geq-1}\sup_{s\in[1,t]}\langle s\rangle^{\alpha}\int_{0}^{s}e^{-\frac{1}{\tau}(s-s')}2^{j(\frac{d}{2}-1)}\|\dot{\Delta}_{j}(u\cdot\nabla u,v\cdot\nabla v,\tau^{-1}\rho(u-v),F_{1},F_{2})\|_{L^{2}}{\rm{d}}s'.
\end{aligned}
\end{align}
It follows from \eqref{epntep1}-\eqref{expvar} that
\begin{equation*}
\sum_{j\geq-1}\sup_{s\in[1,t]}\langle s\rangle^{\alpha}e^{-\frac{1}{\tau}s}2^{j(\frac{d}{2}-1)}\|\dot{\Delta}_{j}(u_{0}-v_{0})\|_{L^{2}}\lesssim \tau\|(u_{0},v_{0})\|_{\dot{B}^{\frac{d}{2}-1}_{2,1}}^h\lesssim \tau\mathcal{Y}_{0}.
\end{equation*}
Due to \eqref{nsgmm6} and \eqref{nsvgm2}, we also have the following estimate:
\begin{equation*}
\begin{aligned}
&\sum_{j\geq-1}\sup_{s\in[1,t]}\langle s\rangle^{\alpha}\int_{0}^{s}e^{-\frac{1}{\tau}(s-s')}2^{j(\frac{d}{2}-1)}\|\dot{\Delta}_{j}(\nabla a,\nabla^2 v)\|_{L^{2}}{\rm{d}}s'	\\
&\quad\lesssim \tau\|\langle s\rangle^{\alpha} a\|_{\widetilde{L}^{\infty}_t(\dot{B}^{\frac{d}{2}}_{2,1})}^h+\tau\|\langle s\rangle^{\alpha} v\|_{\widetilde{L}^{\infty}_{t}(1,t;\dot{B}^{\frac{d}{2}+1}_{2,1})}^h\lesssim \tau\mathcal{Y}_{0}+\tau\mathcal{Y}_0\mathcal{Y}(t)+\tau\mathcal{Y}(t)^{2}.
\end{aligned}
\end{equation*}
Since $\tau<1$, the remaining terms in \eqref{veugm1} can be estimated by a similar argument as \eqref{nsgmm2}-\eqref{nsgmm4}, and thus we end up with
\begin{equation}
\frac{1}{\tau}\|\langle s\rangle^{\alpha}(u-v)\|_{\widetilde{L}^{\infty}_{t}(1,t;\dot{B}^{\frac{d}{2}-1}_{2,1})}^{h}\lesssim \mathcal{Y}_{0}+\mathcal{Y}_0\mathcal{Y}(t)+\mathcal{Y}(t)^{2}.\label{relativehighdecay}
\end{equation}
Collecting \eqref{nsgmm6}, \eqref{nsvgm2}, \eqref{uh}, and \eqref{relativehighdecay} together, we obtain
\begin{equation}\label{highdecay}
\begin{aligned}
&\|\langle s\rangle^{\alpha}(u,v)\|_{\widetilde{L}^{\infty}_{t}(1,t;\dot{B}^{\frac{d}{2}+1}_{2,1})}^{h}+\|\langle s\rangle^{\alpha}a\|_{\widetilde{L}^{\infty}_{t}(\dot{B}^{\frac{d}{2}}_{2,1})}^{h}+\frac{1}{\tau}\|\langle s\rangle^{\alpha}(u-v)\|_{\widetilde{L}^{\infty}_{t}(1,t;\dot{B}^{\frac{d}{2}-1}_{2,1})}^{h}\\
&\quad\lesssim \mathcal{Y}_{0}+\mathcal{Y}_0\mathcal{Y}(t)+\mathcal{Y}(t)^{2}.
\end{aligned}
\end{equation}

\subsection{Proof of Theorem \ref{theorem2}}
In conclusion, under the definition  \eqref{prioridc} of $\mathcal{Y}(t)$, the combination of \eqref{lowtime} and \eqref{highdecay} gives rise to
\begin{equation}
\mathcal{Y}(t)\lesssim \mathcal{Y}_{0}+\mathcal{Y}_{0}\mathcal{Y}(t)+\mathcal{Y}(t)^{2}.\label{sff}
\end{equation}
According to \eqref{idcd1}, \eqref{idcd2},  \eqref{sff}, and the standard bootstrap argument, we are able to deduce that
\begin{equation}\label{eclsyt}
\mathcal{Y}(t)\leq C_{0}\mathcal{Y}_0,
\end{equation}
where $C_{0}>0$ is a positive constant independent of the time and friction parameter $\tau$. Consequently, the estimate \eqref{eclsyt} implies for $t\geq1$ that
\begin{equation}\label{ammmf1}
\begin{aligned}
\|(u,v)\|_{\dot{B}^{\sigma}_{2,1}}\lesssim \|(u,v)\|_{\dot{B}^{\sigma}_{2,1}}^\ell+\|(u,v)\|_{\dot{B}^{\frac{d}{2}+1}_{2,1}}^h\lesssim \langle t\rangle^{-\frac{1}{2}(\sigma-\sigma_1)}\mathcal{Y}_0,\quad \sigma\in(\sigma_1,\frac{d}{2}+1],
\end{aligned}
\end{equation}
and
\begin{equation}\label{ammmf}
\begin{aligned}
\|a\|_{\dot{B}^{\sigma}_{2,1}}\lesssim \|a\|_{\dot{B}^{\sigma}_{2,1}}^\ell+\|a\|_{\dot{B}^{\frac{d}{2}}_{2,1}}^h\lesssim \langle t\rangle^{-\frac{1}{2}(\sigma-\sigma_1)}\mathcal{Y}_0,\quad \sigma\in(\sigma_1,\frac{d}{2}].
\end{aligned}
\end{equation}
In addition, the relative velocity $u-v$ has improved estimates as
\begin{align}
&\|u-v\|_{\dot{B}^{\sigma_1}_{2,\infty}}\lesssim \|u-v\|_{\dot{B}^{\sigma_1}_{2,\infty}}^\ell+\|u-v\|_{\dot{B}^{\frac{d}{2}-1}_{2,1}}^h\lesssim \tau\langle t\rangle^{-\frac{1}{2}}\mathcal{Y}_0,\label{uv1f}\\
&\|u-v\|_{\dot{B}^{\sigma}_{2,1}}\lesssim \|u-v\|_{\dot{B}^{\sigma}_{2,1}}^\ell+\|u-v\|_{\dot{B}^{\frac{d}{2}-1}_{2,1}}^h\lesssim \tau\langle t\rangle^{-\frac{1}{2}(\sigma-\sigma_1+1)}\mathcal{Y}_0,\quad \sigma\in(\sigma_1,\frac{d}{2}-1],\label{uv2f}
\end{align}
and
\begin{equation}
\|u-v\|_{\dot{B}^{\sigma}_{2,1}}\lesssim \|u-v\|_{\dot{B}^{\sigma}_{2,1}}^\ell+\|(u,v)\|_{\dot{B}^{\frac{d}{2}+1}_{2,1}}^h\lesssim \langle t\rangle^{-\frac{1}{2}(\sigma-\sigma_1+1)}\mathcal{Y}_0,\quad \sigma\in(\frac{d}{2}-1,\frac{d}{2}].\label{uv3f}
\end{equation}
Hence, the combination of \eqref{ammmf1}-\eqref{uv3f} yields that the desired estimates in \eqref{result2} hold.

Finally, we show the convergence of $\rho(t,x)$ to its equilibrium state $\rho_\infty(x)$ for $d\geq 3$. Define the asymptotic profile $\rho_{\infty}(x)$ by \eqref{aspf}. In view of $\eqref{pENS}_1$ and \eqref{aspf}, $\rho(t,x)$ can be represented by \eqref{rhominusrhoinfty}, and thus, applying the product law \eqref{uv2} since $\sigma+1\leq\frac{d}{2}$, it holds for $\sigma\in(\sigma_{1}+1,\frac{d}{2}-1]$ that
\begin{equation*}
\|\rho-\rho_{\infty}\|_{\dot{B}^{\sigma}_{2,1}}\leq\int_{t}^{\infty}\|\rho u\|_{\dot{B}^{\sigma+1}_{2,1}}{\rm d}s\lesssim \|\rho \|_{L^{\infty}(t,\infty;\dot{B}^{\frac{d}{2}}_{2,1})} \int_{t}^{\infty}  \|u\|_{\dot{B}^{\sigma+1}_{2,1}}{\rm d}s.
\end{equation*}
This, together with the uniform bound for $\rho$ in \eqref{result1} and the decay estimates of $u$ in \eqref{result2}, leads to
\begin{equation*}
\begin{aligned}
\|\rho-\rho_{\infty}\|_{\dot{B}^{\sigma}_{2,1}}&\leq\int_{t}^{\infty} \langle s\rangle^{-\frac{1}{2}(\sigma-\sigma_1+1)}{\rm d}s\lesssim \langle t\rangle^{-\frac{1}{2}(\sigma-\sigma_1-1)},\quad \sigma\in(\sigma_{1}+1,\frac{d}{2}-1],
\end{aligned}
\end{equation*}
where we have used the fact $\frac{1}{2}(\sigma+1-\sigma_1)>1$ for $\sigma>\sigma_1+1$, and the desired estimate \eqref{result3} is obtained. The proof of Theorem \ref{theorem2} is complete.

\section{Incompressible limits}\label{section:incompressible} 
In this section, we prove Theorems \ref{Theoremincompressible1} and \ref{Theoremincompressible2} concerning the incompressible limits from the DF model \eqref{DF} and the Euler-NS system \eqref{pENS} to the TNS system \eqref{TNS}.

\subsection{Convergence of the DF system to the TNS system}\label{subsection:incompressible1}

In the case $\var=1$, let $(\rho,n,v)$ be the global solution to the DF system \eqref{DF} subject to the initial data $(\rho_0,n_0,v_0)$ (replacing the coefficients $\mu$ and $\lambda$ by $\bar{\mu}$ and $\bar{\lambda}$, respectively). Such a solution exists and satisfies \eqref{result5}, which has been addressed in Corollary \ref{corollary1}. Inspired by \cite{danchinMA16}, we perform the change of the unknowns
\begin{align}\label{scaling2}
(\rho,n-1,v)(t,x)=\var(\rho^\var,a^\var,v^\var)(\var^2 t,\var x),
\end{align}
and the change of the initial data
\begin{align}\label{scaling20}
(\rho_0,n_0-1,v_0)(x)=\var(\rho^\var_0,a^\var_0,v_0^\var)(\var x).
\end{align}
Then, one can verify that $(\rho^\var,a^\var,v^\var)$ solves the scaled DF system  \eqref{DFvar} for $\var\in(0,1)$. Indeed, with the scaling properties of Besov norms (see Lemma \ref{lember}), we have
\begin{align*}
\|(\rho,n-1,v) \|_{\dot{B}^{\frac{d}{2}-1}_{2,1}}&\sim \|(\rho^\var,a^\var,v^\var)\|_{\dot{B}^{\frac{d}{2}-1}_{2,1}},\\
\|(\rho,n-1) \|_{\dot{B}^{\frac{d}{2}}_{2,1}}&\sim \var \|(\rho^\var,a^\var)\|_{\dot{B}^{\frac{d}{2}}_{2,1}},\\
\|n-1\|_{L^{1}(\mathbb{R}_+;\dot{B}^{\frac{d}{2}+1}_{2,1}+\dot{B}^{\frac{d}{2}}_{2,1})}
&\sim\|a^\var\|_{L^{1}(\mathbb{R}_+;\dot{B}^{\frac{d}{2}+1}_{2,1})}^{\var,\ell}+\frac{1}{\var}\|a^\var\|_{L^{1}(\mathbb{R}_+;\dot{B}^{\frac{d}{2}}_{2,1})}^{\var,h},\\
\|v\|_{L^{1}(\mathbb{R}_+;\dot{B}^{\frac{d}{2}+1}_{2,1})}&\sim \|v^\var\|_{L^{1}(\mathbb{R}_+;\dot{B}^{\frac{d}{2}+1}_{2,1})}.
\end{align*}
Thus, under the condition \eqref{Incompress:a1}, the global existence and uniform regularity estimates \eqref{Incompress:r1}-\eqref{Incompress:r11} are direct consequences of Corollary \ref{corollary1} and the transformations \eqref{scaling2}-\eqref{scaling20}.

We now justify the incompressible limit by a compactness argument based on \eqref{Incompress:r1}-\eqref{Incompress:r11}. The strong convergence of $n^\var$ is clear: since $a^\var$ is uniformly bounded in $L^{\infty}(\mathbb{R}_+;\dot{B}^{\frac{d}{2}-1}_{2,1})\cap L^2(\mathbb{R}_+;\dot{B}^{\frac{d}{2}}_{2,1})$, which implies
\begin{align}
\|n^\var-1\|_{L^{\infty}(\mathbb{R}_+;\dot{B}^{\frac{d}{2}-1}_{2,1})\cap L^2(\mathbb{R}_+;\dot{B}^{\frac{d}{2}}_{2,1})}\lesssim \var.
\end{align}
To justify the convergence of nonlinear terms, we need some strong compactness properties of $\rho^\var$ and $v^\var$. Since $v^\var$ is uniformly bounded in $L^2_t(\dot{B}^{\frac{d}{2}}_{2,1})\hookrightarrow L^2_t(L^{\infty})$, we can find a limit $w$ such that, up to a subsequence, as $\var\rightarrow0$, it holds
\begin{equation}\label{limit1:in}
\left\{
\begin{aligned}
&v^\var \rightharpoonup w && \text{weakly$^*$ in}~~ L^{2}(\mathbb{R}_{+};L^{\infty})\quad\text{for}~~d\geq3,\\
&v^\var \rightharpoonup w && \text{weakly\,\,\, in}~~ L^{2}(\mathbb{R}_{+};\dot{H}^{1})\quad\,\text{for}~~d=2.
\end{aligned}
\right.
\end{equation}
Obviously, $\rho^\var$ and $\partial_t \rho^\var=-\operatorname{div}(\rho^\var v^\var)$ are uniformly bounded in $L^{\infty}_t(\dot{B}^{\frac{d}{2}-1}_{2,1})$ and $L^2_t(\dot{B}^{\frac{d}{2}-2}_{2,1})$, respectively, and we have $\dot{B}^{\frac{d}{2}-1}_{2,1}\hookrightarrow\!\!\!\!\rightarrow  L^2_{{\rm{loc}}}$ for $d\geq3$ and $\dot{B}^{\frac{d}{2}-1}_{2,1}\hookrightarrow\!\!\!\!\rightarrow  H^{-1}_{{\rm{loc}}}$ for $d=2$. Then the standard compactness techniques (the Aubin-Lions lemma and the Cantor diagonal argument) imply that, as $\var\rightarrow0$, a limit $\varrho$ exists such that, up to a subsequence,
\begin{equation}
\left\{
\begin{aligned}\label{limit2:in}
&\rho^\var\rightarrow \varrho && \text{strongly in}~~ C([0,T];L^2_{{\rm{loc}}})\quad\, \text{for}~~ d\geq 3,\\
&\rho^\var\rightarrow \varrho && \text{strongly in}~~ C([0,T];H^{-1}_{{\rm{loc}}})\quad \text{for}~~ d=2,
\end{aligned}
\right.
\end{equation}
for any finite time $T>0$. In view of $\eqref{DFvar}_1$, \eqref{limit1:in}, and \eqref{limit2:in}, the equation $\eqref{TNS}_1$ holds in the distributional sense. According to $\eqref{DFvar}_2$, i.e., $\operatorname{div} v^\var=-\var \partial_t a^\var-\var\operatorname{div}(a^\var v^\var)$, we derive that, up to a subsequence, $\operatorname{div} v^\var\rightharpoonup 0$, or equivalently, $\mathcal{Q}v^\var\rightarrow0$ in $\mathcal{D}'(\mathbb{R}_+\times\mathbb{R}^d)$ as $\var\rightarrow0$.

Let us prove the convergence of $\eqref{DFvar}_3$ toward $\eqref{TNS}_2$. Compared with the incompressible limit of the isentropic compressible flows (cf. \cite{danchin02,danchinMA16}), the main difficulty lies in the presence of $\rho^\var$. Indeed, we cannot eliminate the pressure term after applying the projector $\mathcal{P}$ to the equation of velocity. To overcome this difficulty, we focus on the momentum equation of $m^\var:=(1+\var \rho^\var+\var a^\var)v^\var$ and obtain
\begin{align}\label{6.6}
&\partial_t \mathcal{P}m^\var+\mathcal{P}\operatorname{div}\big( (1+\var \rho^\var+\var a^\var)v^\var\otimes v^\var\big)=\bar\mu\Delta \mathcal{P}v^\var.
\end{align} 
Note that $\mathcal{P}m^\var$ is uniformly bounded in $L^2(\mathbb{R}_+;\dot{B}^{\frac{d}{2}}_{2,1})$, and one can check that $\partial_t \mathcal{P}m^\var$ is uniformly bounded in $L^1(\mathbb{R}_+;\dot{B}^{\frac{d}{2}-1}_{2,1})$. Consequently, $\mathcal{P}m^\var$ is strongly compact in $L^2(0,T;H^{\frac{d}{2}- \eta_*}_{\rm{loc}})$ for any $T>0$ and $\eta_*\in(0,1)$. Since $\mathcal{P}m^\var=\mathcal{P}v^\var+\var\mathcal{P}(\rho^\var v^\var+a^\var v^\var)$ and noticing that $\rho^\var v^\var+a^\var v^\var$ is uniformly bounded in $L^2(0,T;\dot{B}^{\frac{d}{2}-1}_{2,1})$, after extraction, we have
\begin{align}\label{limit3:in}
\mathcal{P}m^\var \rightarrow w \quad\text{and}\quad \mathcal{P}v^\var \rightarrow w \qquad \text{strongly in}~~ L^2(0,T;H^{\frac{d}{2}- \eta_*}_{\rm{loc}}).
\end{align}
The remainder on the left-hand side of \eqref{6.6} holds
\begin{align*}
\mathcal{P}\operatorname{div}\big( (1+\var \rho^\var+\var a^\var)v^\var\otimes v^\var\big) \rightharpoonup w\cdot\nabla w \quad  \text{in}~~ \mathcal{D}'(\mathbb{R}_+\times\mathbb{R}^d).
\end{align*}
To address this, we observe that
\begin{equation}\label{sfgv}
\mathcal{P}\operatorname{div}\big((1+\var \rho^\var+\var a^\var)v^\var\otimes v^\var\big)=\var \mathcal{P}\operatorname{div} \big( (\rho^\var+a^\var)v^\var\otimes v^\var\big)+\mathcal{P}(v^\var\cdot\nabla v^\var)+\mathcal{P}(\operatorname{div} v^\var v^\var),
\end{equation}
where the first term on the right-hand side converges to $0$ in the distributional sense due to the uniform regularity estimates of $(\rho^\var, a^\var,v^\var)$. For the second term, one has
\begin{align*}
v^\var\cdot \nabla v^\var =\frac{1}{2}\nabla |\mathcal{Q}v^\var|^2+\mathcal{P} v^\var \cdot \nabla v^\var+\mathcal{Q}v^\var \cdot \nabla \mathcal{P} v^\var,
\end{align*}
which implies $\mathcal{P}(v^\var\cdot\nabla v^\var)=\mathcal{P}(\mathcal{P} v^\var \cdot \nabla v^\var)+\mathcal{P}(\mathcal{Q}v^\var \cdot \nabla \mathcal{P} v^\var)$. Thus, using \eqref{limit2:in} and \eqref{limit3:in}, we have 
\begin{align}\label{6.11}
\mathcal{P}(v^\var\cdot\nabla v^\var)\rightharpoonup w\cdot\nabla w\quad  \text{in}~~ \mathcal{D}'(\mathbb{R}_+\times\mathbb{R}^d).
\end{align}
In order to prove $\mathcal{P}(\operatorname{div}  v^\var v^\var)\rightharpoonup 0$ in \eqref{sfgv}, we need to establish the strong convergence of $\operatorname{div}v^\var$ to $0$. Hence, we aim to establish the dispersive estimates for $(a^\var,\mathcal{Q}v^\var)$. Denote $\mathbf{v}^\var=\Lambda^{-1}\operatorname{div} \mathcal{Q}v^\var$, we rewrite $\eqref{DFvar}_2$ and $\eqref{DFvar}_3$ as
\begin{equation}\label{ad}
\left\{
\begin{aligned}
&\partial_t a^\var+\frac{1}{\var} \Lambda \mathbf{v}^\var=\mathbf{F}_1,\\
&\partial_t \mathbf{v}^\var-\frac{1}{\var}\Lambda a^\var=\mathbf{F}_2,\\
&(a^\var, \mathbf{v}^\var)(0,x)=(a_0^\var,\Lambda^{-1}\operatorname{div} \mathcal{Q}v_0^\var), 
\end{aligned}
\right.
\end{equation}
where $\mathbf{F}_1$ and $\mathbf{F}_2$ satisfy
\begin{align*}
\mathbf{F}_1&:=-\operatorname{div} (a^\var v^\var),\\
\mathbf{F}_2&:=-\Lambda^{-1}\operatorname{div}\Big( v^\var\cdot\nabla v^\var+\frac{\var^{-1}(P'(1+\var a^\var)-1)-\rho^\var-a^\var}{1+\var\rho^\var+\var a^\var}\nabla a^\var\\
&\qquad\qquad\qquad\qquad+\frac{\var \rho^\var+\var a^\var}{1+\var \rho^\var+\var a^\var}(\bar\mu\Delta v^\var+(\bar\mu+\bar\lambda)\nabla\operatorname{div} v^\var)\Big)+(2\bar\mu+\bar\lambda)\Delta v^\var.
\end{align*}

In the case $d\geq3$, for any $q\in(2,\infty)$, applying the dispersive inequalities presented in Lemma \ref{lemmadispersive} to \eqref{ad}, we deduce that
\begin{equation}\label{ssfbbb}
\begin{aligned}
\|(a^\var, \mathbf{v}^\var)\|_{\widetilde{L}^\frac{2q}{q-2}(\mathbb{R}_+;\dot{B}^{\frac{d-1}{q}-\frac{1}{2}}_{q,1})}\lesssim \var^{\frac{1}{2}-\frac{1}{q}} \|(a_0^\var, \mathbf{v}_0^\var)\|_{\dot{B}^{\frac{d}{2}-1}_{2,1}}+\var^{\frac{1}{2}-\frac{1}{q}}\|(\mathbf{F}_1,\mathbf{F}_2)\|_{L^{1}(\mathbb{R}_+;\dot{B}^{\frac{d}{2}-1}_{2,1})}.
\end{aligned} 
\end{equation}
By similar computations used in the proof of Lemma \ref{lemav}, the nonlinear terms have the following estimates
\begin{equation*}
\begin{aligned}
\|(\mathbf{F}_1,\mathbf{F}_2)\|_{L^{1}(\mathbb{R}_+;\dot{B}^{\frac{d}{2}-1}_{2,1})}\lesssim 1.
\end{aligned} 
\end{equation*}
Since $\Lambda^{-1}\operatorname{div}$ is a pseudo-differential operator of degree $0$, we arrive at
\begin{equation}\label{converLq}
\begin{aligned}
\|(a^\var, \mathcal{Q}v^\var)\|_{\widetilde{L}^\frac{2q}{q-2}(\mathbb{R}_+;\dot{B}^{\frac{d-1}{q}-\frac{1}{2}}_{q,1})}\lesssim \var^{\frac{1}{2}-\frac{1}{q}}.
\end{aligned} 
\end{equation}
It then follows from \eqref{Incompress:r11}, \eqref{converLq}, and \eqref{inter} that
\begin{equation*}
\begin{aligned}
&\|(a^\var,\mathcal{Q}v^\var)\|_{\widetilde{L}^{2}(\mathbb{R}_+;\dot{B}^{\frac{d+1}{p}-\frac{1}{2}}_{p,1})}^{\var,\ell}\\
&\quad\lesssim \Big(\|(a^\var, \mathcal{Q}v^\var)\|_{\widetilde{L}^\frac{2q}{q-2}(\mathbb{R}_+;\dot{B}^{\frac{d-1}{q}-\frac{1}{2}}_{q,1})}^{\var,\ell}\Big)^{\frac{q}{q+2}} \Big(\|(a^\var, \mathcal{Q}v^\var)\|_{L^1(\mathbb{R}_+;\dot{B}^{\frac{d}{2}+1}_{2,1})}^{\var,\ell}\Big)^{\frac{2}{q+2}}\lesssim  \var^{\frac{1}{2}-\frac{1}{p}},
\end{aligned} 
\end{equation*}
for $p=\frac{q+2}{2}\in(2,\infty)$, which gives rise to
\begin{equation*}
\begin{aligned}
&\|(a^\var, \mathcal{Q}v^\var)\|_{\widetilde{L}^{2}(\mathbb{R}_+;\dot{B}^{\frac{d+1}{p}-\frac{1}{2}}_{p,1})}\\
&\quad\lesssim \|(a^\var, \mathcal{Q}v^\var)\|_{\widetilde{L}^{2}(\mathbb{R}_+;\dot{B}^{\frac{d+1}{p}-\frac{1}{2}}_{p,1})}^{\var,\ell}+\|(a^\var, \mathcal{Q}v^\var)\|_{\widetilde{L}^{2}(\mathbb{R}_+;\dot{B}^{\frac{d+1}{p}-\frac{1}{2}}_{p,1})}^{\var,h}\\
&\quad\lesssim \|(a^\var, \mathcal{Q}v^\var)\|_{\widetilde{L}^{2}(\mathbb{R}_+;\dot{B}^{\frac{d+1}{p}-\frac{1}{2}}_{p,1})}^{\var,\ell}+\var^{\frac{1}{2}-\frac{1}{p}}\|(a^\var, \mathcal{Q}v^\var)\|_{\widetilde{L}^{2}(\mathbb{R}_+;\dot{B}^{\frac{d}{p}}_{p,1})}^{\var,h}\lesssim \var^{\frac{1}{2}-\frac{1}{p}}.
\end{aligned} 
\end{equation*}

In the case $d=2$, similarly, for $q\in(2,\infty]$, we have
\begin{align*}
&\|(a^\var,\mathbf{v}^\var)\|_{\widetilde{L}^\frac{4q}{q-2}(\mathbb{R}_+;\dot{B}^{\frac{3}{2q}-\frac{3}{4}}_{q,1})} \lesssim \var^{\frac{1}{4}-\frac{1}{2q}} \|(a_0^\var, \mathbf{v}_0^\var)\|_{\dot{B}^{\frac{d}{2}-1}_{2,1}}+\var^{\frac{1}{4}-\frac{1}{2q}}\|(\mathbf{F}_1,\mathbf{F}_2)\|_{L^{1}(\mathbb{R}_+;\dot{B}^{\frac{d}{2}-1}_{2,1})}\lesssim \var^{\frac{1}{4}-\frac{1}{2q}}.
\end{align*}
Adapting the interpolation argument for $p=\frac{6q+4}{q+6}\in(2,6]$, we can verify that
\begin{equation*}
\begin{aligned}
&\|(a^\var, \mathcal{Q}v^\var)\|_{\widetilde{L}^{2}(\mathbb{R}_+;\dot{B}^{\frac{5}{2p}-\frac{1}{4}}_{p,1})}^{\var,\ell}\\
&\quad\lesssim\Big(\|(a^\var, \mathcal{Q}v^\var)\|_{\widetilde{L}^\frac{4q}{q-2}(\mathbb{R}_+;\dot{B}^{\frac{3}{2q}-\frac{3}{4}}_{q,1})}^{\var,\ell}\Big)^{\frac{2q}{3q+2}} \Big(\|(a^\var, \mathcal{Q}v^\var)\|_{L^1(\mathbb{R}_+;\dot{B}^{\frac{d}{2}+1}_{2,1})}^{\var,\ell}\Big)^{\frac{q+2}{3q+2}}\lesssim \var^{\frac{1}{4}-\frac{1}{2p}}.
\end{aligned} 
\end{equation*}
Consequently, it holds that
\begin{equation*}
\begin{aligned}
&\|(a^\var, \mathcal{Q}v^\var)\|_{\widetilde{L}^{2}(\mathbb{R}_+;\dot{B}^{\frac{5}{2p}-\frac{1}{4}}_{p,1})}\\
&\quad\lesssim \|(a^\var, \mathcal{Q}v^\var)\|_{\widetilde{L}^{2}(\mathbb{R}_+;\dot{B}^{\frac{5}{2p}-\frac{1}{4}}_{p,1})}^{\var,\ell}+\var^{\frac{1}{4}-\frac{1}{2p}} \|(a^\var, \mathcal{Q}v^\var)\|_{\widetilde{L}^{2}(\mathbb{R}_+;\dot{B}^{\frac{2}{p}}_{p,1})}^{\var,h}\lesssim \var^{\frac{1}{4}-\frac{1}{2p}}.
\end{aligned} 
\end{equation*}
Combining the above properties, we complete the proof of Theorem \ref{Theoremincompressible1}.

\subsection{Convergence of the Euler-NS system to the TNS system}\label{subsection:incompressible2}

Then, we study the singular limit of the Euler-NS system \eqref{pENSvar} toward the TNS system \eqref{TNS} as $\tau=\var\rightarrow0$. In the case $\tau=\var=1$, it follows from Theorem \ref{theorem1} that if
\begin{equation*}
\begin{aligned}
&\|(\rho_{0},n_{0}-1)\|_{\dot{B}^{\frac{d}{2}-1}_{2,1}\cap\dot{B}^{\frac{d}{2}}_{2,1}}
+\|(u_{0},v_{0})\|_{\dot{B}^{\frac{d}{2}-1}_{2,1}}
+\|u_{0}\|_{\dot{B}^{\frac{d}{2}+1}_{2,1}}
\leq \delta_{0}\leq \delta_{0},
\end{aligned}
\end{equation*}
then the original Cauchy problem \eqref{pENS}-\eqref{indt} admits a unique global strong solution $(\rho,u,n,v)$ that satisfies  
\begin{align}
&\|\rho\|_{\widetilde{L}^{\infty}(\mathbb{R}_{+};\dot{B}^{\frac{d}{2}-1}_{2,1})}\leq C_1\|\rho_0\|_{\dot{B}^{\frac{d}{2}-1}_{2,1}},\quad \|\rho\|_{\widetilde{L}^{\infty}(\mathbb{R}_{+};\dot{B}^{\frac{d}{2}}_{2,1})}\leq C_1\|\rho_0\|_{\dot{B}^{\frac{d}{2}}_{2,1}},\label{sfgg}\\
&\begin{aligned}
&\|(u,v)\|_{\widetilde{L}^{\infty}(\mathbb{R}_+;\dot{B}^{\frac{d}{2}-1}_{2,1})}+\|u\|_{\widetilde{L}^{\infty}(\mathbb{R}_+;\dot{B}^{\frac{d}{2}+1}_{2,1})}+\|n-1\|_{\widetilde{L}^{\infty}(\mathbb{R}_+;\dot{B}^{\frac{d}{2}-1}_{2,1}\cap\dot{B}^{\frac{d}{2}}_{2,1})}\\
&\quad+\|(u,v)\|_{L^{1}(\mathbb{R}_+;\dot{B}^{\frac{d}{2}+1}_{2,1})}+\|n-1\|_{L^{1}(\mathbb{R}_+;\dot{B}^{\frac{d}{2}+1}_{2,1}+\dot{B}^{\frac{d}{2}}_{2,1})}\leq C_{1}\delta_{0},\label{sfgg1}
\end{aligned}
\end{align}
and
\begin{equation}\label{sfgg2}
\begin{aligned}
&\|u-v\|_{L^{1}(\mathbb{R}_+;\dot{B}^{\frac{d}{2}}_{2,1})}+\|u-v\|_{\widetilde{L}^{2}(\mathbb{R}_+;\dot{B}^{\frac{d}{2}-1}_{2,1})}\leq C_{1}\delta_{0}.
\end{aligned}
\end{equation}
We perform a similar rescaling to recover the uniform regularity of solutions to the Euler-NS system \eqref{pENS}. For the global solution $(\rho,u,n,v)$ to \eqref{pENS} subject to the initial data $(\rho_0,u_0,n_0,v_0)$, with $\mu$ and $\lambda$ replaced by $\bar\mu$ and $\bar\lambda$, respectively, we take the change of unknowns
\begin{align}\label{scaling3}
(\rho,u,n-1,v)(t,x)=\var(\rho^\var,u^\var,a^\var,v^\var)(\var^2 t, \var x),
\end{align}
and the change of the initial data
\begin{align}\label{scaling30}
(\rho_0,u_0,n_0-1,v_0)(x)=\var(\rho^\var_0,u_0^\var,a_0^\var,v_0^\var)(\var x).
\end{align}
Then, one can check that $(\rho^\var,u^\var,a^\var,v^\var)$ solves the scaled Euler-NS system \eqref{pENSvar} for $\var\in(0,1)$. Based on \eqref{sfgg}-\eqref{scaling30} and the scaling properties of Besov norms (see Lemma \ref{lember}), we can recover the uniform regularity estimates \eqref{Incompress:r2}-\eqref{Incompress:r22}.

The proof of the convergence is similar to that in Subsection \ref{subsection:incompressible1}. Here we only give a little explanation for key differences. In view of \eqref{Incompress:r22}, the velocities $u^\var$ and $v^\var$ must align with each other as $\var\rightarrow 0$. Hence, we derive from $\eqref{pENSvar}$ that
\begin{equation}\label{rhov123}
\left\{
\begin{aligned}
&\partial_{t}\rho^\var+\operatorname{div} (\rho^\var v^\var)=-\operatorname{div} \big(\rho^\var (u^\var-v^\var) \big),\\
&\partial_{t}\big(\var \rho^\var u^\var+(1+\var a^\var) v^\var\big)+\operatorname{div}\big(\var \rho^\var u^\var \otimes u^\var+(1+\var a^\var)v^\var\otimes v^\var\big)+\frac{1}{\var^2}\nabla P(1+\var a^\var)\\
&\qquad=\bar\mu\Delta v^\var+(\bar\mu+\bar\lambda)\nabla \operatorname{div}v^\var.
\end{aligned}
\right.
\end{equation}
By carrying out a similar argument as in \eqref{limit1:in}-\eqref{limit2:in} and using the alignment property of two velocities, as well as the convergence of the relative velocity $u^\var-v^\var$, one can show that $\eqref{rhov123}_1$ converges to $\eqref{TNS}_1$ in the distributional sense as $\var\rightarrow0$. 

Then, we justify the convergence of $\eqref{rhov123}_2$ toward $\eqref{TNS}_2$. Following the same line as in \eqref{6.6}-\eqref{6.11}, we apply the operator $\mathcal{P}$ to $\eqref{rhov123}_2$ and get
\begin{equation}\label{6.27}
\begin{aligned}
&\partial_{t}\mathcal{P}M^\var+\mathcal{P}\operatorname{div} \big(\var \rho^\var u^\var \otimes u^\var+(1+\var a^\var)v^\var\otimes v^\var\big)=\bar\mu \Delta \mathcal{P} v^\var,
\end{aligned}
\end{equation}
where $M^\var:=\var \rho^\var u^\var+(1+\var a^\var) v^\var$. By \eqref{6.27} and the bounds \eqref{Incompress:r2}-\eqref{Incompress:r22}, we have the uniform estimate of the time derivative term $\partial_t\mathcal{P}M^\var$. Therefore, there exists a limit $w$ such that, up to a subsequence, $\mathcal{P}M^\var$ converges to $w$ strongly in $L^2(0,T;H^{\frac{d}{2}- \eta_*}_{\rm{loc}})$ for any $T>0$ and $\eta_*\in(0,1)$. Due to $M^\var-v^\var=\var(\rho^\var u^\var+a^\var v^\var)=\mathcal{O}(\var)$ and $u^\var-v^\var=\mathcal{O}(\var)$ derived from \eqref{Incompress:r22}, it holds
\begin{align*}
& \mathcal{P} u^\var\rightarrow w\quad\text{and}\quad \mathcal{P} v^\var\rightarrow w\qquad\text{strongly in}~~ L^2(0,T;H^{\frac{d}{2}- \eta_*}_{\rm{loc}}).
\end{align*}
To analyze the second term in \eqref{6.27}, one needs to have the strong convergence of $\mathcal{Q}u^\var$ and $\mathcal{Q}v^\var$. The main challenge is to establish the estimates in \eqref{Incompress:limit2}, which, in particular, implies the convergence of the convection form in \eqref{6.27}. To overcome this, we introduce the {\emph{effective mixed velocity}}:
\begin{align}\label{V000}
V^\var:=\frac{\var \rho^\var}{1+\var \rho^\var+\var a^\var}u^\var+\frac{1+\var a^\var}{1+\var \rho^\var+\var a^\var}v^\var,
\end{align}
which yields that
\begin{align*}
&\partial_{t}\big(\var \rho^\var u^\var+(1+\var a^\var) v^\var\big)+\operatorname{div}\big(\var \rho^\var u^\var \otimes u^\var+(1+\var a^\var)v^\var\otimes v^\var\big)\\
&\quad= \partial_{t}\big( (1+\var \rho^\var+\var a^\var) V^\var\big)+ \operatorname{div} \big( (1+\var \rho^\var+\var a^\var) v^\var \otimes V^\var \big)+\var\operatorname{div}\big(  \rho^\var (u^\var-v^\var) \otimes u^\var\big) \\
&\quad=(1+\var \rho^\var+\var a^\var) (\partial_t V^\var+v^\var \cdot \nabla V^\var)-\var\operatorname{div} \big(\rho^\var (u^\var-v^\var) \big) V^\var+\var\operatorname{div}\big(  \rho^\var (u^\var-v^\var) \otimes u^\var\big).
\end{align*}
This enables us to rewrite \eqref{pENSvar} as
\begin{equation}\label{628}
\left\{
\begin{aligned}
&\partial_t a^\var+\frac{1}{\var} \operatorname{div} V^\var=\mathbf{G}_1,\\
& \partial_t V^\var+\frac{1}{\var}\nabla a^\var=\mathbf{G}_2,
\end{aligned}
\right.
\end{equation}
where $\mathbf{G}_1$ and $\mathbf{G}_2$ satisfy
\begin{align*}
&\mathbf{G}_1:=\operatorname{div}\Big(\frac{\rho^\var (u^\var-v^\var)}{1+\var \rho^\var+\var a^\var}\Big)-\operatorname{div} (a^\var v^\var) ,\\
&\mathbf{G}_2:=- v^\var\cdot\nabla V^\var-\frac{\var^{-1}(P'(1+\var a^\var)-1)-\rho^\var-a^\var}{1+\var \rho^\var+\var a^\var}\nabla a^\var+\frac{1}{1+\var a^\var+\var \rho^\var} (\mu \Delta v^\var+(\mu+\lambda)\nabla\operatorname{div} v^\var)\\
&\qquad\quad+\frac{1}{1+\var \rho^\var+\var a^\var}\Big(\var\operatorname{div}\big(\rho^\var (u^\var-v^\var) \big) V^\var-\var\operatorname{div}\big(  \rho^\var (u^\var-v^\var) \otimes u^\var\big)\Big).
\end{align*}

By direct computations based on the uniform estimates \eqref{Incompress:r2}-\eqref{Incompress:r22}, both $\mathbf{G}_1$ and $\mathbf{G}_2$ are uniformly bounded in $L^1(\mathbb{R}_+;\dot{B}^{\frac{d}{2}-1}_{2,1})$. Then, one can repeat the calculations as those in Subsection \ref{subsection:incompressible1} to obtain the estimates of $(a^\var,\mathcal{Q}V^\var)$:
\begin{equation*}
\left\{
\begin{aligned}
&\|(a^\var,\mathcal{Q}V^\var)\|_{\widetilde{L}^{2}(\mathbb{R}_+;\dot{B}^{\frac{d+1}{p}-\frac{1}{2}}_{p,1})}\leq C\var^{\frac{1}{2}-\frac{1}{p}}&&\text{if}~~ d\geq 3,\\
&\|(a^\var, \mathcal{Q}V^\var)\|_{\widetilde{L}^{2}(\mathbb{R}_+;\dot{B}^{\frac{5}{2p}-\frac{1}{4}}_{p,1})}\leq C\var^{\frac{1}{4}-\frac{1}{2p}}&& \text{if}~~ d=2.
\end{aligned}
\right.
\end{equation*}
For brevity, we omit the details. These, together with the uniform bound of $\rho^\var$ in \eqref{Incompress:r2}, lead to the estimates of $\mathcal{Q}v^\var$ in \eqref{Incompress:limit2}. Indeed, we notice that
$$
v^\var=V^\var-\frac{\var\rho^\var(u^\var-v^\var)}{1+\var \rho^\var+\var a^\var}.
$$
In the case $d\geq3$, for any $p\in(2,\infty)$, one can verify
\begin{align*}
\|\mathcal{Q}v^\var\|_{\widetilde{L}^{2}(\mathbb{R}_+;\dot{B}^{\frac{d+1}{p}-\frac{1}{2}}_{p,1})}^{\var,\ell}&\lesssim \|\mathcal{Q}V^\var\|_{\widetilde{L}^{2}(\mathbb{R}_+;\dot{B}^{\frac{d+1}{p}-\frac{1}{2}}_{p,1})}^{\var,\ell}+\var^{\frac{1}{2}-\frac{1}{p}} \Big\|\frac{\rho^\var(u^\var-v^\var)}{1+\var a^\var+\var\rho^\var}\Big\|_{\widetilde{L}^{2}(\mathbb{R}_+;\dot{B}^{\frac{d}{p}-1}_{p,1})}^{\var,\ell}\\
&\lesssim \|\mathcal{Q}V^\var\|_{\widetilde{L}^{2}(\mathbb{R}_+;\dot{B}^{\frac{d+1}{p}-\frac{1}{2}}_{p,1})}^{\var,\ell}+\var^{\frac{1}{2}-\frac{1}{p}}\|\rho^\var\|_{\widetilde{L}^{\infty}(\mathbb{R}_{+};\dot{B}^{\frac{d}{2}}_{2,1})}\|u^\var-v^\var\|_{\widetilde{L}^{2}(\mathbb{R}_+;\dot{B}^{\frac{d}{p}-1}_{p,1})}\lesssim \var^{\frac{1}{2}-\frac{1}{p}}.
\end{align*}
Therefore, with the high-frequency bounds, it holds
\begin{align*}
\|\mathcal{Q}v^\var\|_{\widetilde{L}^{2}(\mathbb{R}_+;\dot{B}^{\frac{d+1}{p}-\frac{1}{2}}_{p,1})}&\lesssim \|\mathcal{Q}v^\var\|_{\widetilde{L}^{2}(\mathbb{R}_+;\dot{B}^{\frac{d+1}{p}-\frac{1}{2}}_{p,1})}^{\var,\ell}+\var^{\frac{1}{2}-\frac{1}{p}}\|\mathcal{Q}v^\var\|_{\widetilde{L}^{2}(\mathbb{R}_+;\dot{B}^{\frac{d}{p}}_{p,1})}^{\var,h}\lesssim \var^{\frac{1}{2}-\frac{1}{p}}.
\end{align*}
In the case $d=2$, one can perform a similar computation and achieve the desired convergence rate. 

The bounds of $\mathcal{Q}u^\var$ then follow from the estimates of $\mathcal{Q}v^\var$ and $u^\var-v^\var$. In conclusion, we finally complete the proof of Theorem \ref{Theoremincompressible2}.

\appendix
\renewcommand{\thesection}{\Alph{section}}
\renewcommand{\theequation}{\Alph{section}.\arabic{equation}} 

\section*{Appendix A: {Linear analysis}}\label{sectionlinear}
\addcontentsline{toc}{section}{Appendix A}
\setcounter{section}{1}
\setcounter{equation}{0}
We analyze the spectral structure and the Green's function for the linear system of \eqref{pENS}, and then establish the optimal time-decay rates of the linear problem uniformly with respect to the parameter $\tau>0$. It is obvious that the linear part of density $\rho$ is decoupled from the other variables. Hence, we perform a spectral analysis of the following linear problem
\begin{equation}\label{LpENS1}
\left\{
\begin{aligned}
&\partial_{t}u+\frac{1}{\tau}(u-v)=0,\\
&\partial_{t}a+\operatorname{div} v=0,\\
&\partial_{t}v+\nabla a-\Delta v=0,\\
&(u,a,v)(0,x)=(u_{0},a_{0},v_{0})(x).
\end{aligned}
\right.
\end{equation}
By the Hodge decomposition, we denote the "compressible part"  $\phi:=\Lambda^{-1}\operatorname{div}u$ and the "incompressible part" $\psi:=\Lambda^{-1}\operatorname{div}v$,  respectively, subject to the initial values $(\phi_{0},\psi_{0}):=(\Lambda^{-1}\operatorname{div}u_{0},\Lambda^{-1}\operatorname{div}v_{0})$ and $(\Phi_{0},\Psi_{0}):=(\Lambda^{-1}\nabla\times u_{0},\Lambda^{-1}\nabla\times v_{0})$. Then, the system \eqref{LpENS1} translates into
\begin{equation}\notag
\left\{
\begin{aligned}
&\partial_{t}\phi+\frac{1}{\tau}(\phi-\psi)=0,\\
&\partial_{t}a+\Lambda\psi=0,\\
&\partial_{t}\psi-\Lambda a+\Lambda^{2}\psi=0,\\
&(\phi,a,\psi)(0,x)=(\phi_{0},a_{0},\psi_{0})(x),
\end{aligned}
\right.\quad\text{and}\quad
\left\{
\begin{aligned}
&\partial_{t}\Phi+\frac{1}{\tau}(\Phi-\Psi)=0,\\
&\partial_{t}\Psi+\Lambda^{2}\Psi=0,\\
&(\Phi,\Psi)(0,x)=(\Phi_{0},\Psi_{0})(x).
\end{aligned}
\right.
\end{equation}
Taking the Fourier transform with respect to the space variable, we have
\begin{equation}\notag
\partial_{t}\left(\begin{array}{c}
\widehat{\phi}  \\
\widehat{a}  \\
\widehat{\psi}
\end{array}\right)=\mathcal{A}\left(\begin{array}{c}
\widehat{\phi}  \\
\widehat{a}  \\
\widehat{\psi}
\end{array}\right)\quad\text{and}\quad\partial_{t}\left(\begin{array}{c}
\widehat{\Phi}  \\
\widehat{\Psi}
\end{array}\right)=\mathcal{B}\left(\begin{array}{c}
\widehat{\Phi}  \\
\widehat{\Psi}
\end{array}\right),
\end{equation}
where the linear operators $\mathcal{A}$ and $\mathcal{B}$ are given by
\begin{equation}\notag
\mathcal{A}=\left(\begin{array}{ccc}
-\frac{1}{\tau} & 0 & \frac{1}{\tau}  \\
0 & 0 & -|\xi|  \\
0 & |\xi| & -|\xi|^{2}
\end{array}\right)\quad\text{and}\quad\mathcal{B}=\left(\begin{array}{cc}
-\frac{1}{\tau} & \frac{1}{\tau}  \\
0 & -|\xi|^{2}
\end{array}\right).
\end{equation}
Note that the eigenvalues of the matrix $\mathcal{A}$ and $\mathcal{B}$ satisfy
\begin{equation}\notag
\lambda_{1}=-\frac{1}{\tau},\quad\lambda_{2,3}=-\frac{1}{2}|\xi|^{2}\pm\frac{1}{2}\sqrt{|\xi|^{4}-4|\xi|^{2}},
\end{equation}
and
\begin{equation}\notag
\lambda_{4}=-\frac{1}{\tau},\quad\lambda_{5}=-|\xi|^{2},
\end{equation}
respectively. These eigenvalues reveal different behaviors of the solution in low and high frequencies. Indeed, we have the following frequency-localized estimates of \eqref{LpENS1}.

\begin{lemma}\label{lemma21}
Let $\tau\in(0,\frac{1}{4})$. The solution $(u,a,v)$ to the problem \eqref{LpENS1} has the following properties{\rm{:}}
\begin{itemize}
\item {\rm(}Low frequencies{\rm)}: For any $j\leq 0$ and $t\geq 0$, there exists a generic constant $r_{*}>0$ such that
\begin{equation}\label{lowj}
\begin{aligned}
\|\dot{\Delta}_j (u,a,v)(t)\|_{L^{2}}&\lesssim e^{-r_{*}2^{2j}t}\|\dot{\Delta}_j (u_{0},a_{0},v_{0})\|_{L^{2}}.
\end{aligned}
\end{equation}

More precisely, it holds
\begin{equation}\label{lowj-1}
\left\{
\begin{aligned}
&\|\dot{\Delta}_j u(t)\|_{L^{2}}\lesssim e^{-\frac{1}{\tau}t}\Big(\|\dot{\Delta}_{j}(u_{0},v_{0})\|_{L^{2}}+\tau 2^{j}\|\dot{\Delta}_{j}a_{0}\|_{L^{2}}\Big)+e^{-r_{*}2^{2j}t}\|\dot{\Delta}_{j}(a_{0},v_{0})\|_{L^{2}},\\
&\|\dot{\Delta}_j (a,v)(t)\|_{L^{2}}\lesssim e^{-r_{*}2^{2j}t}\|\dot{\Delta}_{j}(a_{0},v_{0})\|_{L^{2}},
\end{aligned}
\right.
\end{equation}
and
\begin{equation}\label{relativej-1}
\begin{aligned}
\|\dot{\Delta}_{j}(u-v)(t)\|_{L^{2}}\lesssim e^{-\frac{1}{\tau}t}\Big(\|\dot{\Delta}_{j}(u_{0},v_{0})\|_{L^{2}}+\tau 2^{j}\|\dot{\Delta}_{j}a_{0}\|_{L^{2}}\Big)+\tau 2^{j}e^{-r_{*}2^{2j}t}\|\dot{\Delta}_{j}(a_{0},v_{0})\|_{L^{2}}.
\end{aligned}
\end{equation}

\item {\rm(}High frequencies{\rm)}:  For any $j\geq 1$ and $t\geq 0$, there exists a generic constant $R_{*}>0$ such that
\begin{equation}\label{highj}
\|\dot{\Delta}_{j}(u,a,v)(t)\|_{L^{2}}\lesssim e^{-R_{*}t}\|\dot{\Delta}_{j}(u_{0},a_{0},v_{0})\|_{L^{2}}.
\end{equation}
More precisely, it holds for $t\geq 1$ that
\begin{equation}\label{highj-1}
\left\{
\begin{aligned}
&\|\dot{\Delta}_{j}u(t)\|_{L^{2}}\lesssim e^{-R_{*}t}\Big(\tau\|\dot{\Delta}_{j}u_{0}\|_{L^{2}}+2^{-j}\|\dot{\Delta}_{j}a_{0}\|_{L^{2}}+2^{-2j}\|\dot{\Delta}_{j}v_{0}\|_{L^{2}}\Big),\\
&\|\dot{\Delta}_{j}a(t)\|_{L^{2}}\lesssim e^{-R_{*}t}\Big(\|\dot{\Delta}_{j}a_{0}\|_{L^{2}}+2^{-j}\|\dot{\Delta}_{j}v_{0}\|_{L^{2}}\Big),\\
&\|\dot{\Delta}_{j}v(t)\|_{L^{2}}\lesssim e^{-R_{*}t}\Big(2^{-j}\|\dot{\Delta}_{j}a_{0}\|_{L^{2}}+2^{-2j}\|\dot{\Delta}_{j}v_{0}\|_{L^{2}}\Big),
\end{aligned}
\right.
\end{equation}
and
\begin{equation}\label{relativej-2}
\|\dot{\Delta}_{j}(u-v)(t)\|_{L^{2}}\lesssim \tau e^{-R_{*}t}\Big(\|\dot{\Delta}_{j}(u_{0},v_{0})\|_{L^{2}}+2^{-j}\|\dot{\Delta}_{j}a_{0}\|_{L^{2}}\Big).
\end{equation}
\end{itemize}
\end{lemma}

\begin{remark}
Lemma {\rm\ref{lemma21}} shows that the solution $(u,a,v)$ is coupled with each other and the sum has the same behavior as the heat equation. However, different components and relative velocity may have better dissipative structures. In particular, the estimates \eqref{relativej-1} and \eqref{relativej-2} for the relative velocity exhibit higher-order decay and the convergence rate of $\mathcal{O}(\tau)$. 
\end{remark}

\begin{proof}[Proof of Lemma {\rm\ref{lemma21}}]
After direct computations on the project operators associated with the eigenvalues, we have the Green's function expressions:
\begin{equation}\notag
\begin{aligned}
\widehat{\phi}&=e^{-\frac{1}{\tau}t}\widehat{\phi_{0}}+\Big(\frac{\tau|\xi|e^{-\frac{1}{\tau}t}}{1-\tau|\xi|^{2}+\tau^{2}|\xi|^{2}}+\frac{|\xi|e^{\lambda_{2}t}}{(1+\tau\lambda_{2})(\lambda_{2}-\lambda_{3})}-\frac{|\xi|e^{\lambda_{3}t}}{(1+\tau\lambda_{3})(\lambda_{2}-\lambda_{3})}\Big)\widehat{a_{0}}\\
&\quad+\Big(-\frac{e^{-\frac{1}{\tau}t}}{1-\tau|\xi|^{2}+\tau^{2}|\xi|^{2}}+\frac{\lambda_{2}e^{\lambda_{2}t}}{(1+\tau\lambda_{2})(\lambda_{2}-\lambda_{3})}-\frac{\lambda_{3}e^{\lambda_{3}t}}{(1+\tau\lambda_{3})(\lambda_{2}-\lambda_{3})}\Big)\widehat{\psi_{0}},\\[2mm]
\widehat{a}&=\frac{\lambda_{2}e^{\lambda_{3}t}-\lambda_{3}e^{\lambda_{2}t}}{\lambda_{2}-\lambda_{3}}\widehat{a_{0}}+\frac{|\xi|e^{\lambda_{3}t}-|\xi|e^{\lambda_{2}t}}{\lambda_{2}-\lambda_{3}}\widehat{\psi_{0}},\qquad \widehat{\psi}=\frac{|\xi|e^{\lambda_{2}t}-|\xi|e^{\lambda_{3}t}}{\lambda_{2}-\lambda_{3}}\widehat{a_{0}}+\frac{\lambda_{2}e^{\lambda_{2}t}-\lambda_{3}e^{\lambda_{3}t}}{\lambda_{2}-\lambda_{3}}\widehat{\psi_{0}},
\end{aligned}
\end{equation}
and
\begin{equation}\notag
\widehat{\Phi}=e^{-\frac{1}{\tau}t}\widehat{\Phi_{0}}+\frac{e^{-|\xi|^{2}t}-e^{-\frac{1}{\tau}t}}{1-\tau|\xi|^{2}}\widehat{\Psi_{0}},\qquad\widehat{\Psi}=e^{-|\xi|^{2}t}\widehat{\Psi_{0}}.
\end{equation}
In addition, the relative velocity $\widehat{\phi}-\widehat{\psi}$ and $\widehat{\Phi}-\widehat{\Psi}$ are given by
\begin{equation}\notag
\begin{aligned}
\widehat{\phi}-\widehat{\psi}&=e^{-\frac{1}{\tau}t}\widehat{\phi_{0}}+\Big(\frac{\tau|\xi|e^{-\frac{1}{\tau}t}}{1-\tau|\xi|^{2}+\tau^{2}|\xi|^{2}}-\frac{\tau|\xi|\lambda_{2}e^{\lambda_{2}t}}{(1+\tau\lambda_{2})(\lambda_{2}-\lambda_{3})}+\frac{\tau|\xi|\lambda_{3}e^{\lambda_{3}t}}{(1+\tau\lambda_{3})(\lambda_{2}-\lambda_{3})}\Big)\widehat{a_{0}}\\
&\quad+\Big(-\frac{e^{-\frac{1}{\tau}t}}{1-\tau|\xi|^{2}+\tau^{2}|\xi|^{2}}-\frac{\tau\lambda_{2}^{2}e^{\lambda_{2}t}}{(1+\tau\lambda_{2})(\lambda_{2}-\lambda_{3})}+\frac{\tau\lambda_{3}^{2}e^{\lambda_{3}t}}{(1+\tau\lambda_{3})(\lambda_{2}-\lambda_{3})}\Big)\widehat{\psi_{0}},\\[2mm]
\widehat{\Phi}-\widehat{\Psi}&=e^{-\frac{1}{\tau}t}\widehat{\Phi_{0}}+\frac{\tau|\xi|^{2}e^{-|\xi|^{2}t}-e^{-\frac{1}{\tau}t}}{1-\tau|\xi|^{2}}\widehat{\Psi_{0}}.
\end{aligned}
\end{equation}

In the low frequencies ($|\xi|<2$), the distinct complex conjugated eigenvalues $\lambda_{2}$ and $\lambda_{3}$ satisfy
\begin{equation}\notag
\lambda_{2,3}=-\frac{1}{2}|\xi|^{2}(1\mp iL(\xi))\quad\text{with}\quad L(\xi):=\sqrt{\frac{4}{|\xi|^{2}}-1}.
\end{equation}
Since $\tau<1$, there exists a constant $r_{*}>0$ such that
\begin{equation}\label{Lcharaclf}
\max\big\{\lambda_{1},\,{\rm{Re}}\lambda_{2},\,{\rm{Re}}\lambda_{3},\,\lambda_{4},\,\lambda_{5}\big\}\leq-r_{*}|\xi|^{2},
\end{equation}
from which we expect the linearized system \eqref{LpENS1} behaves like a heat flow in low frequencies. More precisely, as $|\xi|\rightarrow 0$, It follows from the Taylor's expansion that the leading term of each variable can be derived by
\begin{equation}\notag
\begin{aligned}
&\widehat{\phi}\sim e^{-\frac{1}{\tau}t}\Big(\widehat{\phi_{0}}-\widehat{\psi_{0}}+\tau|\xi|\widehat{a_{0}}\Big)+\frac{1}{2}e^{\lambda_{2}t}\Big(\widehat{\psi_{0}}-i\widehat{a_{0}}\Big)+\frac{1}{2}e^{\lambda_{3}t}\Big(\widehat{\psi_{0}}+i\widehat{a_{0}}\Big),\\
&\widehat{a}\sim\frac{1}{2}e^{\lambda_{2}t}\Big(\widehat{a_{0}}+i\widehat{\psi_{0}}\Big)+\frac{1}{2}e^{\lambda_{3}t}\Big(\widehat{a_{0}}-i\widehat{\psi_{0}}\Big),\qquad \widehat{\psi}\sim\frac{1}{2}e^{\lambda_{2}t}\Big(\widehat{\psi_{0}}-i\widehat{a_{0}}\Big)+\frac{1}{2}e^{\lambda_{3}t}\Big(\widehat{\psi_{0}}+i\widehat{a_{0}}\Big),\\
&\widehat{\Phi}\sim e^{-\frac{1}{\tau}t}\Big(\widehat{\Phi_{0}}-\widehat{\Psi_{0}}\Big)+e^{-|\xi|^{2}t}\widehat{\Psi_{0}},\qquad \widehat{\Psi}\sim e^{-|\xi|^{2}t}\widehat{\Psi_{0}},
\end{aligned}
\end{equation}
and further, the leading terms of the relative velocity hold
\begin{equation}\notag
\begin{aligned}
&\widehat{\phi}-\widehat{\psi}\sim e^{-\frac{1}{\tau}t}\Big(\widehat{\phi_{0}}-\widehat{\psi_{0}}+\tau|\xi|\widehat{a_{0}}\Big)-\frac{1}{2}\tau|\xi|e^{\lambda_{2}t}\Big(\widehat{a_{0}}+i\widehat{\psi_{0}}\Big)-\frac{1}{2}\tau|\xi|e^{\lambda_{3}t}\Big(\widehat{a_{0}}-i\widehat{\psi_{0}}\Big),\\
&\widehat{\Phi}-\widehat{\Psi}\sim e^{-\frac{1}{\tau}t}\Big(\widehat{\Phi_{0}}-\widehat{\Psi_{0}}\Big)+\tau|\xi|^{2}e^{-|\xi|^{2}t}\widehat{\Psi_{0}}.
\end{aligned}
\end{equation}
The above approximation holds up to higher–order terms. According to the fact $|\widehat{f}|\sim|\widehat{\Lambda^{-1}\operatorname{div} f}|+|\widehat{\Lambda^{-1}\nabla\times f}|$, \eqref{Lcharaclf}, and Parseval's formula, we obtain for any $j\leq 0$ that
\begin{align}
\label{Lulf}&\|\dot{\Delta}_{j}u\|_{L^{2}}\lesssim e^{-\frac{1}{\tau}t}\Big(\|\dot{\Delta}_{j}(u_{0},v_{0})\|_{L^{2}}+\tau 2^{j}\|\dot{\Delta}_{j}a_{0}\|_{L^{2}}\Big)+e^{-r_{*}2^{2j}t}\|\dot{\Delta}_{j}(a_{0},v_{0})\|_{L^{2}},\\
\label{Lavlf}&\|\dot{\Delta}_{j}(a,v)\|_{L^{2}}\lesssim e^{-r_{*}2^{2j}t}\|\dot{\Delta}_{j}(a_{0},v_{0})\|_{L^{2}},\\
\label{Lrvlf}&\|\dot{\Delta}_{j}(u-v)\|_{L^{2}}\lesssim e^{-\frac{1}{\tau}t}\Big(\|\dot{\Delta}_{j}(u_{0},v_{0})\|_{L^{2}}+\tau 2^{j}\|\dot{\Delta}_{j}a_{0}\|_{L^{2}}\Big)+\tau 2^{j}e^{-r_{*}2^{2j}t}\|\dot{\Delta}_{j}(a_{0},v_{0})\|_{L^{2}}.
\end{align}

On the other hand, for $|\xi|>2$, the two distinct real eigenvalues $\lambda_{2}$ and $\lambda_{3}$ are, respectively, given by
\begin{equation}\notag
\lambda_{2}=-\frac{1}{2}|\xi|^{2}(1- H(\xi)) \quad\text{and}\quad \lambda_{3}=-\frac{1}{2}|\xi|^{2}(1+ H(\xi))\quad\text{with}\quad H(\xi):=\sqrt{1-\frac{4}{|\xi|^{2}}},
\end{equation}
which, noting again $\tau<1$, yields that there exist some constants $r^{*}>0$ and $R^{*}>0$ such that
\begin{equation}\label{Lcharachf}
\max\big\{\lambda_{3},\lambda_{5}\big\}\leq-r^{*}|\xi|^{2}\quad \text{and}\quad \max\big\{\lambda_{1},\lambda_{2},\lambda_{3},\lambda_{4},\lambda_{5}\big\}\leq-R^{*},
\end{equation}
and it holds for any $\alpha>0$ that
\begin{equation}\label{epvart1}
e^{-\frac{1}{\tau}t}\lesssim \tau^{\alpha}e^{-\frac{1}{2\tau}t},\qquad t\geq1.
\end{equation}
These allow us to expect the exponentially damped behavior in high frequencies. We first consider the regime $|\xi|\geq\frac{1}{\sqrt{\tau}}$, which implies
\begin{align}\label{epvart2}
e^{-r^{*}|\xi|^{2}t}\lesssim \tau e^{-\frac{r^{*}}{2}|\xi|^{2}t},\qquad t\geq1.
\end{align}
The asymptotic behavior of each variable for $\sqrt{\tau}|\xi|\rightarrow\infty$ is given by
\begin{equation}\notag
\begin{aligned}
&\widehat{\phi}\sim e^{-\frac{1}{\tau}t}\Big(\widehat{\phi_{0}}-\frac{\widehat{a_{0}}}{|\xi|}+\frac{\widehat{\psi_{0}}}{\tau|\xi|^{2}}\Big)+e^{\lambda_{2}t}\Big(\frac{\widehat{a_{0}}}{|\xi|}-\frac{\widehat{\psi_{0}}}{|\xi|^{2}}\Big)+\frac{1}{\tau}e^{\lambda_{3}t}\Big(-\frac{\widehat{\psi_{0}}}{|\xi|^{2}}+\frac{\widehat{a_{0}}}{|\xi|^{3}}\Big),\\
&\widehat{a}\sim e^{\lambda_{2}t}\Big(\widehat{a_{0}}-\frac{\widehat{\psi_{0}}}{|\xi|}\Big)+e^{\lambda_{3}t}\Big(\frac{\widehat{\psi_{0}}}{|\xi|}-\frac{\widehat{a_{0}}}{|\xi|^{2}}\Big),\qquad \widehat{\psi}\sim e^{\lambda_{2}t}\Big(\frac{\widehat{a_{0}}}{|\xi|}-\frac{\widehat{\psi_{0}}}{|\xi|^{2}}\Big)+e^{\lambda_{3}t}\Big(\widehat{\psi_{0}}-\frac{\widehat{a_{0}}}{|\xi|}\Big),\\
&\widehat{\Phi}\sim e^{-\frac{1}{\tau}t}\Big(\widehat{\Phi_{0}}+\frac{\widehat{\Psi_{0}}}{\tau|\xi|^{2}}\Big)-e^{-|\xi|^{2}t}\frac{\widehat{\Psi_{0}}}{\tau|\xi|^{2}},\qquad \widehat{\Psi}\sim e^{-|\xi|^{2}t}\widehat{\Psi_{0}},
\end{aligned}
\end{equation}
and
\begin{equation}\notag
\begin{aligned}
&\widehat{\phi}-\widehat{\psi}\sim e^{-\frac{1}{\tau}t}\Big(\widehat{\phi_{0}}-\frac{\widehat{a_{0}}}{|\xi|}+\frac{\widehat{\psi_{0}}}{\tau|\xi|^{2}}\Big)+\tau e^{\lambda_{2}t}\Big(\frac{\widehat{a_{0}}}{|\xi|}-\frac{\widehat{\psi_{0}}}{|\xi|^{2}}\Big)+e^{\lambda_{3}t}\Big(-\widehat{\psi_{0}}+\frac{\widehat{a_{0}}}{|\xi|}\Big),\\
&\widehat{\Phi}-\widehat{\Psi}\sim e^{-\frac{1}{\tau}t}\Big(\widehat{\Phi_{0}}+\frac{\widehat{\Psi_{0}}}{\tau|\xi|^{2}}\Big)-e^{-|\xi|^{2}t}\widehat{\Psi_{0}},
\end{aligned}
\end{equation}
up to higher-order terms. Making use of Parseval's formula and \eqref{Lcharachf}-\eqref{epvart2}, we deduce for any $j\geq -\frac{1}{2}\log_{2}\tau$ and $t\geq1$ that
\begin{align}
\label{Luhf}& \|\dot{\Delta}_{j}u\|_{L^{2}}\lesssim e^{-\frac{1}{2}R^{*}t}\Big(\tau\|\dot{\Delta}_{j}u_{0}\|_{L^{2}}+2^{-j}\|\dot{\Delta}_{j}a_{0}\|_{L^{2}}+2^{-2j}\|\dot{\Delta}_{j}v_{0}\|_{L^{2}}\Big),\\
\label{Lahf}& \|\dot{\Delta}_{j}a\|_{L^{2}}\lesssim e^{-R^{*}t}\Big(\|\dot{\Delta}_{j}a_{0}\|_{L^{2}}+2^{-j}\|\dot{\Delta}_{j}v_{0}\|_{L^{2}}\Big),\\
\label{Lvhf}&\|\dot{\Delta}_{j}v\|_{L^{2}}\lesssim e^{-R^{*}t}\Big(2^{-j}\|\dot{\Delta}_{j}a_{0}\|_{L^{2}}+2^{-2j}\|\dot{\Delta}_{j}v_{0}\|_{L^{2}}\Big)+e^{-r^{*}2^{2j}t}\|\dot{\Delta}_{j}v_{0}\|_{L^{2}},\\
\label{Lrvhf}&
\begin{aligned}
\|\dot{\Delta}_{j}(u-v)\|_{L^{2}}&\lesssim \tau e^{-\frac{1}{2}R^{*}t}\Big(\|\dot{\Delta}_{j}u_{0}\|_{L^{2}}+2^{-j}\|\dot{\Delta}_{j}a_{0}\|_{L^{2}}+\|\dot{\Delta}_{j}v_{0}\|_{L^{2}}\Big).
\end{aligned}
\end{align}
Furthermore, for the case $|\xi|<\frac{1}{\sqrt{\tau}}$, the asymptotic behaviors of the eigenvalues as $|\xi|\rightarrow\frac{1}{\sqrt{\tau}}$ gives rise to
\begin{equation}\label{Lcharacmf}
\max\big\{\lambda_{1},\lambda_{3},\lambda_{4},\lambda_{5}\big\}\leq -\frac{1}{2\tau} \quad\text{and}\quad \lambda_{2}\leq -1.
\end{equation}
The leading term of each variable for $|\xi|\rightarrow\frac{1}{\sqrt{\tau}}$ has the following form:
\begin{align*}
&\widehat{\phi}\sim e^{-\frac{1}{\tau}t}\Big(\widehat{\phi_{0}}+\frac{1}{\sqrt{\tau}}\widehat{a_{0}}-\frac{1}{\tau}\widehat{\psi_{0}}\Big)+e^{\lambda_{2}t}\Big(\frac{2\sqrt{\tau}}{1-4\tau+\sqrt{1-4\tau}}\widehat{a_{0}}-\frac{1-\sqrt{1-4\tau}}{1-4\tau+\sqrt{1-4\tau}}\widehat{\psi_{0}}\Big)\\
&\qquad+e^{\lambda_{3}t}\Big(\frac{2\sqrt{\tau}}{1-4\tau-\sqrt{1-4\tau}}\widehat{a_{0}}-\frac{1+\sqrt{1-4\tau}}{1-4\tau-\sqrt{1-4\tau}}\widehat{\psi_{0}}\Big),\\
&\widehat{a}\sim e^{\lambda_{2}t}\Big(\frac{1+\sqrt{1-4\tau}}{2\sqrt{1-4\tau}}\widehat{a_{0}}-\frac{\sqrt{\tau}}{\sqrt{1-4\tau}}\widehat{\psi_{0}}\Big)+e^{\lambda_{3}t}\Big(-\frac{1-\sqrt{1-4\tau}}{2\sqrt{1-4\tau}}\widehat{a_{0}}+\frac{\sqrt{\tau}}{\sqrt{1-4\tau}}\widehat{\psi_{0}}\Big),\\
&\widehat{\psi}\sim e^{\lambda_{2}t}\Big(\frac{\sqrt{\tau}}{\sqrt{1-4\tau}}\widehat{a_{0}}-\frac{1-\sqrt{1-4\tau}}{2\sqrt{1-4\tau}}\widehat{\psi_{0}}\Big)+e^{\lambda_{3}t}\Big(-\frac{\sqrt{\tau}}{\sqrt{1-4\tau}}\widehat{a_{0}}+\frac{1+\sqrt{1-4\tau}}{2\sqrt{1-4\tau}}\widehat{\psi_{0}}\Big),\\
&\widehat{\Phi}\sim e^{-\frac{1}{\tau}t}\Big(\widehat{\Phi_{0}}+\frac{1}{2\sqrt{\tau}(|\xi|-\frac{1}{\sqrt{\tau}})}\widehat{\Psi_{0}}\Big)-e^{-|\xi|^{2}t}\frac{1}{2\sqrt{\tau}(|\xi|-\frac{1}{\sqrt{\tau}})}\widehat{\Psi_{0}},\qquad\widehat{\Psi}\sim e^{-|\xi|^{2}t}\widehat{\Psi_{0}},
\end{align*}
and
\begin{equation}\notag
\begin{aligned}
&\widehat{\phi}-\widehat{\psi}\sim e^{-\frac{1}{\tau}t}\Big(\widehat{\phi_{0}}+\frac{1}{\sqrt{\tau}}\widehat{a_{0}}-\frac{1}{\tau}\widehat{\psi_{0}}\Big)+e^{\lambda_{2}t}\Big(\frac{\sqrt{\tau}(1-\sqrt{1-4\tau})}{1-4\tau+\sqrt{1-4\tau}}\widehat{a_{0}}-\frac{1-2\tau-\sqrt{1-4\tau}}{1-4\tau+\sqrt{1-4\tau}}\widehat{\psi_{0}}\Big)\\
&\qquad\qquad+e^{\lambda_{3}t}\Big(\frac{\sqrt{\tau}(1+\sqrt{1-4\tau})}{1-4\tau-\sqrt{1-4\tau}}\widehat{a_{0}}-\frac{2-4\tau+2\sqrt{1-4\tau}}{2-8\tau-2\sqrt{1-4\tau}}\widehat{\psi_{0}}\Big),\\
&\widehat{\Phi}-\widehat{\Psi}\sim e^{-\frac{1}{\tau}t}\Big(\widehat{\Phi_{0}}+\frac{1}{2\sqrt{\tau}(|\xi|-\frac{1}{\sqrt{\tau}})}\widehat{\Psi_{0}}\Big)-e^{-|\xi|^{2}t}\frac{1}{2\sqrt{\tau}(|\xi|-\frac{1}{\sqrt{\tau}})}\widehat{\Psi_{0}}.
\end{aligned}
\end{equation}
Noticing \eqref{Lcharachf}-\eqref{epvart1}, \eqref{Lcharacmf} and the fact $1-\sqrt{1-4\tau}\lesssim\tau$ for $\tau\in(0,\frac{1}{4})$, we obtain for $2^j\in(2,\frac{1}{\sqrt{\tau}})$ and $t\geq1$ that
\begin{align}
\label{Lumf} &\|\dot{\Delta}_{j}u\|_{L^{2}}\lesssim e^{-\frac{1}{4}R^{*}t}\Big(\sqrt{\tau}\|\dot{\Delta}_{j}a_{0}\|_{L^{2}}+\tau\|\dot{\Delta}_{j}(u_{0},v_{0})\|_{L^{2}}\Big),\\
\label{Lamf} &\|\dot{\Delta}_{j}a\|_{L^{2}}\lesssim e^{-\frac{1}{2}R^{*}t}\Big(\|\dot{\Delta}_{j}a_{0}\|_{L^{2}}+\sqrt{\tau}\|\dot{\Delta}_{j}v_{0}\|_{L^{2}}\Big),\\
\label{Lvmf} &\|\dot{\Delta}_{j}v\|_{L^{2}}\lesssim e^{-\frac{1}{4}R^{*}t}\Big(\sqrt{\tau}\|\dot{\Delta}_{j}a_{0}\|_{L^{2}}+\tau\|\dot{\Delta}_{j}v_{0}\|_{L^{2}}\Big),\\
\label{Lrvmf} &\|\dot{\Delta}_{j}(u-v)\|_{L^{2}}\lesssim \tau e^{-\frac{1}{4}R^{*}t}\Big(\|\dot{\Delta}_{j}(u_{0},v_{0})\|_{L^{2}}+\sqrt{\tau}\|\dot{\Delta}_{j}a_{0}\|_{L^{2}}\Big).
\end{align}
Substituting the relation $\sqrt{\tau}\leq 2^{-j}$ into \eqref{Lumf}-\eqref{Lrvmf} and combining the resultant inequalities with \eqref{Luhf}-\eqref{Lrvhf}, we obtain the desired estimates \eqref{highj}-\eqref{relativej-2} by setting $R_{*}=\frac{1}{4}R^{*}$.
\end{proof}

Then, we establish the uniform decay estimates of solutions to the problem \eqref{LpENS1}.
\begin{lemma}\label{lemmalinearoptimal}
Let $\tau\in(0,\frac{1}{4})$, $\sigma_{0}>\sigma_{1}$, and $(u,a,v)$ be the solution to the problem \eqref{LpENS1}. Assume that the initial data $(u_{0},a_{0},v_{0})$ satisfies
\begin{equation}
\bar{\mathcal{X}}_{0}:=\|(u_{0},\nabla a_{0},v_{0})\|_{\dot{B}^{\sigma_{0}}_{2,1}}+\tau\|u_{0}\|_{\dot{B}^{\sigma_{0}+2}_{2,1}}<\infty.
\end{equation}

\begin{itemize}
\item {\rm(}Upper bounds{\rm)}. If the initial data further satisfies $(u_{0},a_{0},v_{0})^{\ell}\in \dot{B}^{\sigma_{1}}_{2,\infty}$ such that
\begin{equation}
\bar{\mathcal{Y}}_{0}=\|(u_{0},a_{0},v_{0})\|_{\dot{B}^{\sigma_{1}}_{2,\infty}}<\infty,
\end{equation}
then it holds for some uniform constant $C>0$ that
\begin{equation}\label{Luavtdes}
\left\{
\begin{aligned}
&\|(u,v)(t)\|_{\dot{B}^{\sigma}_{2,1}}\leq C(\bar{\mathcal{X}}_{0}+\bar{\mathcal{Y}}_{0})(1+t)^{-\frac{1}{2}(\sigma-\sigma_{1})}, \quad&& \sigma\in(\sigma_{1},\sigma_{0}+2],\\
&\|a(t)\|_{\dot{B}^{\sigma}_{2,1}}\leq C(\bar{\mathcal{X}}_{0}+\bar{\mathcal{Y}}_{0})(1+t)^{-\frac{1}{2}(\sigma-\sigma_{1})}, && \sigma\in(\sigma_{1},\sigma_{0}+1],
\end{aligned}
\right.
\end{equation}
and
\begin{equation}\label{Lruvtdes}
\left\{
\begin{aligned}
&\|(u-v)(t)\|_{\dot{B}^{\sigma}_{2,1}}\leq C(\bar{\mathcal{X}}_{0}+\bar{\mathcal{Y}}_{0})(1+t)^{-\frac{1}{2}(\sigma-\sigma_{1}+1)}, && \sigma\in(\sigma_{1},\sigma_{0}+1],\\
&\|(u-v)(t)\|_{\dot{B}^{\sigma}_{2,1}}\leq C\tau(\bar{\mathcal{X}}_{0}+\bar{\mathcal{Y}}_{0})(1+t)^{-\frac{1}{2}(\sigma-\sigma_{1}+1)}, \quad && \sigma\in(\sigma_{1},\sigma_{0}].
\end{aligned}
\right.
\end{equation}

\item {\rm(}Upper and Lower bounds{\rm)}. If the Fourier transform $(\widehat{a_{0}}, \widehat{v_{0}})$ further satisfies
\begin{equation}\label{idcd3}
\liminf_{|\xi|\rightarrow 0}|\xi|^{\sigma_{1}+\frac{d}{2}}|\widehat{a_{0}}(\xi)|>0,\quad  \widehat{v_{0}}(0)=\int_{\mathbb{R}^{d}}v_{0}(x)\,{\rm{d}}x=0,
\end{equation}
then it holds for some uniform constants $c,C>0$ that 
\begin{equation}
c(1+t)^{-\frac{1}{2}(\sigma-\sigma_{1})}\leq \|(u,a,v)(t)\|_{\dot{B}^{\sigma}_{2,1}}\leq C(1+t)^{-\frac{1}{2}(\sigma-\sigma_{1})},\qquad \sigma\in(\sigma_{1},\sigma_{0}+1].
\end{equation}
\end{itemize}
\end{lemma}

\begin{proof}
It follows from \eqref{lowj} and $\eqref{highj-1}_1$ that
\begin{equation}\label{Lutdes}
\begin{aligned}
\|u\|_{\dot{B}^{\sigma}_{2,1}}&\lesssim\sum_{j\leq 0}2^{j(\sigma-\sigma_{1})}e^{-r_{*}2^{2j}t}\sup_{j\leq 0}2^{j\sigma_{1}}\|\dot{\Delta}_{j}(u_{0},a_{0},v_{0})\|_{L^{2}}\\
&\qquad+\sum_{j\geq -1}2^{j\sigma}e^{-R_{*}t}\Big(\tau\|\dot{\Delta}_{j}u_{0}\|_{L^{2}}+2^{-j}\|\dot{\Delta}_{j}a_{0}\|_{L^{2}}+2^{-2j}\|\dot{\Delta}_{j}v_{0}\|_{L^{2}}\Big)\\
&\lesssim\langle t \rangle^{-\frac{1}{2}(\sigma-\sigma_{1})}\Big(\|(u_{0},a_{0},v_{0})^{\ell}\|_{\dot{B}^{\sigma_{1}}_{2,\infty}}+\tau\|u_{0}^{h}\|_{\dot{B}^{\sigma_{0}+2}_{2,1}}+\|a_{0}^{h}\|_{\dot{B}^{\sigma_{0}+1}_{2,1}}+\|v_{0}^{h}\|_{\dot{B}^{\sigma_{0}}_{2,1}}\Big),
\end{aligned}
\end{equation}
for $\sigma\in(\sigma_{1},\sigma_{0}+2]$. In a similar argument, we deduce from $\eqref{lowj-1}_2$ and $\eqref{highj-1}_{2,3}$ that
\begin{equation}\label{Lvtdes}
\left\{
\begin{aligned}
&\|a\|_{\dot{B}^{\sigma}_{2,1}}\lesssim\langle t \rangle^{-\frac{1}{2}(\sigma-\sigma_{1})}\Big(\|(a_{0},v_{0})^{\ell}\|_{\dot{B}^{\sigma_{1}}_{2,\infty}}+\|a_{0}^{h}\|_{\dot{B}^{\sigma_{0}+1}_{2,1}}+\|v_{0}^{h}\|_{\dot{B}^{\sigma_{0}}_{2,1}}\Big),&& \sigma\in(\sigma_{1},\sigma_{0}+1],\\
&\|v\|_{\dot{B}^{\sigma}_{2,1}}\lesssim\langle t \rangle^{-\frac{1}{2}(\sigma-\sigma_{1})}\Big(\|(a_{0},v_{0})^{\ell}\|_{\dot{B}^{\sigma_{1}}_{2,\infty}}+\|a_{0}^{h}\|_{\dot{B}^{\sigma_{0}+1}_{2,1}}+\|v_{0}^{h}\|_{\dot{B}^{\sigma_{0}}_{2,1}}\Big),&& \sigma\in(\sigma_{1},\sigma_{0}+2],
\end{aligned}
\right.
\end{equation}
The desired estimates \eqref{Luavtdes} and $\eqref{Lruvtdes}_1$ can be easily verified from \eqref{Lutdes} and \eqref{Lvtdes}. To prove the second estimate in \eqref{Lruvtdes}, we combine \eqref{relativej-1}, \eqref{relativej-2}, and \eqref{epvart1} to obtain
\begin{equation*}
\begin{aligned}
\|u-v\|_{\dot{B}^{\sigma}_{2,1}}&\lesssim\tau\sum_{j\leq 0}\Big(2^{j\sigma}e^{-\frac{1}{2\tau}t}\|\dot{\Delta}_{j}(u_{0},a_{0},v_{0})\|_{L^{2}}+2^{j(\sigma-\sigma_{1}+1)}e^{-r_{*}2^{2j}t}\sup_{j\leq 0}2^{j\sigma_{1}}\|\dot{\Delta}_{j}(a_{0},v_{0})\|_{L^{2}}\Big)\\
&\qquad+\tau\sum_{j\geq -1}2^{j\sigma}e^{-R_{*}t}\|\dot{\Delta}_{j}(u_{0},a_{0},v_{0})\|_{L^{2}}\\
&\lesssim\tau\langle t\rangle^{-\frac{1}{2}(\sigma-\sigma_{1}+1)}\Big(\|(u_{0},a_{0},v_{0})^{\ell}\|_{\dot{B}^{\sigma_{1}}_{2,\infty}}+\|(u_{0},a_{0},v_{0})^{h}\|_{\dot{B}^{\sigma_{0}}_{2,1}}\Big),\qquad\sigma\in(\sigma_{1},\sigma_{0}].
\end{aligned}
\end{equation*}

In order to derive the lower bounds of the solution, we recall the Green's function expressions of $(a,v)$ in low frequencies that
\begin{equation*}
\begin{aligned}
\widehat{a}(\xi,t)&=\frac{\lambda_{2}e^{\lambda_{3}t}-\lambda_{3}e^{\lambda_{2}t}}{\lambda_{2}-\lambda_{3}}\widehat{a_{0}}+\frac{|\xi|e^{\lambda_{3}t}-|\xi|e^{\lambda_{2}t}}{\lambda_{2}-\lambda_{3}}\widehat{\psi_{0}}\\
&=e^{-\frac{1}{2}|\xi|^{2}t}\bigg\{\Big(\cos(bt)+\frac{\sin(bt)}{2b}|\xi|^{2}\Big)\widehat{a_{0}}-\frac{\sin(bt)}{b}|\xi|\widehat{\psi_{0}}\bigg\},
\end{aligned}
\end{equation*}
where $b=\frac{1}{2}|\xi|^{2}L(\xi)=|\xi|+\mathcal{O}(|\xi|^{3})$ as $|\xi|\rightarrow 0$, and similarly, one has
\begin{equation*}
\begin{aligned}
\widehat{\psi}(\xi,t)=e^{-\frac{1}{2}|\xi|^{2}t}\bigg\{\frac{\sin(bt)}{b}|\xi|\widehat{a_{0}}+\Big(\cos(bt)-\frac{\sin(bt)}{2b}|\xi|^{2}\Big)\widehat{\psi_{0}}\bigg\},\qquad\widehat{\Psi}(\xi,t)=e^{-|\xi|^{2}t}\widehat{\Psi_{0}}.
\end{aligned}
\end{equation*}
Under the assumption \eqref{idcd3}, there exists some constant $c_{1}>0$ such that
\begin{equation*}
\begin{aligned}
|\widehat{a}(\xi,t)|^{2}&=e^{-|\xi|^{2}t}\bigg\{\Big(\cos(bt)+\frac{\sin(bt)}{2b}|\xi|^{2}\Big)^{2}|\widehat{a_{0}}(\xi)|^{2}+\frac{|\sin(bt)|^{2}}{b^{2}}|\xi|^{2}|\widehat{\psi_{0}}(\xi)|^{2}\\
&\qquad\qquad\quad-\Big(\cos(bt)+\frac{\sin(bt)}{2b}|\xi|^{2}\Big)\frac{\sin(bt)}{b}|\xi|\Big(\widehat{a_{0}}(\xi)\overline{\widehat{\psi_{0}}(\xi)}+\overline{\widehat{a_{0}}(\xi)}\widehat{\psi_{0}}(\xi)\Big)\bigg\}\\
&\geq e^{-|\xi|^{2}t}|\widehat{a_{0}}(\xi)|^{2}-C|\xi|^{2}(|\widehat{a_{0}}(\xi)|^{2}+1)\\
&\geq c_{1}e^{-|\xi|^{2}t}|\xi|^{-2\sigma_{1}-d}|\cos(|\xi|t)|^2,
\end{aligned}
\end{equation*}
which, together with the embedding property $\dot{B}^{\sigma}_{2,1}\hookrightarrow \dot{H}^{\sigma}$, yields for sufficiently large time $t$ that
\begin{equation*}
\begin{aligned}
\|a(t)\|_{\dot{B}^{\sigma}_{2,1}}
&\gtrsim
\left(
\int_{\substack{\langle t\rangle^{-\frac12}\leq|\xi|\leq2\langle t\rangle^{-\frac12}\\
|\cos(bt)|\geq\frac12}}
e^{-|\xi|^2t}|\xi|^{2(\sigma-\sigma_1)-d}\,{\rm d}\xi
\right)^{\frac12}\\
&\gtrsim
\left(
\int_{\langle t\rangle^{-\frac12}}^{2\langle t\rangle^{-\frac12}}
r^{2(\sigma-\sigma_1)-1}\,{\rm d}r
\right)^{\frac12}
\gtrsim
\langle t\rangle^{-\frac12(\sigma-\sigma_1)},\qquad \sigma>\sigma_1 .
\end{aligned}
\end{equation*}
One can repeat a similar argument and apply the equivalent relation $|\widehat{v}(\xi,t)|^{2}\sim|\widehat{\psi}(\xi,t)|^{2}+|\widehat{\Psi}(\xi,t)|^{2}$ to obtain, for some constant $c_{2}>0$, that
\begin{equation*}
|\widehat{v}(\xi,t)|^{2}\geq c_{2}te^{-|\xi|^{2}t}|\xi|^{2-2\sigma_{1}-d}.
\end{equation*}
Hence, it holds for sufficiently large time $t$ that
\begin{equation*}
\|v\|_{\dot{B}^{\sigma}_{2,1}}\gtrsim\Big(\int_{|\xi|\ll 1}te^{-|\xi|^{2}t}|\xi|^{2(\sigma-\sigma_{1}+1)-d}{\rm{d}}\xi\Big)^{\frac{1}{2}}\gtrsim\langle t\rangle^{-\frac{1}{2}(\sigma-\sigma_{1})},\quad\sigma>\sigma_{1},
\end{equation*}
and this, combined with \eqref{Lruvtdes}, gives rise to
\begin{equation*}
\|u\|_{\dot{B}^{\sigma}_{2,1}}\gtrsim\|v\|_{\dot{B}^{\sigma}_{2,1}}-\|u-v\|_{\dot{B}^{\sigma}_{2,1}}\gtrsim\langle t\rangle^{-\frac{1}{2}(\sigma-\sigma_{1})}-\langle t\rangle^{-\frac{1}{2}(\sigma-\sigma_{1}+1)}\gtrsim\langle t\rangle^{-\frac{1}{2}(\sigma-\sigma_{1})},
\end{equation*}
for $t\gg 1$. The proof of Lemma \ref{lemmalinearoptimal} is completed.
\end{proof}

\section*{Appendix B: Notations and technical lemmas}\label{sectionpre}
\addcontentsline{toc}{section}{Appendix B}
\setcounter{section}{2}
\setcounter{equation}{0}

The notations and technical lemmas used throughout this paper are explained as follows. We denote by $c>0$ or $C>0$ the constant independent of time and the parameter $\tau$. $A\lesssim B(A\gtrsim B)$ means $A\leq CB(A\geq CB)$, and $A\sim B$ stands for $A\lesssim B$ and $A\gtrsim B$. For any Banach space $X$ and the functions $g,h\in X$, let $\|(g,h)\|_{X}:=\|g\|_{X}+\|h\|_{X}$. For any $T>0$ and $1\leq \varrho\leq\infty$, we denote by $L^{\varrho}(0,T;X)$ the set of measurable functions $g:[0,T]\rightarrow X$ such that $t\mapsto \|g(t)\|_{X}$ is in $L^{\varrho}(0,T)$ and write $\|\cdot\|_{L_{T}^{\varrho}(X)}:=\|\cdot\|_{L^{\varrho}(0,T;X)}$. Let $\mathcal{F}(f)=\widehat{f}$ and $\mathcal{F}^{-1}(f)=\check{f}$ be the Fourier transform of $f$ and its inverse, respectively, along with $\Lambda^{s}f=(-\Delta)^{\frac{s}{2}}f:= \mathcal{F}^{-1}(|\xi|^{s}\mathcal{F}(f))$ and $s \in \mathbb{R}$.

The Littlewood-Paley decomposition, Besov spaces, and related analysis tools (e.g., see  \cite{bahourietal11,danchinMA16,danchinxu17}) are collected below. Let $\chi(\xi)$ be a smooth, radial, non-increasing function compactly supported in $B(0,\frac{4}{3})$, satisfying $\chi(\xi)=1$ in $B(0,\frac{3}{4})$. Then $\varphi(\xi):=\chi(\frac{\xi}{2})-\chi(\xi)$ satisfies
\begin{equation*}
\sum_{j\in \mathbb{Z}}\varphi(2^{-j}\cdot)=1\quad\text{and}\quad \text{{\rm{Supp}}}~ \varphi\subset \Big\{\xi\in \mathbb{R}^{d}~|~\frac{3}{4}\leq |\xi|\leq \frac{8}{3}\Big\}.
\end{equation*}
For any $j\in \mathbb{Z}$, define the homogeneous dyadic blocks $\dot{\Delta}_{j}$ by
\begin{equation*}
\dot{\Delta}_{j}u:=\mathcal{F}^{-1}\big{(} \varphi(2^{-j}\cdot )\mathcal{F}(u) \big{)}=2^{jd}h(2^{j}\cdot)\star u\quad\text{with}\quad h:=\mathcal{F}^{-1}\varphi.
\end{equation*}
Then the Littlewood-Paley decomposition
\begin{equation*}
u=\sum_{j\in \mathbb{Z}}\dot{\Delta}_{j}u
\end{equation*}
holds in $\mathcal{S}_h'$, i.e.,  tempered distributions $\mathcal{S}'$ such that $\lim\limits_{j\rightarrow-\infty}\|\dot{S}_ju\|_{L^{\infty}}=0$ holds for any $u\in \mathcal{S}_h'$, where $\dot{S}_j$ denotes the low frequency cut-off $\dot{S}_j:=\chi(2^{-j}D)$.

\vspace{2mm}
With the help of those dyadic blocks, we give the definition of homogeneous Besov spaces as follows.
\begin{defn}\label{defnbesov}
For $s\in \mathbb{R}$ and $1\leq p,r\leq \infty$, the  homogeneous Besov space $\dot{B}^{s}_{p,r}$ is defined by
\begin{equation*}
\dot{B}^{s}_{p,r}:=\Big{\{} u\in \mathcal{S}_{h}'~|~\|u\|_{\dot{B}^{s}_{p,r}}:=\|\{2^{js}\|\dot{\Delta}_{j}u\|_{L^{p}}\}_{j\in\mathbb{Z}}\|_{l^{r}}<\infty \Big{\}} .
\end{equation*}
\end{defn}

Next, we state a class of mixed space-time Besov spaces introduced by Chemin-Lerner \cite{cheminlerner95}.
\begin{defn}\label{defntimespace}
For $T>0$, $s\in\mathbb{R}$, and $1\leq \varrho,r, q \leq \infty$, the space $\widetilde{L}^{\varrho}(0,T;\dot{B}^{s}_{p,r})$ is defined as
\begin{equation*}
\widetilde{L}^{\varrho}(0,T;\dot{B}^{s}_{p,r}):= \Big{\{} u\in L^{\varrho}(0,T;\mathcal{S}'_{h})~|~ \|u\|_{\widetilde{L}^{\varrho}_{T}(\dot{B}^{s}_{p,r})}:=\|\{2^{js}\|\dot{\Delta}_{j}u\|_{L^{\varrho}_{T}(L^{p})}\}_{j\in\mathbb{Z}}\|_{l^{r}}<\infty \Big{\}}.
\end{equation*}
By the Minkowski inequality, it holds 
\begin{equation*}
\|u\|_{\widetilde{L}^{\varrho}_{T}(\dot{B}^{s}_{p,r})}\leq(\geq) \|u\|_{L^{\varrho}_{T}(\dot{B}^{s}_{p,r})}\quad\text{if}~r\geq(\leq)\varrho,
\end{equation*}
where $\|\cdot\|_{L^{\varrho}_{T}(\dot{B}^{s}_{p,r})}$ is the usual Lebesgue-Besov norm.
\end{defn}

To restrict Besov norms to the low-frequency part and the high-frequency part, we often use the following notations for any $s\in\mathbb{R}$ and $p\in[1,\infty]$:
\begin{equation}\nonumber
\left\{
\begin{aligned}
&\|u\|_{\dot{B}^{s}_{p,r}}^{\ell}:=\|\{2^{js}\|\dot{\Delta}_{j}u\|_{L^{p}}\}_{j\leq 0}\|_{l^r}, && \|u\|_{\dot{B}^{s}_{p,r}}^{h}:=\|\{2^{js}\|\dot{\Delta}_{j}u\|_{L^{p}}\}_{j\geq-1}\|_{l^r},\\
&\|u\|_{\widetilde{L}^{\varrho}_{T}(\dot{B}^{s}_{p,r})}^{\ell}:=\|\{2^{js}\|\dot{\Delta}_{j}u\|_{L^{\varrho}_{T}(L^{p})}\}_{j\leq 0}\|_{l^r}, && \|u\|_{\widetilde{L}^{\varrho}_{T}(\dot{B}^{s}_{p,r})}^{h}:=\|\{2^{js}\|\dot{\Delta}_{j}u\|_{L_{T}^{\varrho}(L^{p})}\}_{j\geq-1}\|_{l^r}.
\end{aligned}
\right.
\end{equation}
Define
\begin{equation*}
u^{\ell}:=\sum_{j\leq -1}\dot{\Delta}_{j}u,\qquad u^{h}:=u-u^{\ell}=\sum_{j\geq0}\dot{\Delta}_{j}u.
\end{equation*}
It is easy to check for any $s'>0$ that
\begin{equation}\label{lh}
\left\{
\begin{aligned}
&\|u^{\ell}\|_{\dot{B}^{s}_{p,r}}\lesssim \|u\|_{\dot{B}^{s}_{p,r}}^{\ell}\lesssim \|u\|_{\dot{B}^{s-s'}_{p,r}}^{\ell}, && \|u^{h}\|_{\dot{B}^{s}_{p,1}}\lesssim \|u\|_{\dot{B}^{s}_{p,r}}^{h}\lesssim \|u\|_{\dot{B}^{s+s'}_{p,r}}^{h},\\
&\|u^{\ell}\|_{\widetilde{L}^{\varrho}_{T}(\dot{B}^{s}_{p,r})}\lesssim \|u\|_{\widetilde{L}^{\varrho}_{T}(\dot{B}^{s}_{p,r})}^{\ell}\lesssim \|u\|_{\widetilde{L}^{\varrho}_{T}(\dot{B}^{s-s'}_{p,r})}^{\ell}, && \|u^{h}\|_{\widetilde{L}^{\varrho}_{T}(\dot{B}^{s}_{p,r})}\lesssim \|u\|_{\widetilde{L}^{\varrho}_{T}(\dot{B}^{s}_{p,r})}^{h}\lesssim \|u\|_{\widetilde{L}^{\varrho}_{T}(\dot{B}^{s+s'}_{p,r})}^{h},
\end{aligned}
\right.
\end{equation}
Furthermore, one has
\begin{equation}\label{lhl}
\left\{
\begin{aligned}
&\|u\|_{\dot{B}^{s}_{p,r}\cap\dot{B}^{s'}_{p,r}}=\|u\|_{\dot{B}^{s}_{p,r}}^{\ell}+\|u\|_{\dot{B}^{s'}_{p,r}}^{h}, && s<s',\\
&\|u\|_{\dot{B}^{s}_{p,r}+\dot{B}^{s'}_{p,r}}=\|u\|_{\dot{B}^{s}_{p,r}}^{\ell}+\|u\|_{\dot{B}^{s'}_{p,r}}^{h}, && s>s'.
\end{aligned}
\right.
\end{equation}

We recall some basic properties of Besov spaces and product estimates which will be used repeatedly in this paper. Remark that all the properties remain true for the Chemin-Lerner type spaces whose time exponent has to behave according to the H${\rm{\ddot{o}}}$lder inequality for the time variable.

The first lemma is devoted to the Bernstein inequalities, which in particular imply that $\dot{\Delta}_{j}u$ is smooth for every $u$ in any Besov spaces so that one can take direct calculations on linear equations after applying the operator $\dot{\Delta}_{j}$.
\begin{lemma}\label{Bernstein}
Let $0<r<R$, $1\leq p\leq q\leq \infty$, and $k\in \mathbb{N}$. For any $u\in L^p$ and $\lambda>0$, it holds
\begin{equation}\notag
\left\{
\begin{aligned}
&{\rm{Supp}}~ \mathcal{F}(u) \subset \{\xi\in\mathbb{R}^{d}~| ~|\xi|\leq \lambda R\}\Rightarrow \|D^{k}u\|_{L^q}\lesssim\lambda^{k+d(\frac{1}{p}-\frac{1}{q})}\|u\|_{L^p},\\
&{\rm{Supp}}~ \mathcal{F}(u) \subset \{\xi\in\mathbb{R}^{d}~|~ \lambda r\leq |\xi|\leq \lambda R\}\Rightarrow \|D^{k}u\|_{L^{p}}\sim\lambda^{k}\|u\|_{L^{p}}.
\end{aligned}
\right.
\end{equation}
\end{lemma}

Due to the Bernstein inequalities, the Besov spaces have the following properties.
\begin{lemma}\label{lember}
The following properties hold{\rm{:}}
\begin{itemize}
\item{} For $s\in\mathbb{R}$, $1\leq p_{1}\leq p_{2}\leq \infty$, and $1\leq r_{1}\leq r_{2}\leq \infty$, it holds
\begin{equation*}
\dot{B}^{s}_{p_{1},r_{1}}\hookrightarrow \dot{B}^{s-d(\frac{1}{p_{1}}-\frac{1}{p_{2}})}_{p_{2},r_{2}}.
\end{equation*}
\item{} For $1\leq p\leq q<\infty$, we have the following chain of continuous embedding:
\begin{equation*}
\dot{B}^{0}_{p,1}\hookrightarrow L^{p}\hookrightarrow \dot{B}^{0}_{p,\infty}\hookrightarrow \dot{B}^{\sigma}_{q,\infty},\quad \sigma=-d(\frac{1}{p}-\frac{1}{q})<0.
\end{equation*}
\item {} For all $\lambda>0$, $s\in\mathbb{R}$, and $1\leq p,r\leq\infty$, we have the scaling property
\begin{equation*}
\lambda^{s-\frac{d}{p}}\|u\|_{\dot{B}^{s}_{p,r}}\lesssim \|u(\lambda\cdot)\|_{\dot{B}^{s}_{p,r}}\lesssim \lambda^{s-\frac{d}{p}}\|u\|_{\dot{B}^{s}_{p,r}}.
\end{equation*}
\item{} If $p<\infty$, then $\dot{B}^{\frac{d}{p}}_{p,1}$ is continuously embedded in the set of continuous functions decaying to 0 at infinity.
\item{} The following real interpolation property is satisfied for $1\leq p\leq\infty$, $s_{1}<s_{2}$, and $\theta\in(0,1)$:
\begin{equation}\label{inter}
\|u\|_{\dot{B}^{\theta s_{1}+(1-\theta)s_{2}}_{p,1}}\lesssim \frac{1}{\theta(1-\theta)(s_{2}-s_{1})}\|u\|_{\dot{B}^{ s_{1}}_{p,\infty}}^{\theta}\|u\|_{\dot{B}^{s_{2}}_{p,\infty}}^{1-\theta},
\end{equation}
which in particular implies for any $s'>0$ that
\begin{equation*}
H^{s+s'}\hookrightarrow \dot{B}^{s}_{2,1}\hookrightarrow \dot{H}^{s}.
\end{equation*}
\item{} $\Lambda^{\sigma}$ is an isomorphism from $\dot{B}^{s}_{p,r}$ to $\dot{B}^{s-\sigma}_{p,r}$.
\item{} Let $1\leq p_{1},p_{2},r_{1},r_{2}\leq \infty$, $s_{1}\in\mathbb{R}$, and $s_{2}\in\mathbb{R}$ satisfy
\begin{equation*}
s_{2}<\frac{d}{p_{2}}\quad\text{\text{or}}\quad s_{2}=\frac{d}{p_{2}}~\text{and}~r_{2}=1.
\end{equation*}
The space $\dot{B}^{s_{1}}_{p_{1},r_{1}}\cap \dot{B}^{s_{2}}_{p_{2},r_{2}}$ endowed with the norm $\|\cdot \|_{\dot{B}^{s_{1}}_{p_{1},r_{1}}}+\|\cdot\|_{\dot{B}^{s_{2}}_{p_{2},r_{2}}}$ is a Banach space and has the weak compact and Fatou properties{\rm{:}} If $u^{q}$ is a uniformly bounded sequence of $\dot{B}^{s_{1}}_{p_{1},r_{1}}\cap \dot{B}^{s_{2}}_{p_{2},r_{2}}$, then an element $u$ of $\dot{B}^{s_{1}}_{p_{1},r_{1}}\cap \dot{B}^{s_{2}}_{p_{2},r_{2}}$ and a subsequence $u^{q_{k}}$ exist such that $u^{q_{k}}\rightarrow u $ in $\mathcal{S}'$ and
\begin{equation*}
\|u\|_{\dot{B}^{s_{1}}_{p_{1},r_{1}}\cap \dot{B}^{s_{2}}_{p_{2},r_{2}}}\lesssim \liminf_{q_{k}\rightarrow \infty} \|u_{q_{k}}\|_{\dot{B}^{s_{1}}_{p_{1},r_{1}}\cap \dot{B}^{s_{2}}_{p_{2},r_{2}}}.
\end{equation*}
\end{itemize}
\end{lemma}

The Moser-type product estimates in Besov spaces play a fundamental role in the analysis of nonlinear terms:
\begin{lemma}\label{lemmaB5}
The following statements hold:
\begin{itemize}
\item{} Let $s>0$ and $1\leq p,r\leq \infty$. Then $\dot{B}^{s}_{p,r}\cap L^{\infty}$ is a algebra and
\begin{equation}\label{uv1}
\|uv\|_{\dot{B}^{s}_{p,r}}\lesssim \|u\|_{L^{\infty}}\|v\|_{\dot{B}^{s}_{p,r}}+ \|v\|_{L^{\infty}}\|u\|_{\dot{B}^{s}_{p,r}}.
\end{equation}
\item{} Let $s_{1}$, $s_{2}$, and $p$ satisfy $2\leq p\leq \infty$, $s_{1}\leq \frac{d}{p}$, $s_{2}\leq \frac{d}{p}$, and $s_{1}+s_{2}>0$. Then we have
\begin{equation}\label{uv2}
\|uv\|_{\dot{B}^{s_{1}+s_{2}-\frac{d}{p}}_{p,1}}\lesssim \|u\|_{\dot{B}^{s_{1}}_{p,1}}\|v\|_{\dot{B}^{s_{2}}_{p,1}}.
\end{equation}
\item{} Assume that $s_{1}$, $s_{2}$, and $p$ satisfy $2\leq p\leq \infty$, $s_{1}\leq \frac{d}{p}$, $s_{2}<\frac{d}{p}$, and $s_{1}+s_{2}\geq0$. Then it holds 
\begin{equation}\label{uv3}
\|uv\|_{\dot{B}^{s_{1}+s_{2}-\frac{d}{p}}_{p,\infty}}\lesssim \|u\|_{\dot{B}^{s_{1}}_{p,1}}\|v\|_{\dot{B}^{s_{2}}_{p,\infty}}.
\end{equation}
\end{itemize}
\end{lemma}

We state the following result about the continuity for composition functions:
\begin{lemma}\label{composition}
Let $G:I\rightarrow \mathbb{R}$ be a smooth function satisfying $G(0)=0$. For any $1\leq p\leq \infty$, $s>0$, and $g\in\dot{B}^{s}_{2,1}\cap L^{\infty}$, there exists a constant $C_{g}>0$ depending only on $\|g\|_{L^{\infty}}$, $G'$, $s$, and $d$ such that $G(g)\in \dot{B}^{s}_{p,r}\cap L^{\infty}$ and
\begin{equation}
\|G(g)\|_{\dot{B}^{s}_{p,r}}\leq C_{g}\|g\|_{\dot{B}^{s}_{p,r}}.\label{F0}
\end{equation}
In addition, if $g_{1}, g_{2}\in \dot{B}^{s}_{p,1}\cap L^{\infty}$, then it holds
\begin{align}
&\|G(g_{1})-G(g_{2})\|_{\dot{B}^{s}_{p,1}}\leq C_{g_{1},g_{2}}\Big(1+\|(g_{1},g_{2})\|_{\dot{B}^{\frac{d}{p}}_{p,1}}\Big)\|g_{1}-g_{2}\|_{\dot{B}^{s}_{p,1}},\quad ~s\in (-\frac{d}{p},\frac{d}{p}],\label{F2}
\end{align}
where the constant $C_{g_{1},g_{2}}>0$ depends only on $\|(g_{1},g_{2})\|_{L^{\infty}}$, $G'$, $s$, $p$, and $d$.
\end{lemma}

The following lemma is about the continuity for the composition of multi-component functions, which can be found in \cite[pages 387--388]{runstsickel96}.
\begin{lemma}\label{multicompo}
Let $m\in \mathbb{N}$, $s>0$, $1\leq p,r \leq \infty$, and $\mathbb{G}\in C^{\infty}(\mathbb{R}^{m})$ satisfy $\mathbb{G}(0,...,0)=0$. Then, for any $f_{i}\in\dot{B}_{p,r}^{s}\cap L^{\infty}$ $(i=1,...,m)$, there exists a constant $C_{f}>0$ depending on $\sum_{i=1}^{m}\|f_{i}\|_{L^{\infty}}$, $\mathbb{G}$, $s$, $m$, $p$, $r$, and $d$ such that
\begin{equation}\label{F1}
\begin{aligned}
\|\mathbb{G}(f_{1},...,f_{m})\|_{\dot{B}^{s}_{p,r} }\leq C_{f}\sum_{i=1}^{m}\|f_{i}\|_{\dot{B}^{s}_{p,r} }.
\end{aligned}
\end{equation}	
\end{lemma}

We present the following commutator estimates that are useful for controlling the nonlinear terms in high frequencies:
\begin{lemma}\label{commutator}
Let $1\leq p\leq \infty$ and $-\frac{d}{p}< s\leq 1+\frac{d}{p}$. Then it holds
\begin{align*}
&\sum_{j\in\mathbb{Z}}2^{js}\|[u\cdot\nabla,\dot{\Delta}_{j}] \rho\|_{L^{p}}\lesssim\|u\|_{\dot{B}^{\frac{d}{p}+1}_{p,1}}\|\rho\|_{\dot{B}^{s}_{p,1}},\\
&\sum_{j\in\mathbb{Z}}2^{j(s-1)}\|[u\cdot\nabla,\partial_{x_k}\dot{\Delta}_{j}]a\|_{L^{p}}\lesssim\|u\|_{\dot{B}^{\frac{d}{p}+1}_{p,1}}\|a\|_{\dot{B}^{s}_{p,1}},~~k=1,2,...,d,
\end{align*}
with the commutator $[A,B]:=AB-BA$.
\end{lemma}

We consider the estimates for the linear transport equation
\begin{equation}\label{trans}
\left\{
\begin{aligned}
&\partial_{t}\rho+u\cdot\nabla\rho=F,\qquad x\in\mathbb{R}^{d},\quad t>0,\\
&\rho(0,x)=\rho_{0}(x),\qquad x\in\mathbb{R}^{d}.
\end{aligned}
\right.
\end{equation}

\begin{lemma}[\!\!\cite{bahourietal11}]\label{estrans}
Let $T>0$, $-\frac{d}{2}<s\leq \frac{d}{2}+1$, $1\leq r\leq \infty$, $\rho_{0}\in\dot{B}^{s}_{2,1}$, $u\in L^{1}(0,T;\dot{B}^{\frac{d}{2}+1}_{2,1})$, and $F\in L^{1}(0,T;\dot{B}^{s}_{2,1})$. Then there exists a constant $C>0$ independent of $T$ and $\rho_{0}$ such that the solution $\rho$ to \eqref{trans} satisfies for any $t\in[0,T]$ that
\begin{equation*}
\|\rho\|_{\widetilde{L}^{\infty}_{t}(\dot{B}^{s}_{2,r})}\leq \exp\Big(C\|\nabla u\|_{L^{1}_{t}(\dot{B}^{\frac{d}{2}}_{2,1})}\Big)\Big((\|\rho_{0}\|_{\dot{B}^{s}_{2,r}}+\int_{0}^{t}\|F\|_{\dot{B}^{s}_{2,r}}{\rm{d}}s\Big).
\end{equation*}
In addition, if $r$ is finite, then the solution $\rho$ belongs to $\mathcal{C}([0,T];\dot{B}^{s}_{2,r})$.
\end{lemma}

We also state the optimal regularity estimates for the Lam\'e system
\begin{equation}\label{heat}
\left\{
\begin{aligned}
&\partial_{t}u-\mu\Delta u-(\mu+\lambda)\nabla \operatorname{div} u=G,\qquad x\in\mathbb{R}^{d},\quad t>0,\\
&u(0,x)=u_{0}(x),\qquad x\in\mathbb{R}^{d}.
\end{aligned}
\right.
\end{equation}

\begin{lemma}[\!\!\cite{danchinCPDE07}]\label{esheat}
Let $T>0$, $\mu>0$, $2\mu+\lambda>0$, $s\in\mathbb{R}$, $1\leq p,r\leq\infty$, and $1\leq\varrho_{2}\leq\varrho_{1}\leq\infty$. Assume that $u_{0}\in\dot{B}^{s}_{p,r}$ and $G\in\widetilde{L}^{\varrho_{2}}(0,T;\dot{B}^{s-2+\frac{2}{\varrho_{2}}}_{p,r})$. Then the solution $u$ to \eqref{heat} satisfies for any $t\in[0,T]$ that
\begin{equation*}
\min\{\mu,2\mu+\lambda\}^{\frac{1}{\varrho_{1}}}\|u\|_{\widetilde{L}^{\varrho_{1}}_{t}(\dot{B}^{s+\frac{2}{\varrho_{1}}}_{p,r})}\lesssim \|u_{0}\|_{\dot{B}^{s}_{p,r}}+\min\{\mu,2\mu+\lambda\}^{\frac{1}{\varrho_{2}}-1}\|G\|_{\widetilde{L}^{\varrho_{2}}_{t}(\dot{B}^{s-2+\frac{2}{\varrho_{2}}}_{p,r})}.
\end{equation*}
\end{lemma}

Finally, when studying the incompressible limits, we need the following Strichartz estimates.

\begin{lemma}[\!\!\cite{danchin02}]\label{lemmadispersive}
Assume $(\mathbf{b}_0, \mathbf{d}_0)\in \dot{B}^{s}_{2,1}$ and $(\mathbf{F}, \mathbf{G})\in L^1(0,T;\dot{B}^{s}_{2,1})$ with $s\in\mathbb{R}$ and $T>0$. Let $(\mathbf{b},\mathbf{d})$ be a solution on $[0,T]$ to the following system with $\var>0${\rm{:}}
\begin{equation}
\left\{
\begin{aligned}
&\partial_t \mathbf{b}+\frac{1}{\var}\Lambda \mathbf{d}=\mathbf{F},\\
&\partial_t \mathbf{d}-\frac{1}{\var}\Lambda \mathbf{d}=\mathbf{G},\\
&(\mathbf{b}, \mathbf{d})(0,x)=(\mathbf{b}_0, \mathbf{d}_0)(x).
\end{aligned}
\right.
\end{equation}
Then for $p\geq2$, $\frac{2}{r}\leq \min\{1,(d-1)(\frac{1}{2}-\frac{1}{p})\}$, and $(r,p,d)\neq (2,\infty,3)$, the following estimate holds{\rm{:}}
\begin{align*}
\|(\mathbf{b}, \mathbf{d})\|_{\widetilde{L}^r_T(\dot{B}^{s+d(\frac{1}{p}-\frac{1}{2})+\frac{1}{r}}_{p,1})}\lesssim \var^{\frac{1}{r}}\Big(\|(\mathbf{b}_0, \mathbf{d}_0)\|_{\dot{B}^{s}_{2,1}}+\|(\mathbf{F}, \mathbf{G})\|_{L^1_T(\dot{B}^{s}_{2,1})}\Big).
\end{align*}
\end{lemma}

\noindent \textbf{Acknowledgments} 
Li's research was partially supported by the National Natural Science Foundation of China Grant Nos. 12331007, 12326613, and 11931010, the Beijing Scholar Foundation, and the key research project of Academy for Multidisciplinary Studies, Capital Normal University.
Shou's research was supported by the National Natural Science Foundation of China Grant No. 12301275.

\vspace{2mm}

\noindent \textbf{Conflict of interest.} The authors do not have any possible conflicts of interest.

\vspace{2mm}

\noindent \textbf{Data availability statement.}
Data sharing is not applicable to this article, as no data sets were generated or analyzed during the current study.

\bibliographystyle{abbrv} 
\parskip=0pt
\small	
\bibliography{Refer}

\bigbreak\bigbreak
\noindent Hai-Liang Li\\
{\scshape School of Mathematical Sciences and Academy for Multidisciplinary Studies,\\
Capital Normal University, Beijing 100048, China\\
{\itshape Email address}: {\tt hailiang.li.math@gmail.com}}

\bigbreak

\noindent Ling-Yun Shou\\
{\scshape School of Mathematical Sciences, Ministry of Education Key Laboratory of NSLSCS,\\
and Key Laboratory of Jiangsu Provincial Universities of FDMTA,\\
Nanjing Normal University, Nanjing 210023, China\\
{\itshape Email address}: {\tt shoulingyun11@gmail.com}}

\bigbreak
\noindent Yue Zhang\\
{\scshape Institute of Applied Physics and Computational Mathematics,\\
Beijing 100088, China\\
{\itshape Email address}: {\tt yuezhangmath@126.com}}

\end{document}